\documentclass[11pt]{article}
\usepackage{amsmath,amssymb,color,amsthm,amsfonts}
\usepackage{enumerate}
\usepackage{cite}
\usepackage{mathrsfs}

\usepackage[bbgreekl]{mathbbol}
\DeclareSymbolFontAlphabet{\mathbb}{AMSb}
\DeclareSymbolFontAlphabet{\mathbbl}{bbold}
\DeclareMathSymbol{\bbepsilon}{\mathord}{bbold}{"0F}
\usepackage[pdftex]{graphicx}
\usepackage{float}
\usepackage{wrapfig}
\usepackage{layout}
\usepackage{cite}
\usepackage{braket}
\usepackage
[draft=false,setpagesize=false,pdfstartview=FitH,
colorlinks=true,linkcolor=blue,citecolor=magenta,pagebackref=true,bookmarks=false]
{hyperref}
\usepackage{mathtools}
\usepackage{mleftright}
\usepackage{physics}

\usepackage[truedimen, margin=35truemm, bottom=20truemm, footskip=15truemm]{geometry}
\usepackage{layout} 

\usepackage{bm} 

\usepackage{aliascnt}

\usepackage{cleveref}

\newtheorem{thm}{Theorem}[section]
\newtheorem{prop}[thm]{Proposition}
\newtheorem{lemma}[thm]{Lemma}
\newtheorem{definition}[thm]{Definition}

\theoremstyle{definition}

 \newtheorem{remark}[thm]{Remark}
 \newtheorem{example}[thm]{Example}

\newcommand{\R}{\mathbb{R}}

\newcommand{\Z}{\mathbb{Z}}

\newcommand{\DD}{\mathbb{D}}

\newcommand{\Rn}{\mathbb{R}^n}

\newcommand{\N}{\mathbb{N}}
\newcommand{\e}{\varepsilon}
\newcommand{\rd}{\mathrm{d}}
\newcommand{\cA}{\mathcal{A}}

\newcommand{\cE}{\mathcal{E}}

\newcommand{\cD}{\mathcal{D}}
\newcommand{\cK}{\mathcal{K}}

\newcommand{\cO}{G}

\newcommand{\cY}{\mathcal{Y}}

\newcommand{\scrS}{\mathscr{S}}

\newcommand{\la}{\left\langle}
\newcommand{\ra}{\right\rangle}

\newcommand{\cX}{\mathcal{Y}}
\newcommand{\eps}{\varepsilon}
\newcommand{\genus}{\mathcal{G}}
\newcommand{\mm}{\mathfrak{m}}

\def\ov#1{\overline{#1}}
\def\wh#1{\widehat{#1}}
\def\wt#1{\widetilde{#1}}

\newcommand{\rad}{\mathrm{rad}}

\DeclareMathOperator{\diver}{div}
\DeclareMathOperator{\supp}{supp}

\DeclareMathOperator{\dist}{dist}

\DeclareMathOperator{\Lip}{Lip}

\DeclareMathOperator{\Int}{Int}
\DeclareMathOperator{\lip}{Lip}

\newcommand{\disp}{\displaystyle}
\newcommand{\di}{\mathrm{d}}
\newcommand{\loc}{\mathrm{loc}}

 \newcommand{\be}{\begin{equation}}
 \newcommand{\ee}{\end{equation}}

 \newcommand{\bd}{\begin{description}}
 \newcommand{\ed}{\end{description}}

 \newcommand{\bdp}{\begin{displaymath}}
 \newcommand{\edp}{\end{displaymath}}
 \newcommand{\bea}{\begin{eqnarray}}
 \newcommand{\eea}{\end{eqnarray}}
 \newcommand{\brr}{\begin{eqnarray*}}
 \newcommand{\err}{\end{eqnarray*}}

\usepackage[french,english]{babel}

\makeatletter
\renewenvironment{proof}[1][\proofname]{\par
	\pushQED{\qed}%
	\normalfont \topsep6\p@\@plus6\p@\relax
	\trivlist
	\item\relax
	{\bfseries
		#1\@addpunct{.}}\hspace\labelsep\ignorespaces
}{%
	\popQED\endtrivlist\@endpefalse
}
\makeatother

 \usepackage{accessibility}
 \usepackage{mathtools}

\begin{document}
\title{
	Compactness via monotonicity in nonsmooth critical point theory, with application to Born-Infeld type equations}

%
%
%

\author{Jaeyoung Byeon \and Norihisa Ikoma \and Andrea Malchiodi \and Luciano Mari}
\date{\today}
\maketitle

{\scriptsize \begin{center} Department of Mathematical Sciences, KAIST, \\ 
291 Daehak-ro, Yuseong-gu, Daejeon 305-701, Republic of Korea\\
E-mail: byeon@kaist.ac.kr
\end{center}}

{\scriptsize \begin{center} 
Department of Mathematics, Faculty of Science and Technology, Keio University,\\ 
Yagami Campus: 3-14-1 Hiyoshi, Kohoku-ku, Yokohama, Kanagawa 2238522, Japan\\
E-mail: ikoma@math.keio.ac.jp
\end{center}}

{\scriptsize \begin{center}
Scuola Normale Superiore,\\
Piazza dei Cavalieri, 7, 56126 Pisa, Italy\\
E-mail: andrea.malchiodi@sns.it
\end{center}}

{\scriptsize \begin{center} Dipartimento di Matematica, Universit\`a degli Studi di Milano,\\
Via Saldini 50, 20123 Milano, Italy\\
E-mail: luciano.mari@unimi.it
\end{center}}

\begin{abstract}
In this paper, we prove new existence and multiplicity results for critical points of lower semicontinuous functionals 
in Banach spaces, complementing the nonsmooth critical point theory set forth by Szulkin and avoiding the need of the Palais--Smale condition. 
We apply our  abstract results to get entire solutions with finite energy to Born-Infeld type autonomous equations.
More precisely, under almost optimal conditions on the nonlinearity, 
we construct a positive solution and infinitely many solutions 
both in the classes of radially symmetric functions and nonradiallly symmetric ones. 
\end{abstract}

\smallskip 

\noindent
\textbf{MSC2020:} 
35A15; 58E05; 58E35; 35B38; 35J62.

%
%
%
%
%
%

\smallskip 

\noindent
\textbf{Keywords:} Nonsmooth critical point theory; Monotonicity trick; critical points; Born--Infeld equations; 
Existence and nonexistence of solutions.

\maketitle



\tableofcontents

\section {Introduction and main abstract results}
Let $(X,\| \cdot \|)$ be a Banach space and consider functionals of the type 
	\[
		I \ : \ X \to (-\infty, \infty], \qquad I (u) \equiv  \Psi(u) - \Phi(u),
	\]
	where we assume that $\Psi$ and $\Phi$ satisfy the following properties: 
	\begin{enumerate}
		\item [($\Psi_1$)] 
		$\Psi : X \to [0,\infty]$ is lower semicontinuous and convex, and $\Psi(0) = 0$;
		\item [($\Phi_1$)] $\Phi \in C^1(X,\R)$.
		\end{enumerate} 
We denote
	\[
	D(\Psi) \equiv \Set{ u \in X | \Psi(u) < \infty }. 
	\]
This kind of functionals arises, for instance, in some problems involving PDEs with contraints for admissible functions.
A typical example we are interested in is the following: 

	\[
		\Psi(u) \equiv \begin{dcases} 
			\int_{\R^n} 1 - \sqrt{1 - |Du|^2} &  \ \ \text{if} \ \|Du\|_{\infty} \leq 1, \\
			\infty & \ \ \text{otherwise},
		\end{dcases}
		\qquad \Phi(u) \equiv  \int_{\R^n} F(u),
	\]
which comes from the prescribed mean curvature problem for space-like hypersurfaces in
 the Lorentz-Minkowski space 
	\[
	\mathbb{L}^{n+1} \equiv \R \times \R^n \qquad \text{with metric} \qquad - \di x_0^2 + \sum_{j=1}^n \di x_j^2. 
	\]
In Lorentzian geometry, 
space-like hypersurfaces with prescribed mean curvature play a major role highlighted, for instance, 
in Marsden \& Tipler \cite{mt}. 
A space-like hypersurface whose mean curvature is a given function $f: \mathbb{L}^{n+1} \to \R$ is described, 
at least locally, by the graph $x_0 = u(x_1,\dots, x_n)$ of a solution $u : \Omega \to \R$ to 
\begin{equation} \label{BSbd}
	\diver \left(\frac{D u}{\sqrt{1-|D u|^2}}\right) + f(x,u) = 0 \qquad \text { in } \ \ \Omega \subset \R^n.
\end{equation}
Formally, \eqref{BSbd} is the Euler-Lagrange equation of the action 
\[
I(u) \equiv \int_\Omega 1 - \sqrt{1 - |Du|^2} - \int_\Omega F(x,u), \qquad 
F(x,s) \equiv \int_0^s f(x,t) \, \di t, 
\]
which is non-smooth where $|Du|=1$ (in the terminology of \cite{mt}, where the graph of $u$ goes null). 
Therefore, even assuming that critical points of $I$ (in a non-smooth sense) exist, 
they may not correspond 
to solutions to \eqref{BSbd}, a fact that makes the existence problem quite challenging.

Equation \eqref{BSbd} also appears in the framework of the Born-Infeld theory for electromagnetism \cite{borninfeld}, 
according to which the identity describes the interplay between an electrostatic potential $u$ and the charge density it generates, 
required, in this specific example, to be $f(x,u)$. 
If $f$ is independent of $u$, including the case of $f$ a mere Radon measure, 
in recent years a few authors investigated the existence and regularity properties of solutions to \eqref{BSbd}. 
The resulting picture is still fragmentary, and many interesting open problems are yet to be solved, see \cite{bpd,bimm,bia}. 
We stress that the Born-Infeld Lagrangian also appears in other branches of theoretical physics. 
For more details we recommend the survey \cite{yang} and the references therein.

	In a pioneering paper, Bartnik \& Simon \cite{BS82} studied the Dirichlet problem for \eqref{BSbd} in bounded domains 
with boundary data $\varphi \in C(\partial \Omega)$ satisfying a very mild space-like condition, 
and proved a general existence theorem when $f \in C ( \ov{\Omega} \times \R)$. 
A candidate solution is produced by minimizing $I$, in their setting a consequence of the direct method. 
A core part of their work is to show that the minimizer $u$ actually solves \eqref{BSbd} and enjoys nice regularity properties. 
Remarkably, this is achieved with no regularity requirement on $\partial \Omega$. 
Uniqueness of solutions holds if $f(x,s)$ is non-increasing in $s$ (since $I$ becomes strictly convex),
but may fail otherwise. 
Indeed, under suitable conditions on $f$, 
for $\varphi= 0$ 
Bereanu, Jebelean \& Mawhin \cite{BJM2014} obtained a second solution of mountain pass type to \eqref{BSbd}. 
To reach the goal, they exploited a general nonsmooth critical point theory developed by Szulkin \cite{S86}. 
Here the boundedness of $\Omega$ implies the uniform boundedness in $C(\overline{\Omega})$ of any admissible function $u$ with $\Vert Du\Vert_\infty \le 1$ attaining the boundary value, an essential fact to guarantee the validity of the Palais-Smale condition. 
Szulkin's nonsmooth critical point theory is quite versatile, 
and applies to various problems that do not allow for a treatment via the classical theory by Ambrosetti and Rabinowitz \cite{AR}.

	Our starting point for the present work was the search for solutions to \eqref{BSbd} 
in the autonomous case on the entire $\R^n$: 
\begin{equation} \label{nbi_intro} \tag{$\mathcal{BI}_f$}
	\diver \left(\frac{D u}{\sqrt{1-|D u|^2}}\right) + f(u) = 0 \qquad \text { in } \ \ \R^n. 
\end{equation}
We focus on solutions $u$ vanishing at infinity. 
The problem was already considered in the literature, see below for more details. 
In this setting, as we shall see, Szulkin's theory is not sufficient anymore. 
Seeking for existence results tailored to \eqref{nbi_intro} led us to complement Szulkin's theory 
in its general framework.

	Our goal is to prove two existence theorems for minimax critical points of $I$, valid under fairly general conditions for $\Psi$ and $\Phi$ 
that are suited to applications and might be somehow optimal. 
Differently from \cite{S86}, we only require a form of \emph{bounded} Palais-Smale condition. 
To find bounded Palais-Smale sequences (and then, critical points of $I$), 
our approach is based on the monotonicity trick due to Struwe \cite{St88}, Jeanjean \& Toland \cite{JT98, J99}, 
a tool widely used in the literature, see for instance \cite{JR20}. 
We adapt the trick to the family of functionals 
\[
I_\lambda \ : \ X \to \R, \qquad I_\lambda(u) \equiv \lambda\Psi(u) - \Phi(u)
\] 
for $\lambda \in \R^+$. 
However, to make it effective for the applications we are considering, 
the adaptation is far from trivial. See below for more explanations.

Following Szulkin \cite{S86}, we set

\begin{definition}\label{def:cri-PS}
	Let $X$ be a Banach space and $I_\lambda = \lambda\Psi - \Phi$, with $\Psi,\Phi$ satisfying $(\Psi_1),(\Phi_1)$ and $\lambda \in \R^+$.
	\begin{itemize}
		\item An element $u \in X$ is called a \emph{critical point} of $I_\lambda$ if
		\begin{equation} \label{cp} 
			\lambda (\Psi(v)-\Psi(u)) -\Phi^\prime(u)(v-u) \ge 0 \qquad \text{for all }v \in X.
		\end{equation}
		\item A sequence $\{u_j\}_{j=1}^\infty \subset X$ is called a \emph{Palais-Smale sequence for $I_\lambda$ at level $c \in \R$} if
		\begin{itemize}
			\item[{\rm (i)}] $I_\lambda(u_j) \to c$ as $j \to \infty$, 
			\item[{\rm (ii)}] there exists $\{\e_j\}_{j=1}^\infty \subset \R^+$ with $\e_j \to 0$ as $j \to \infty$ such that
			\begin{equation} \label{ps}
				\lambda (\Psi(v)-\Psi(u_j)) -\Phi^\prime(u_j)(v-u_j) \ge -\e_j\Vert v-u_j\Vert \qquad \text{for all } v \in X.
			\end{equation} 
		\end{itemize}
	\end{itemize}
\end{definition}

\begin{remark}
	By choosing $v \in D(\Psi)$, notice that any critical point of $I_\lambda$ belongs to $D(\Psi)$. Also, it is easy to see that if $u \in D(\Psi)$ is a critical point of $I_\lambda$ and $\Psi$ is differentiable at $u$, then $u$ is a classical critical point of $I_\lambda$ that is, $\lambda \Psi^\prime(u)- \Phi^\prime(u) = 0.$ More generally, \eqref{cp} can be rephrased as $\lambda^{-1} \Phi'(u) \in \partial \Psi(u)$, where $\partial \Psi(u)$ denotes the subdifferential of $\Psi$ at $u$. 
\end{remark}

To state our results, we need some further conditions on $\Psi$ and $\Phi$.

\begin{enumerate}
	\item[{\rm ($\Psi_2$)}] 
	$\Psi(u) \to \infty$ as $\| u \| \to \infty$. 
	\item[($\Psi_3$)]
	If $\{u_j\}_{j=1}^\infty \subset D(\Psi)$ and $u \in D(\Psi)$ satisfy $u_j \rightharpoonup u$ weakly in $X$ and $\Psi(u_j) \to \Psi(u)$, then $\Vert u_j -u \Vert \to 0$.

	\item  [($\Phi_2$)] 
	Any bounded sequence $\{u_l\}_l$ in $X$ has a weakly convergent subsequence $\{u_{l_j}\}_j$ to $u$ in $X$ 
	such that
	\[
	\liminf_{j \to \infty}\Phi^\prime(u_{l_j})(v-u_{l_j}) \ge \Phi^\prime(u)(v-u) \ \ \text{ for any }\ \ v \in X.
	\]

	\item [${\rm (IB)}$] For any $0< a \le b < \infty$ and $c \in \R,$
	\[ 
	\cK_{[a,b]}^c \equiv \left\{ u \in X \ | \ \text{for some }  \lambda \in [a,b], \text{$u$ is a critical point of $I_\lambda $ with $I_\lambda(u) \le c$} \right\}
	\]
	is bounded in $X$. 
\end{enumerate}	

\begin{remark}
	Observe that $(\Phi_2)$ is satisfied if $X$ is reflexive and $\Phi' : X \to X^*$ is compact; in this case, the liminf can be replaced by a limit, and equality holds. Notice however that $(\Phi_2)$ is a strictly weaker condition: for instance, $(\Phi_2)$ holds if $X$ is a Hilbert space and 
	$\Phi(u) = -\Vert u\Vert^2.$
	This improvement will be useful in our application to the Born-Infeld equation. 
\end{remark}

\begin{remark}
	As we shall see in Proposition \ref{IBtoIC}, properties $(\Psi_3)$ and $(\Phi_2)$ are assumed so that $I_\lambda$ satisfies a bounded Palais-Smale condition in a strengthened form, namely including a dependence on $\lambda$.
\end{remark}	

\begin{remark}
	We assume (IB) as a replacement for the boundedness of Palais-Smale sequences, a condition which may fail for general $I$. For large classes of interesting functionals arising from nonlinear elliptic problems, (IB) may be proved via Pohozaev type identities.
\end{remark}		
Our main achievements are the following minimax theorems. 
We first suppose that $ I_\lambda$ has a uniform mountain pass geometry for $\lambda$ close to $1$:

\begin{equation}\label{mp}
	\begin{dcases}
		&\text{there exist $\rho_0, \alpha_0, \ov{\e} >0$ and $u_0 \in X$ such that} 
		\\[0.3cm]
		& \inf_{\| u \| = \rho_0} I_{1 - \ov{\e}} (u) \equiv \alpha_0 > 0, \qquad \|u_0\| > \rho_0, \qquad \max \left\{ I_{1 + \ov{\e} } (0), \  I_{1 + \ov{\e}} ( u_0 ) \right\} \leq 0. 
		\tag{MP}
	\end{dcases}
\end{equation}
For $\lambda \in [1 - \ov{\e} , 1 + \ov{\e}],$ we define the mountain pass value 
\begin{equation} \label{mpv}
	\begin{array}{l}
		\disp c_0(\lambda) \equiv \inf_{\gamma \in \Gamma_0} \sup_{t \in [0,1]} I_\lambda (\gamma(t)), \qquad \text{where} \\[0.5cm]
		\disp \Gamma_0 \equiv \left\{ \gamma \in C([0,1] , X) \ | \ \gamma(0)  = 0, \ \gamma(1) = u_0 \right\}.
	\end{array}
\end{equation}
We shall prove later that $\Gamma_0 \neq \emptyset$ and $0 < \alpha_0 \leq c_0(\lambda) < \infty$ for each $\lambda \in [1-\ov{\e} , 1 + \ov{\e}]$, 
see Lemma \ref{lem:Gkck}. 
We also note that, since $\Psi \ge 0$, $c_0(\lambda)$ is nondecreasing in $\lambda$.
\begin{thm}\label{mpt}
	Assume that $(\Psi_1)$, $(\Psi_2)$, $(\Psi_3),$ $(\Phi_1)$, $(\Phi_2)$, \textup{(IB)} and \eqref{mp} hold. 
	Then $I=I_1$ admits a critical point $u \in X$ with $I(u) = c_0(1)$. 
\end{thm}

The second result concerns the multiplicity of critical points of $I$ when $I$ is even, so we assume 

\begin{equation}\label{even}
	\Psi(-u) = \Psi(u), \quad \Phi(-u) = \Phi(u) \quad \text{for any $u \in X$}.\tag{E}
\end{equation}
For $k \in \N$, we let $\DD^k(r)$ be the closed disk in $\R^k$ with center $0$ and radius $r$.  
The disk $\DD^k(1)$ is simply denoted  by $\DD^k.$
We assume that $I$ has the following uniform symmetric mountain pass geometry:
\begin{equation}\label{smp}
	\begin{dcases}
		&\text{there exist $\rho_0,\alpha_0,\ov{\e} >0$ and an odd map $\pi_{0,k} \in C(\partial \DD^k, X)$ for each $k\in \N$ such that } 
		\\[0.3cm]
		& \inf_{\| u \| = \rho_0} I_{1 - \ov{\e}} (u)\equiv \alpha_0 > 0, \qquad \min_{\zeta \in \partial \DD^k} \Vert \pi_{0,k} (\zeta) \Vert > \rho_0. 
		\qquad 
		\sup_{ \zeta \in \partial \DD^k} I_{1 + \overline{\e} } (\pi_{0,k} (\zeta)) < 0. 
		\tag{SMP}
	\end{dcases}
\end{equation}
For $\lambda \in [1 - \ov{\e} , 1 + \ov{\e}]$ and $k =1,2,\dots,$
we define the minimax values
\begin{equation}\label{smpv}
	\begin{array}{l}
		\disp c_{k}(\lambda) \equiv \inf_{\gamma \in \Gamma_k} \sup_{ \zeta \in \DD^k} I_\lambda (\gamma (\zeta)), \qquad \text{where}\\[0.5cm]
		\disp \Gamma_k \equiv \left\{ \gamma \in C(\DD^k, X) \ | \ \text{$\gamma$ is odd}, \ \gamma = \pi_{0,k} \ \text{on} \ \partial \DD^k \right\}.
	\end{array}
\end{equation}
Again, we shall prove that $\Gamma_k \neq \emptyset$ and that $\alpha_0 \leq c_k(\lambda) < \infty$ holds, see Lemma \ref{lem:Gkck}.

\begin{thm}\label{smt}
	Assume that $(\Psi_1)$, $(\Psi_2)$, $(\Psi_3),$ $(\Phi_1)$, $(\Phi_2)$, \textup{(IB)}, \eqref{even}  and  \eqref{smp} hold. 
	Then, $I=I_1$ admits infinitely many critical points $\{u_k\}_{k=1}^\infty \subset X$ with $I(u_k) = c_k(1) \to  \infty$ as $k \to \infty.$   
\end{thm}

The proofs of Theorems \ref{mpt} and \ref{smt} will be given in Section \ref{sec:pf-main}. 
Both results crucially rely on the monotonicity trick in Theorem \ref{amps} below, 
for which we do not need properties $(\Psi_3),(\Phi_2), {\rm (IB)}$. 
Comments on Theorems \ref{mpt} and \ref{smt} and their proof, and comparison with the literature, can be found in Subsection \ref{subsec:cp}.

	In Section \ref{sec:App}, we apply Theorems \ref{mpt} and \ref{smt} to obtain nontrivial solutions to  \eqref{nbi_intro}. Here, it is worth pointing out that the equivalence between critical points of $I$ and (weak) solutions of \eqref{nbi_intro} is not immediate, and requires a regularity result for critical points of $I$ of independent interest. To keep the paper at a reasonable length, 
we only consider Born-Infeld type equations \eqref{nbi_intro}, 
even though our main abstract achievements may be used in other settings: as an example, 
we mention quasilinear elliptic problems including (Euclidean) mean curvature type ones \cite{PS04, JR20, DeGM94}, 
plastoelasticity problems \cite{brezissibony,cellina}, nonlinear obstacle problems \cite{DeGM94}, 
solutions to the Lorentz force equation \cite{abt1,abt2}, see also the references therein.

\subsection{Previous related works and novelties of our approach}
\label{subsec:cp}

Various critical point theories based on nonsmooth analysis (for which we refer to \cite{clarke}) arose in the past 40 years: 
for locally Lipschitz functionals we quote \cite{chang}, 
for lower semicontinuous functionals in Banach spaces we stress the already mentioned \cite{S86}, while for continuous functionals in metric spaces we highlight the theory developed independently by Corvellec, Degiovanni \& Marzocchi \cite{C99, CDeGM93, DeGM94} 
and by Katriel \cite{katriel} (see also \cite{joffe_schw} for the Banach space setting). 
As described in \cite[$\S 4$]{CDeGM93}, functionals $I \equiv \Psi -\Phi$ satisfying $(\Psi_1),(\Phi_1)$ fit within their framework. 
However, to get a critical point the authors need the Palais-Smale condition, so Theorems \ref{mpt} and \ref{smt} do not follow from these results.

Regarding the monotonicity trick for nonsmooth functionals, we mention the work by Squassina \cite{squa_2}. Especially, 
\cite[Theorem 3.1]{squa_2} should be compared to our Theorem \ref{amps} 
and in this respect the discussion in \cite[p.161--163]{squa_2} is informative. 
However, the functionals $I_\lambda$ in \cite{squa_2} are assumed to be continuous on the whole Banach space where they are defined, 
and this makes a difference with our results in obtaining a deformation lemma and a bounded Palais-Smale sequence. 

	Next, we highlight the differences between our paper and \cite{squa_2,S86} from the technical point of view. 
It is known that the Palais-Smale condition 
plays a key role in obtaining existence and multiplicity of critical points; 
see \cite[section 2]{AR} and \cite[section 9]{Ra86}. 
Since the condition is weakened in this paper, the assertions in Theorem \ref{smt} are highly nontrivial. 
In fact, even though \cite{squa_2} considers higher minimax values, the existence of infinitely many critical points is not obtained. 
In this respect, a new idea is necessary to prove that the minimax values diverge. 
We next point out differences from \cite{S86}. 
First, for our purposes we need to refine a deformation lemma in \cite[Lemmas 2.1, 2.2 and Proposition 2.3]{S86}; see Lemma \ref{l:defor-lem}. 
The second point regards Ekeland's variational principle, exploited in \cite{S86} to ensure the existence of Palais--Smale sequences: in our setting, to obtain bounded Palais-Smale sequences we need to consider two different functionals $I_\lambda$ and $I_{\lambda + \e}$, a fact that seems to prevent the use of Ekeland's principle. We thus introduce a new iteration argument for deformations. This method is also useful to prove the divergence of minimax values without the full Palais-Smale condition.

We next describe the issues concerning a variational approach to \eqref{nbi_intro}. 
If we consider \eqref{BSbd} in a bounded domain $\Omega$, suitably choosing the function space $X$ we may prove that the domain $D(\Psi)$ itself is bounded in $X$ 
for any nonlinearity $f \in C(\overline{\Omega} \times \R)$, which leads to the Palais-Smale condition. 
However, for \eqref{nbi_intro}, this is no longer true. 
The difficulty to obtain the boundedness of Palais--Smale sequences already occurs in the study of the scalar field equation
\begin{equation} \label{cfe} 
	\Delta u + f(u) = 0 \qquad \text{ in } \,  \R^n,
\end{equation}
where the nonlinearity $f$ satisfies the so called Berestycki-Lions conditions (which are almost optimal for the solvability of \eqref{cfe}). 
In \cite{BL-1, BL-2, BL-3} 
the existence of least energy solutions and minimax solutions to \eqref{cfe} was proved by constrained minimization and minimax methods. 
The arguments there strongly depend on the difference between the homogeneity in $\lambda$ of 
\[
\int_{\R^n} |D u|^2 \qquad \text{and of} \qquad \int_{\R^n}F(u) = \Phi(u)
\]
with respect to the scaling $u(x) \to u(\lambda x)$, which is essential to normalize a Lagrange multiplier. This is not applicable to problem \eqref{nbi_intro} 
since, under the natural scaling $u(x) \to \lambda^{-1}u(\lambda x)$ keeping the restriction $|D u| \le 1,$ $\Psi$ and $\Phi$ display almost the same homogeneity:
\[
\Psi(u_\lambda) = \lambda^{-n} \Psi(u ) \qquad \text{and} \qquad \Phi(u_\lambda  ) = \lambda^{-n} \Phi(\lambda^{-1} u).
\]
For this reason, in our setting, the use of a constrained problem seems more difficult. 
By the same reason, arguments based on the scaling developed in \cite{HIT10,HiTa19,Ik20} 
are not straightforward to apply to the functional corresponding to \eqref{nbi_intro}.

An approach to recover the existence theorems in \cite{BL-1,BL-2,BL-3} by means of unconstrained minimax methods 
was due to Struwe in \cite{St82}, see also \cite[section 11 of chapter II]{St08}. 
However, to verify a variant of the Palais-Smale condition, 
he still needed to introduce some constraint. 
It seems to us that even this method may hardly apply to  \eqref{nbi_intro}. 
From a different viewpoint, seeking to weaken the regularity of $-u + f(u)$ in \eqref{cfe} we mention \cite{ST80} and \cite{BCKM21}. 
There, solutions were constructed for nonlinearities of bounded variation and locally integrable, respectively.

In \cite{Azz1,Azz2,BCD12,MP23}, commented later on in more detail, 
radial solutions to \eqref{nbi_intro} were obtained by either employing the shooting method or approximations of the operators. 
However, for the above reasons, nonsmooth critical point theories were never applied to \eqref{nbi_intro} in the literature. 
To our knowledge, the present work is the first attempt to do so.

	Finally, we mention the work by Jeanjean \& Lu \cite{JR20} where 
they employed the monotonicity trick to obtain infinitely many solutions to \eqref{cfe}. 
To this end, in \cite[Section 2]{JR20}, the monotonicity trick is proved for higher minimax levels as in Theorem \ref{smt} and in \cite{squa_2}. However, as pointed out in \cite[Remark 2.2]{JR20}, the authors did not obtain the divergence of the minimax levels in an abstract setting, 
the property was later shown in \cite[Lemma 5.7]{JR20} by using the specific structure of \eqref{cfe} and a comparison functional (the same as in \cite{HIT10}). 
To the authors' knowledge, Theorem \ref{smt} seems to be the first abstract result in the literature (even for smooth functionals) to show the divergence of minimax values under general conditions not including the Palais-Smale's one.

\section{Application to the Born-Infeld equation}
\label{sec:App}
We exploit Theorems \ref{mpt} and \ref{smt} to investigate the existence of nontrivial solutions to
\begin{equation} \label{nbi} \tag{$\mathcal{BI}_f$}
\diver \left(\frac{D u}{\sqrt{1-|D u|^2}}\right) + f(u) = 0 \qquad \text { in } \ \ \R^n,
\end{equation} 
in dimension $n \ge 3$. 
In principle, Theorems \ref{mpt} and \ref{smt} could be applied to the case $n=2$, too. However, some technical details in the functional-analytic setting 
need a certain care, and for the sake of simplicity the case $n=2$ will be deferred to a future investigation.
	\begin{definition}\label{def_weaksol} 
	Given a Radon measure $\rho$ on $\R^n$, a function $u \in \lip(\R^n)$ is called a \emph{weak solution} to 
	\[
	\diver \left(\frac{D u}{\sqrt{1-|D u|^2}}\right) + \rho = 0 \qquad \text { in } \ \ \R^n
	\]
	if 
	\begin{equation}\label{eq_condi_energy}
	|D u | \leq 1 \ \ \text{ a.e.}, \qquad \frac{1}{\sqrt{1 - |D u|^2}} \in L^1_{\loc}(\R^n) ,
	\end{equation}
and 
	\begin{equation}\label{soleq}
		\int_{\R^n} \frac{D u\cdot D \eta }{\sqrt{1-|D u|^2}}   = \langle \rho, \eta \rangle  \qquad \text{for any } 
		\eta \in \lip_c(\R^n) \equiv \lip (\Rn) \cap C_c(\Rn),
	\end{equation}
where $\langle \rho, \eta \rangle$ stands for the duality pairing.
	\end{definition}
The definition does not make any assumption at infinity on $u$. However, we will restrict ourselves to solutions (with $\rho = f(u)$) satisfying 
	\begin{equation}\label{eq_vanishing}
	u(x) \to 0 \qquad \text{as } \, |x| \to \infty,
	\end{equation}
and among them solutions lying in the energy space
	\begin{equation}\label{eq:def-D12}
	\mathcal{D}^{1,2}(\R^n) \equiv \overline{C^\infty_c(\R^n)}^{\|\cdot\|}, \qquad \|\phi\| \equiv \left( \int_{\R^n} |D\phi|^2  \right)^{\frac{1}{2}}.
	\end{equation}
Formally, they are critical points of the functional
\begin{equation}\label{def_I}
I(u) \equiv \int_{\R^n} \left( 1-\sqrt{1-|Du|^2} - F(u)\right) \qquad \text{with } \, F(t) = \int_0^t f(s) \di s.
\end{equation}
However, the behavior of the Lagrangian density on $\{|Du| = 1\}$ prevents the differentiability of $I$ and calls for care. 
Nonetheless, in Propositions \ref{prop_criticalpoints_2} and \ref{prop_criticalpoints_1} below  
we prove that  a critical point $u$ of $I$ satisfying \eqref{eq_vanishing} is a weak solution of \eqref{nbi},
and that $u$ belongs to $W^{2,q}_\loc(\R^n)$ for each $q \in [2,\infty)$ (hence, to $C^{1,\alpha}_\loc(\R^n)$ for each $\alpha \in (0,1)$) 
and is \emph{strictly spacelike}, in the sense that $|Du|<1$ on $\R^n$. 
The regularity, whose proof needs some subtle arguments, also holds under much more general growth conditions on $u$ at infinity. 
We expect it will be a handy tool for future works on Born-Infeld equations. 

\vspace{0.2cm}	
	To the best of our knowledge, so far the existence problem for \eqref{nbi} was only considered when restricted to radially symmetric functions. In \cite{Azz2}, Azzollini studied positive solutions in the model case $f(t) = |t|^{p-2}t$, and showed that the problem   
\be  \label{eqnon1} 
\diver \left( \frac{D u}{\sqrt{1-|D u|^2}}\right)  + |u|^{p-2}u = 0 \ \   \textrm {  in  } \ \  \R^n, \qquad  \lim_{|x| \to \infty} u(x) = 0, \ 
\ee
\begin{itemize} 
\item[-]  has a radially symmetric positive solution in $\mathcal{D}^{1,2}(\Rn)$ and infinitely many positive solutions not in $\mathcal{D}^{1,2}(\Rn) $ for $p \in (2^*,\infty)$, where $2^* = \frac{2n}{n-2}$;
\item[-] does not have radially symmetric positive solutions if $p \in (1, 2^*)$.
\end{itemize}

On the other hand, for  $p \in (2^*,\infty)$ infinitely many radially symmetric solutions in $\mathcal{D}^{1,2}(\Rn)$ to \eqref{eqnon1} were constructed by Bonheure, De Coster \& Derlet in \cite{BCD12} using a constrained minimax argument. In the same paper, the authors posed the existence problem for $p = 2^*$ as an interesting question, see \cite[page 262]{BCD12}.

	Other nonlinearities have also been considered. A problem corresponding to the ``positive mass case" of \eqref{eqnon1}, namely,
\be  \label{eqnon2} 
\diver \left( \frac{D u}{\sqrt{1-|D u|^2}}\right)  - u +  |u|^{p-2}u = 0, \qquad \lim_{|x| \to \infty} u(x) = 0 
\ee
was studied in \cite{Azz1} and it was shown that \eqref{eqnon2} admits a radially symmetric positive solution in $H^1(\R^n)$ for each $p \in (1, \infty)$. Indeed, the solution therein was obtained by a shooting argument, for a class of locally Lipschitz functions much larger than $|u|^{p-2}u$. 
A general existence result for one nontrivial radial solution, closely related to our main Theorem \ref{teo_exist_zeromass_intro} below, was recently obtained by Mederski \& Pomponio \cite{MP23}, and will be commented in detail in Remark \ref{rem_compa_mainthm_BI}.

About the existence of solutions, we shall consider both radial and non-radial ones. To construct the latter, we restrict to $n$ and $d \in \N$ satisfying
\begin{equation}\label{nek_intro}
n \ge 4, \qquad 2 \le d \le \frac{n}{2}, \qquad \text{and} \quad n - 2d \neq 1.
\end{equation}
Write a point of $\R^n$ as $x = (x_1,x_2,x_3) \in \R^d \times \R^d \times \R^{n-2d}$ and consider the subspace of functions $u \in \mathcal{D}^{1,2}(\R^n)$ with the following symmetries:
\begin{equation}\label{eq_symmetries}
\begin{array}{ll}
u(Ax_1,Bx_2,Cx_3) = u(x_1,x_2,x_3) & \quad \text{for each } \, A,B \in {\rm O}(d), \ C \in {\rm O}(n-2d), \\[0.3cm]
u(x_2,x_1,x_3) = - u(x_1,x_2,x_3).
\end{array}
\end{equation}
Since Bartsch \& Willem \cite{BW2}, the above subspace has been extensively used (we refer to \cite{M21} and references therein). Notice that the only radial function therein is $u \equiv 0$, whence nontrivial solutions to \eqref{nbi} with the symmetries described in \eqref{eq_symmetries} must be non-radial. Given a mass parameter $\mm \ge 0$ and $p \in [2,2^*]$, we define
\[
\mathcal{D}^{1,2}_{\mm,p}(\R^n) \equiv \overline{C^\infty_c(\R^n)}^{\|\cdot\|_{\mathcal{D}^{1,2}_{\mm,p}}}, \qquad \|\phi\|_{\mathcal{D}^{1,2}_{\mm,p}} \equiv \sqrt{\|D\phi\|_2^2 + \mm \|\phi\|_p^2}.
\]
The exponent $p$ will be related to $\mm$ according to assumption (f1) below. If $\mm = 0$, then $\mathcal{D}^{1,2}_{0,p}(\R^n) = \mathcal{D}^{1,2}(\R^n)$. In this case, we agree to set $p = 2^*$. Likewise, by Sobolev's embedding, for $\mm>0$ and $p = 2^*$ the space $\mathcal{D}^{1,2}_{\mm,p}(\Rn)$ is $\mathcal{D}^{1,2}(\Rn)$ with an equivalent norm.

We are ready to state our main existence theorem:

\begin{thm}\label{teo_exist_zeromass_intro}
Assume that $n \ge 3$ and that $f \in C(\R)$, $\mm \ge 0$ satisfy 
	\begin{enumerate}[{\rm (f1)}]
		\item either 
		\[
		\begin{array}{lll}
		{\rm (f1_0)} & \mm = 0, & \disp -\infty <  \liminf_{s \to 0} \frac{f(s)}{|s|^{2^*-2}s}  \le \limsup_{s \to 0} \frac{f(s)}{|s|^{2^*-2}s} = 0, \qquad \text{or} \\[0.5cm]
		{\rm (f1_{\mm,p})} & \mm > 0, &  \disp -\infty <  \liminf_{s \to 0} \frac{f(s)}{|s|^{p-2}s}  \le \limsup_{s \to 0} \frac{f(s)}{|s|^{p-2}s} \equiv - \mm \qquad \text{for some } \, p \in [2,2^*];
		\end{array}
		\]
		\item  there exists $t_0 > 0$ such that $F(t_0) \equiv \int_0^{t_0}f(t)\di t > 0.$
	\end{enumerate}
Then, \eqref{nbi} admits a positive radial weak solution $u \in \mathcal{D}^{1,2}_{\mm,p}(\R^n)$ with $I(u) \in (0,\infty)$, where $I$ is defined in \eqref{def_I}.
Moreover, if $f$ is odd, then \eqref{nbi} admits 
\begin{itemize}
\item[-] infinitely many distinct radial weak solutions $\{u_k\}_k \subset \mathcal{D}^{1,2}_{\mm,p}(\R^n)$, possibly sign-changing, 
with $I(u_k) \in (0,\infty)$ and $I(u_k) \to \infty$ as $k \to \infty$; 
\item[-] infinitely many distinct nonradial, sign-changing weak solutions $\{u_k\}_k \subset \mathcal{D}^{1,2}_{\mm,p}(\R^n)$  satisfying \eqref{eq_symmetries} with $I(u_k) \in (0,\infty)$ and $I(u_k) \to \infty$ as $k \to \infty$, provided that $n,d$ satisfy \eqref{nek_intro}.
\end{itemize}
Finally, any weak solution $v \in \mathcal{D}^{1,2}_{\mm,p}(\R^n)$ to \eqref{nbi} satisfies $v\in W^{2,q}_\loc(\R^n)$ for each $q \in [2,\infty)$, and $\norm{Dv}_\infty <1$ in $\R^n$. 
\end{thm}

\begin{remark}\label{rem_compa_mainthm_BI}
One main novelty of Theorem \ref{teo_exist_zeromass_intro} is the first construction of non-radial solutions to \eqref{nbi}. Even in the radial case, however, the result is not contained in the previous literature, and indeed, under the only conditions ${\rm (f1)}, {\rm (f2)}$, to the best of our knowledge the existence of a radial solution in $\mathcal{D}^{1,2}(\R^n)$ was still unknown. The most general existence result we are aware of is \cite[Theorem 1.1]{MP23} by Mederski \& Pomponio, where the authors produce a nontrivial radial solution in dimension $n \ge 3$ for odd nonlinearities $f \in C(\R)$ satisfying (f2) and either 
\begin{itemize}
\item[(i)] assumption ${\rm (f1_{\mm,p})}$, for $\mm > 0$ and some $p \in [2,\infty)$, or
\item[(ii)] the following assumption:
\[
- \infty < \liminf_{s \to 0} \frac{f(s)}{|s|^{\gamma-1}} \le  \limsup_{s \to 0} \frac{f(s)}{|s|^{\gamma-1}} = 0 \qquad \text{for some } \, \gamma > 2^* 
\]
and, if $2^* < \gamma \le n$, 
\[
\limsup_{s \to \infty} \frac{f(s)}{|s|^{q^*-1}} = 0 \quad \text{for some } \, q \in \left( \frac{n\gamma}{n+\gamma}, n \right), \quad \text{where } \, q^* = \frac{nq}{n-q}.
\]
\end{itemize}
Notice that (ii) and (i) for $p > 2^*$ are stronger than ${\rm (f1_0)}$, while (i) with $p \in [2,2^*]$ corresponds to ${\rm (f1_{\mm,p})}$. 
\end{remark}

\begin{remark}
By appropriately defining $\Psi, \Phi$ in Theorem \ref{teo_exist_zeromass_intro} so that weak solutions $u$ to \eqref{nbi} correspond to critical points of $I = \Psi - \Phi$, the symmetry group $G$ of $u$ is chosen to ensure $(\Phi_2)$ once $\Phi$ is restricted to the subspace of $G$-invariant functions, and to guarantee that the corresponding restriction of $I_\lambda$ still has a uniform (symmetric) mountain pass geometry. The method should be flexible enough to produce other types of nontrivial solutions. 
\end{remark}

\begin{remark}
In \cite{vansha}, Van Schaftingen described an abstract framework to find $G$-invariant minimax critical points without restricting, a priori, the functional $I$ to a $G$-invariant subspace $X_G$ (see also \cite{squa_2}). One advantage of the method is that the minimax values are taken among maps valued in $X$, not only in $X_G$; hence, the approach in \cite{vansha} may help to establish whether the positive, radially symmetric solution produced in Theorem \ref{teo_exist_zeromass_intro} is a ground state, namely, a minimizer for $I$ among all nonzero solutions. The semilinear case was studied in \cite{jeanjean_tanaka}. 
\end{remark}

The following Pohozaev identity will be crucial to check (IB) for problem \eqref{nbi}.
\begin{prop} \label{poid}
	For $f \in C(\R)$, let $u$ be a weak solution to \eqref{nbi} satisfying 
	\begin{equation}\label{as:finite}
	\int_{\Rn} \qty( \frac{\abs{Du}^2}{\sqrt{1 - \abs{Du}^2}} + F(u) ) < \infty. 
	\end{equation}
	Then, $u \in W^{2,q}_\loc(\R^n)$ for each $q \in [2,\infty)$, $|Du|<1$ on $\R^n$ and 
	\begin{equation}\label{eq_poho_main}
	\int_{\Rn} 
	\frac{|D u|^2}{\sqrt{1 - |D u|^2 }} = n \int_{\Rn} \left( 1 - \sqrt{1 - |D u|^2} - F(u)\right) < \infty.
	\end{equation}
	\end{prop}

 As an application of Proposition \ref{poid}, we rule out the existence of finite energy solutions for $f$ not supercritical. Hence, condition (f1) in Theorem \ref{teo_exist_zeromass_intro} is necessary for the existence of a solution to \eqref{nbi}. 
On the other hand, condition (f2) is needed to show the mini-max structure \eqref{mp} or \eqref{smp} of the energy functional, and we believe that it is necessary as well. In this respect, it is well known that (f2) is a necessary condition for the existence of a finite energy solution to \eqref{cfe}. 

\begin{prop} \label{nonexistence_intro}
	For $n \ge 3$ there exist no weak solutions $u \not\equiv 0$ to 
	\begin{equation}\label{eq_BI_up}
		\diver\left( \frac{D u}{\sqrt{1-|D u|^2}}\right) + |u|^{p-2}u  = 0  \qquad \text{  in  } \,  \R^n, \quad p \in \left[2, 2^*\right]
	\end{equation} 
	satisfying 
	\begin{equation}\label{eq_bounds_poho}
		\int_{\Rn} \left( |Du|^2 + |u|^{p} \right)  < \infty. 
	\end{equation}
	In particular, \eqref{eq_BI_up} for $p = 2^*$ does not admit any weak solution $u \in \mathcal{D}^{1,2}(\R^n)\setminus \{0\}$.
\end{prop}
\begin{remark}
	For $p = 2^*$, we thus answer in the negative the question raised in \cite[page 262]{BCD12} for solutions in $\mathcal{D}^{1,2}(\Rn)$. We expect that there exist no positive weak solutions to \eqref{eq_BI_up}  for $p \in (2,2^*)$ without any restriction, as shown in \cite{GiSp81} for the semilinear case $\Delta u + |u|^{p-2}u = 0$. We also believe that the critical case $p = 2^*$ does not admit any positive weak solutions to \eqref{eq_BI_up}. As a related problem, studying the behavior of mountain pass solutions to \eqref{eq_BI_up} as $p \searrow 2^\ast$ seems also a quite interesting issue.
\end{remark}

\subsection{Critical points of $I$ and weak solutions to \eqref{nbi}}	
	
We first discuss the regularity properties of weak solutions in Definition \ref{def_weaksol}. Given a Radon measure $\rho$ on $\R^n$, define the action $I_\rho$ on a bounded domain $\Omega \Subset \R^n$ as
\[
I_\rho^\Omega(\psi) \equiv \int_{\Omega} \Big( 1- \sqrt{1-|D\psi|^2}\Big)  - \langle \rho, \psi \rangle_\Omega,
\]
with $\langle \, , \, \rangle$ the duality pairing and $\langle \rho, \psi \rangle_\Omega =\langle \rho_{|\Omega}, \psi \rangle .$
For $\phi \in C(\partial \Omega)$, we consider the convex set 
\[
\cX_\phi(\Omega) \equiv \Set{ \psi \in W^{1,\infty}(\Omega) |  \abs{D\psi} \le 1, \ \psi = \phi \ \text{ on } \, \partial \Omega }.
\]
\begin{remark}\label{rem_defboundary}
If $\partial \Omega$ is not regular enough to achieve $\psi = \phi$ in the trace sense, the boundary condition has to be intended as in \cite{BS82}. 
However, by \cite[Proposition 3.5]{bimm} this definition suffices to guarantee that functions $u \in \cX_\phi(\Omega)$ can be extended continuously to $\partial \Omega$ with boundary datum $\phi$.
\end{remark}

We say that 
\begin{itemize}
\item[-] \emph{$u$ minimizes $I_\rho$ on $\Omega$ with boundary value $\phi \in C(\partial \Omega)$} if 
\[
I^\Omega_\rho(u) \le I^\Omega_\rho(\psi) \qquad \text{for each } \, \psi \in \cX_\phi(\Omega);
\] 
\item[-] \emph{$u$ is a local minimizer for $I_\rho$} if it minimizes $I^\Omega_\rho$ on each domain $\Omega \Subset \R^n$ with respect to the boundary value $\phi = u_{|_{\partial \Omega}}$. 
\end{itemize}

The regularity result we need is encoded in the next proposition of independent interest, proved in Appendix \ref{appe_1}. 
 
\begin{prop}\label{prop_criticalpoints_1}
Let $u \in \lip (\Rn)$ satisfy $\abs{Du} \leq 1$ a.e. on $\Rn$, and let $\rho \in L^\infty_\loc(\Rn)$. 
\begin{enumerate}[{\rm (i)}]
	\item 
	Suppose that there exist $c_1>0$ and a function $h : \Rn \to (0,\infty)$ such that
		\begin{equation}\label{eq_growth_infty}
	\begin{array}{l}
	\disp h(x) \to \infty \quad \text{as } \, |x| \to \infty, \qquad \lim_{|x|\to \infty} \frac{h(x)}{|x|} = 0, \\[0.5cm] 
	|u(x) |
	\le c_1 + |x| - h(x) \qquad \text{on } \, \Rn.
	\end{array}
		\end{equation}
		Then, the following are equivalent:
		\begin{itemize}
		\item[{\rm (a)}] $u$ is a weak solution to 
			\begin{equation}\label{eq:BIrho}
				\tag{$\mathcal{BI}_\rho$}
				\diver \left( \frac{Du}{\sqrt{1-|Du|^2}} \right) + \rho = 0 \quad \text{in} \ \Rn
			\end{equation}
			as stated in Definition \ref{def_weaksol}. 
		\item[{\rm (b)}] $u$ is a local minimizer for $I_\rho$.
	\end{itemize}
	Furthermore, if any between {\rm (a)} and {\rm (b)} occurs then $u \in W^{2,q}_\loc(\R^n)$ for each $q \in [2,\infty)$, and $|Du|<1$ on $\Rn$. Moreover, for each $R, c>0$ and $q \ge 2$ there exist $\hat R = \hat R(R,c_1,h)$ and $\delta = \delta(\hat R,c)$, $C_q=C(\hat R,c,q)$ with the following properties: if
	\[
	\|\rho\|_{L^\infty(B_{\hat R}(0))} \le c, 
	\]
then 
	\[
	\sup_{B_R(0)} |Du| \le 1-\delta, \qquad \|u\|_{W^{2,q}(B_R(0))} \le C_q.
	\]	
	\item 
	Assume that $\rho \in L^\infty(\Rn)$ and $u \in L^\infty (\Rn)$ is a weak solution to \eqref{eq:BIrho} with $\|\rho\|_\infty + \|u\|_\infty \le c$. Then, there exists $\delta = \delta(c) \in (0,1)$ such that $\|Du\|_\infty \le 1-\delta.$
\end{enumerate}
\end{prop}

\begin{remark}
A growth requirement like \eqref{eq_growth_infty} is necessary for the equivalence (a) $\Leftrightarrow$ (b) to hold. To see this, consider the example $\rho \equiv 0$ and $u$ an affine function with slope $1$. Then, $u$ satisfies (b) but clearly does not satisfy \eqref{eq_condi_energy} in the definition of weak solution, and \eqref{soleq} cannot be given a reasonable meaning.
\end{remark}

\begin{remark}
Case (ii) in the above proposition improves on \cite[Theorem 1.5]{bpd}.
\end{remark}

We next specify the above proposition to our case of interest, the equivalence of critical points of $I$ and solutions
of \eqref{nbi}.
\begin{prop}\label{prop_criticalpoints_2}
Assume that $n \ge 3$ and $\mm \ge 0$, that $u \in \mathcal{D}^{1,2}_{\mm,p}(\R^n)$, and that the functional 
\begin{equation}\label{eq_funct_f}
\eta \in C^\infty_c(\R^n) \mapsto \langle f(u),\eta \rangle \equiv  \int_{\R^n} f(u)\eta
\end{equation}
extends to a continuous one on $\mathcal{D}^{1,2}_{\mm,p}(\R^n)$. Then, the following are equivalent: 
\begin{itemize}
\item[{\rm (a)}] $u$ is a weak solution to \eqref{nbi};
\item[{\rm (b)}] $u$ is a local minimizer for $I_\rho$ with $\rho = f(u)$;
\item[{\rm (c)}] The following inequality holds for each $\psi \in \mathcal{D}^{1,2}_{\mm,p}(\R^n)$ with $\norm{D\psi}_\infty \le 1$:
\begin{equation}\label{eq_szulkin_BI}
\int_{\R^n} \Big(\sqrt{1-|Du|^2} - \sqrt{1-|D\psi|^2}\Big)  - \int_{\R^n} f(u)(\psi - u) \ge 0.
\end{equation}
\end{itemize}
\end{prop}

\begin{remark}\label{rem_useful_identities}
Hereafter, we will often employ the following useful properties valid for $t \in [0,1]$:
\[
\frac{t^2}{2} \le 1- \sqrt{1-t^2} \le t^2, \qquad 1 - \sqrt{1 - t^2} = \sum_{k=1}^\infty b_k t^{2k} \qquad \text{with } b_k > 0,
\]
and the pointwise inequality
\begin{equation}\label{eq_cauchyschwarz}
\frac{Du \cdot (Du-D\psi)}{\sqrt{1-|Du|^2}} \ge \sqrt{1-|D\psi|^2} - \sqrt{1-|Du|^2}
\end{equation}
on the set of points $\{|Du|<1\} \cap \{|D\psi| \le 1\}$, see \cite[Remark 3.15]{bimm}.
\end{remark}



\begin{proof}[Proof of Proposition \ref{prop_criticalpoints_2}]
\textbf{(a) $\Leftrightarrow$ (b)} 
Since $n \ge 3$, from the Sobolev embedding $\mathcal{D}^{1,2}_{\mm,p}(\R^n) \hookrightarrow L^{2^*}(\R^n)$, and from $u \in \mathcal{D}^{1,2}_{\mm,p}(\R^n)$, $\norm{Du}_\infty \le 1$ we deduce that $u(x) \to 0$ as $|x| \to \infty$. In particular, Proposition \ref{prop_criticalpoints_1} ensures that (a) $\Leftrightarrow$ (b) and that $u \in W^{2,q}_\loc(\R^n)$ with $\|Du\|_\infty< 1$. \\[0.2cm]
\textbf{(a) $\Rightarrow$ (c)} 
Given $\psi \in \mathcal{D}^{1,2}_{\mm,p}(\R^n)$ with $\norm{D\psi}_\infty \le 1$ and a sequence $\eps_j \to 0,$ we define 
		\begin{equation}\label{def_psij}
		\psi_j(x) \equiv \max\{u(x), \psi(x) - \eps_j\} + \min\{u(x), \psi(x) + \eps_j\} - u(x) = \left\{ \begin{array}{ll}
		u(x) & \text{if } \, |\psi(x) - u(x)| < \eps_j, \\[0.2cm]
		\psi(x) + \eps_j & \text{if } \, u(x) \ge \psi(x) + \eps_j, \\[0.2cm]
		\psi (x)-\eps_j & \text{if } \, u(x) \le \psi(x) -\eps_j.
		\end{array}\right.
		\end{equation}
Since both $u$ and $\psi$ vanish as $|x| \to \infty$, by construction $u-\psi_j$ has compact support for each $j$. 
Notice also that $\abs{\psi_j(x)} \leq \abs{u(x)} + \abs{\psi(x) - u(x)}$ holds for any $x \in \Rn$. 
Thus, the dominated convergence theorem shows that 
$\psi_j \to \psi$ strongly in $\mathcal{D}^{1,2}_{\mm,p}(\R^n)$ as $j \to \infty$ (see \cite[Lemma 3.7]{bimm} for more details in the case $\mm =0$. The case $\mm>0$ follows by minor modifications). Plugging $u-\psi_j \in \lip_c(\R^n)$ in the  definition of weak solution to \eqref{nbi} and using \eqref{eq_cauchyschwarz} we get
\[ 
\int_{\R^n} \Big(\sqrt{1-|D\psi_j|^2} - \sqrt{1-|Du|^2}\Big)  \le \int_{\R^n} \frac{Du \cdot (Du - D\psi_j)}{\sqrt{1-|Du|^2}} = \int_{\R^n} f(u)(u-\psi_j) .
\]
Rearranging, we infer \eqref{eq_szulkin_BI} for the test function $\psi_j$. On the other hand, from the definition of $\psi_j$ and the first in Remark \ref{rem_useful_identities} we observe that
\[
\Big|\sqrt{1-|D\psi_j|^2} - \sqrt{1-|Du|^2}\Big| \le 1- \sqrt{1-|D\psi|^2} + 1-  \sqrt{1-|Du|^2} \le |D\psi|^2 + |Du|^2.
\]
Whence, inequality \eqref{eq_szulkin_BI} follows by letting $j \to \infty$ and using the dominated convergence theorem 
together with the assumed continuity of \eqref{eq_funct_f}.	\\[0.2cm]
\textbf{(c) $\Rightarrow$ (b)}
Fix $\Omega \Subset \R^n$ and $\psi \in \cX_u(\Omega)$. Up to replacing $\psi$ with $\psi_j$ as in \eqref{def_psij}, we can assume that $\mathrm{supp}(\psi -u) \Subset \Omega$. Extend $\psi$ to $\R^n$ by setting $\psi = u$ away from $\Omega$. Inequality \eqref{eq_szulkin_BI} implies
\[
\int_{\Omega} \Big(\sqrt{1-|Du|^2} - \sqrt{1-|D\psi|^2}\Big) - \int_{\Omega} f(u)(\psi - u) \ge 0.
\]
Then,  setting $\rho \equiv  f(u)$ and rearranging we get $I^\Omega_\rho(u) \le I^\Omega_\rho(\psi)$, thus $u$ satisfies (b).
\end{proof}

\subsection{Proofs of Propositions \ref{poid} and \ref{nonexistence_intro} }


\begin{proof}[Proof of Propositions \ref{poid}]
By \eqref{as:finite}, $|Du| \in L^2(\Rn)$ and 
the interpolation inequality with $|Du|\le 1$ implies $\|Du\|_q < \infty$ for each $q \ge 2$. 
Hence, by the proof of Morrey's embedding theorem (for instance, see \cite[Theorem 9.12]{brezis}), 
for $q>n$ the $C^{0,1-n/q}$ H\"older seminorm of $u$ on $\R^n$ is globally bounded, thus $|u(x)| = \mathcal{O}\left(|x|^{1-\frac{n}{q}}\right)$ as $|x| \to \infty$. 
Proposition \ref{prop_criticalpoints_1} guarantees the stated regularity of $u$. Set for convenience $r(x) = |x|$. 
We prove that for every $R \in (0,\infty)$ 
	\begin{equation}\label{eq:Poh1}
		\begin{aligned} 
			&\int_{B_R} \left\{ \frac{|D u|^2}{\sqrt{1 - |D u|^2}} -n \left( 1 - \sqrt{1 - |D u|^2} \right) + n F(u) \right\}  
			\\
			= \ & R \int_{\partial B_R} 
			\left\{
			F(u) + \frac{ \left( D u \cdot Dr \right)^2 }{\sqrt{1 - |D u|^2}} - \left( 1 - \sqrt{1 - |D u|^2} \right) 
			\right\} \di \sigma, 
		\end{aligned}
	\end{equation}
where $B_R = B_R(0)$. We consider the piecewise affine function 
	\[
	\tau_\eps(t) = \left\{ \begin{array}{ll}
	1 & \quad \text{if } \, t \leq R - \e/2, \\[0.1cm]
	1 - \e^{-1} (t - R + \e/2) & \quad \text{if } \, R - \e/2 < t < R+\e/2, \\[0.1cm]
	0 & \quad \text{if } \, t \geq R+\e/2.
	\end{array}\right.
	\] 
Because of the regularity of $u$, by density \eqref{soleq} also holds for $W^{1,q}$ test functions with compact support, 
so we can use $\eta_\e \equiv \tau_\e (r) r Dr \cdot D u$ in \eqref{soleq} with $\rho = f(u)$ to obtain
	\[
		\int_{\Rn} \frac{D u \cdot D \eta_\e}{\sqrt{ 1 - |D u|^2 }}  = \int_{\Rn} f(u) \eta_\e .
	\] 
From $\diver (rDr) = n$ we get
	\[
	f(u) \eta_\e = \diver \left( \tau_\e F(u) rDr \right) - r\tau_\e'(r) F(u) - n \tau_\e (r) F(u),
	\]
hence
	\[
		\int_{\Rn} f(u) \eta_\e  
		= - \int_{\Rn} r\tau_\e'(r) F(u)  - n \int_{\Rn} \tau_\e (r) F(u) .
	\]
Notice also that 
	\[
		\begin{aligned}
			\frac{D u \cdot D \eta_\e}{\sqrt{1 - |D u|^2}} 
			&= \frac{r\tau_\e'(r) ( D u \cdot Dr)^2 }{\sqrt{1 - |D u|^2}} 
			+ \frac{ \tau_\e(r) |D u|^2}{\sqrt{1 - |D u|^2}} 
			+ \tau_\e(r) rDr \cdot D \left( 1 - \sqrt{1 - |D u|^2} \right) .
		\end{aligned}
	\]
Therefore, integrating by parts yields 
	\[
		\begin{aligned}
			&\int_{\Rn} \frac{D u \cdot D \eta_\e}{\sqrt{1 - | D u|^2}}  
			\\
			= \ &
			\int_{\Rn} 
			\left\{
			\frac{r\tau_\e'(r) ( D u \cdot Dr)^2 }{\sqrt{1 - |D u|^2}} 
			+ \frac{ \tau_\e(r) |D u|^2}{\sqrt{1 - |D u|^2}} 
			- \left\{ r\tau_\e'(r) + n \tau_\e(r) \right\} \left( 1 - \sqrt{1 - |D u|^2} \right)
			\right\} 
		\end{aligned}
	\]
and we conclude
	\begin{equation}\label{eq:Pho1-e}
		\begin{aligned}
			&\int_{\Rn} \tau_\e(r) 
			\left[  \frac{ |D u|^2}{\sqrt{1- |D u|^2}} - n \left( 1 - \sqrt{1 - |D u|^2} \right) 
			+ n F(u) \right] 
			\\
			= \ & \int_{\Rn} r\tau_\e'(r) 
			\left[ - \frac{ \left( D u \cdot Dr \right)^2 }{\sqrt{1 - |D u|^2}} +  1 - \sqrt{1 - |D u|^2} - F(u)  \right] .
		\end{aligned}
	\end{equation}
The coarea formula guarantees that
	\[
		t \mapsto h(t) \equiv  \int_{\partial B_t} 
		\left[ 
		F(u ) + \frac{ \left( D u  \cdot Dr \right)^2  }{\sqrt{ 1 - |D u|^2 }} 
		- \left( 1 - \sqrt{ 1 - |D u|^2 } \right)
		\right] \di \sigma \in L^1_\loc ((0,\infty)),
	\]
thus, letting $\e \to 0$ in \eqref{eq:Pho1-e} and using the definition of $\tau_\eps$ and the fact $Du \in C(\Rn)$, 
we obtain \eqref{eq:Poh1} for any $R$. 
Moreover, by the triangle inequality and the observation that $1-\sqrt{1-t^2} \le t^2$, see Remark \ref{rem_useful_identities}, we get  
\[
|h(t)| \le \int_{\partial B_t} \left[ F(u) + \frac{|D u|^2}{\sqrt{1-|D u|^2}}
	+ |D u|^2\right] \di \sigma,
\]
hence from \eqref{as:finite} and the coarea formula we infer that $h \in L^1([1,\infty))$. 
Therefore, there exists a sequence $\{R_k\}$ such that $R_k \to \infty$ and $R_k h(R_k) \to 0$ as $k \to \infty$. 
Observing that \eqref{as:finite} and the first in Remark \ref{rem_useful_identities} guarantee the finiteness of the right-hand side of \eqref{eq_poho_main}, letting $R=R_k \to \infty$ in \eqref{eq:Poh1} we conclude the desired identity. 
\end{proof}

With the aid of the Pohozaev identity in Propositions \ref{poid}, we can show Proposition \ref{nonexistence_intro}. 

%

\begin{proof}[Proof of Proposition \ref{nonexistence_intro}]
Let $u$ be a weak solution to \eqref{eq_BI_up} verifying  \eqref{eq_bounds_poho}. From $u \in L^p(\Rn)$ and $\|Du\|_\infty \le 1$, we deduce that $u$ vanishes at infinity, hence $u \in L^\infty(\Rn)$. 
Since $f(u(x)) = \abs{u}^{p-2} (x) u(x) \in L^\infty(\Rn)$, Proposition \ref{prop_criticalpoints_1} (ii) guarantees that $\norm{Du}_\infty < 1$, and 
\eqref{eq_bounds_poho} implies 
	\[
		\int_{\Rn} \frac{\abs{Du}^2}{\sqrt{1 -\abs{Du}^2 }} \leq \frac{1}{ \sqrt{ 1 - \norm{Du}_\infty^2 } } \int_{\Rn} \abs{Du}^2 < \infty. 
	\]
Thus, we apply Proposition \ref{poid} with $f(u) = |u|^{p-2}u$ to deduce that
\begin{equation}\label{eq_primapoho}
\int_{\R^n} \frac{|D u|^2}{\sqrt{1-|D u|^2}} = n \int_{\R^n} \left(1-\sqrt{1-|D u|^2} - \frac{1}{p} |u|^{p}\right) < \infty.
\end{equation}
For $\eps>0$, we plug the compactly supported test function $\eta = \zeta_\e(u)$ with $\zeta_\e(t) \equiv (t-\eps)_+ - (t+\eps)_-$ 
in the weak definition of \eqref{eq_BI_up} to obtain
\[
	\int_{\{|u|>\eps\}} \abs{u}^{p-2} u \zeta_\e(u) = \int_{\{|u|>\eps\}} \frac{|Du|^2}{\sqrt{1-|Du|^2}}.
\]
Letting $\eps \to 0$, using the dominated convergence theorem, the identity $|Du|\equiv 0$ a.e. on $\{u=0\}$, 
and substituting in \eqref{eq_primapoho} we thus get
\brr
0 & = &\int_{\R^n} 
\left\{
\frac{p+n}{p} \frac{|D u|^2}{\sqrt{1-|D u|^2}} - n \left( 1-\sqrt{1-|D u|^2} \right) \right\}. 
\err
Since $p \le 2^*$ and 
\[
\frac{p+n}{p} \frac{t}{\sqrt{1-t}} - n\left( 1-\sqrt{1-t} \right) 
\geq \frac{n}{2} \qty( \frac{t}{\sqrt{1-t}} - 2 \qty( 1 - \sqrt{1-t} )  ) > 0 
\quad \text{for $t \in (0,1)$},
\]
we conclude that $|D u| \equiv 0$, thus $u \equiv 0$, a contradiction. 
\end{proof}

\subsection{Proof of Theorem \ref{teo_exist_zeromass_intro}}

For convenience, having fixed $p_1> \max\{n,p\},$ we define a subspace $\cX_{\mm,p}(\R^n)$  of  $\mathcal{D}^{1,2}_{\mm,p}(\R^n)$ by
			\[
			\cX_{\mm,p}(\R^n) \equiv \overline{C^\infty_c(\R^n)}^{\|\cdot\|_{\cX_{\mm,p}}}, \qquad 
				\| v \|_{\cX_{\mm,p}} \equiv \sqrt{ \mm\norm{v}^2_p + \norm{Dv}_2^2 + \norm{Dv}_{p_1}^2},
			\]
where we agree that $p=2^*$ if $\mm = 0$. Notice that, by Sobolev's embedding, $\cX_{\mm,p}(\R^n) \hookrightarrow \cX_{0,2^*}(\R^n)$.		
%

\begin{lemma}\label{lem_reflexi_embe}
Let $n \ge 3$. Then $\cX_{\mm,p}(\R^n)$ is a reflexive Banach space and satisfies 
	\begin{equation}\label{eq_embe}
	\begin{aligned}
	&{\rm (i)} & & \cX_{\mm,p}(\R^n) \hookrightarrow W^{1,p}(\R^n), \\
	&{\rm (ii)} & & \cX_{\mm,p}(\R^n) \hookrightarrow C_0(\R^n) \equiv \Set{ u \in C(\R^n)  | \lim_{|x| \to \infty} u(x) = 0 }.
	\end{aligned}
	\end{equation}
\end{lemma}

\begin{proof}
The norm $\norm{\cdot}_{\cX_{\mm,p}}$ is uniformly convex by the uniform convexity of $L^q$-spaces and \cite[Exercise 3.29]{brezis}, 
thus $\cX_{\mm,p}(\R^n)$ is reflexive. 
If $\mm = 0$, then (i) and (ii) have been shown in \cite[Proposition 3.3]{bimm} (the norm used there is equivalent to $\|\cdot\|_{\cX_{0,2^*}}$), 
hence (ii) for $\mm > 0$ is a consequence of $\cX_{\mm,p}(\R^n) \hookrightarrow \cX_{0,2^*}(\R^n)$. On the other hand, interpolation gives 
\[
\|Du\|_p \le \norm{Du}_2^\theta \norm{Du}_{p_1}^{1-\theta} \le \norm{u}_{\cX_{\mm,p}}, \qquad \frac{1}{p} = \frac{\theta}{2} + \frac{1-\theta}{p_1},
\]
thus (i) holds for $\mm>0$ as well. 
\end{proof}
In view of the above Lemma, it will be convenient to use $\cX_{\mm,p}(\R^n)$ instead of $\mathcal{D}^{1,2}_{\mm,p}(\R^n)$. Notice however the following properties: for any constant $c > 0$, 
\begin{equation}\label{eq_confronorm}
\begin{aligned}
&{\rm (a)} & &\text{for $u$ with $\norm{Du}_\infty \le c$, } \, u \in \mathcal{D}_{\mm,p}^{1,2}(\R^n) \ \Longleftrightarrow \ u \in \cX_{\mm,p}(\R^n), \ \ \text{ and}
\\
& & & \|u\|_{\mathcal{D}_{\mm,p}^{1,2}} \le \|u\|_{\cX_{\mm,p}} \le \norm{u}_{\cD_{\mm,p}^{1,2}} + c^{ 1 - \frac{2}{p_1} }  \norm{u}^{\frac{2}{p_1}}_{ \cD_{\mm,p}^{1,2} };
\\
&{\rm (b)} & &\text{if $\{u_j\} \subset \cX_{\mm,p}(\R^n), \ u \in \cX_{\mm,p}(\R^n)$ and $\norm{Du_j}_\infty \le c$, $\norm{Du}_\infty \le c$, then} \\
& & & \|u_j-u\|_{\cX_{\mm,p}} \to 0 \ \Longleftrightarrow \  \|u_j-u\|_{\mathcal{D}_{\mm,p}^{1,2}} \to 0.
\end{aligned}
\end{equation}
We define $\Phi,\Psi : \cX_{\mm,p}(\R^n) \to (-\infty,\infty]$ as follows:
\begin{equation}\label{eq_PhiPsi_BI}
\begin{aligned}
\Psi(u) &\equiv \begin{dcases} \int_{\Rn} \qty(1 - \sqrt{1 - |D u|^2} + \frac{\mm |u|^{p}}{2p} ) \,   & \quad  \text{ if } \norm{D u}_\infty \le 1; \\
\infty & \quad \text{ if } \norm{D u}_\infty > 1,
\end{dcases} \\
\Phi(u) &\equiv \int_{\Rn} G(u), \qquad \text{with } \, G(s) \equiv F(s) + \frac{\mm}{2p} \abs{s}^p. 
\end{aligned}
\end{equation}
Because of the first in Remark \ref{rem_useful_identities} and (a) in \eqref{eq_confronorm}, the domain of $\Psi$ is 
\begin{equation}\label{eq_dom_psi}
D(\Psi) \equiv \Set{u \in \cX_{\mm,p}(\R^n) | \Psi(u) < \infty } = \Set{ u \in \cX_{\mm,p}(\R^n) | \norm{Du}_\infty \le 1 }.
\end{equation}

\begin{remark}\label{rem_D12_in_Linfty}
By combining \eqref{eq_embe}, the reflexivity of $\cX_{\mm,p}(\R^n)$ and the fact that functions in $D(\Psi)$ are equi-Lipschitz, the following holds: from every bounded sequence $\{u_k\}_{k=1}^\infty \subset D(\Psi)$ we can extract a subsequence (still labelled the same) such that $u_k \to u$ weakly in $\cX_{\mm,p}(\R^n)$ 
and locally uniformly in $\R^n$.
\end{remark}

Define also
\[
I_\lambda \ \ : \ \ \cX_{\mm,p}(\R^n) \to (-\infty, \infty], \qquad I_\lambda(u) \equiv \lambda \Psi(u) - \Phi(u).
\]	
We first check the conditions for $\Psi$.  

\begin{lemma}\label{lem_psi123}
$\Psi : \cX_{\mm,p}(\R^n) \to [0,\infty]$ satisfies $(\Psi_1)$, $(\Psi_2)$, $(\Psi_3)$. 
\end{lemma}

\begin{proof}
Observe first that if $\{u_j\}_j \subset D(\Psi)$, $u \in D(\Psi)$ and
\begin{equation}\label{eq_lowsem}
\mm \|u\|_p \le \liminf_{j \to \infty} \mm \|u_j\|_p, \qquad \|Du\|_{2k} \le \liminf_{j \to \infty} \|Du_j\|_{2k} \ \  \ 
\text{for every } k \ge 1, 
\end{equation}
then by the second in Remark \ref{rem_useful_identities} the following chain of inequalities holds:
\begin{align}\label{eq_lowsem_2}
	\Psi(u) =   \frac{\mm}{2p}\|u\|_p^p + \sum_{k=1}^\infty b_k\int_{\Rn}|D u|^{2k}   
	& \le \frac{\mm}{2p}\|u\|_p^p + \sum_{k=1}^\infty b_k\liminf_{j\to \infty}\int_{\Rn}|D u_j|^{2k} \\
	& \le  \liminf_{j\to \infty} \left[ \frac{\mm}{2p}\|u_j\|_p^p + \sum_{k=1}^\infty b_k\int_{\Rn}|D u_j|^{2k}\right]  =\liminf_{j\to \infty}\Psi(u_j).
\end{align}
We prove $(\Psi_1)$. Since $t \mapsto 1-\sqrt{1-t^2}$ is convex for $t \in (0,1)$, $\Psi$ is convex on $\cX_{\mm,p}(\Rn)$. To check that $\Psi$ is lower-semicontinuous, let $u_j \to u$ in $\cX_{\mm,p}(\R^n)$. 
The claim is obvious if $\|Du_j\|_\infty > 1$ for all large $j$. On the other hand, if $\|Du_{j_i}\|_\infty \le 1$ for a subsequence $j_i \to \infty$, then $\norm{Du_{j_i}}_{2k} \to \norm{Du}_{2k}$ for all $k \geq 1$. Using also $\mm \norm{u_{j_i}}_p \to \mm \norm{u}_p$, we can apply \eqref{eq_lowsem_2} to deduce that $\Psi(u) \leq \liminf_{i \to \infty} \Psi(u_{j_i})$.

	For property $(\Psi_2)$, if $\norm{Du}_\infty \leq 1$, then $\norm{Du}_q^q \leq \norm{Du}_2^2$ for each $q \geq 2$. 
Hence, the first in Remark \ref{rem_useful_identities} and the definition of $\Psi$ yield 
	\[
		\begin{aligned}
		\norm{u}_{\cX_{\mm,p}}^2 = \mm \norm{u}_p^2 + \norm{Du}_2^2 + \norm{Du}_{p_1}^2 
		&\leq C_{p,\mm} \qty( \Psi(u) )^{\frac{2}{p}} + \norm{Du}_2^2 + \norm{Du}_2^{\frac{4}{p_1}} 
		\\
		&\leq C_{p,\mm} \qty( \Psi(u) )^{ \frac{2}{p} } + b_1^{-1} \Psi(u) + \qty( b_1^{-1} \Psi(u) )^{\frac{2}{p_1}}.
		\end{aligned}
	\]
Thus, $(\Psi_2)$ holds. 

%
%

	To show $(\Psi_3)$, assume that $\{u_j\}_j \subset D(\Psi)$ and $u \in D(\Psi)$ 
satisfy $u_j \rightharpoonup u$ in $\cX_{\mm,p}(\R^n)$ and $\Psi(u_j) \to \Psi(u)$ as $j \to \infty$. The weak convergence and $u_j,u \in D(\Psi)$ yields 
$Du_j \rightharpoonup Du$ in $L^{2k} (\R^n; \Rn)$ as $j \to \infty$ for each fixed $k \geq 1$, hence both of the inequalities in \eqref{eq_lowsem} hold. From the uniform convexity of $L^q$-spaces, it suffices to prove that $\mm\norm{u_j}_{p} \to \mm\norm{u}_p$ and that $\norm{Du_j}_{2 k} \to \norm{Du}_{2 k}$ for each $k \in \N$. By contradiction, if one of these fails, then up to passing to a subsequence one of the inequalities in \eqref{eq_lowsem} is strict. As a consequence, one of the two inequalities in \eqref{eq_lowsem_2} is strict as well (recall: here we use the fact $b_k > 0$ for each $k$), whence we would conclude from \eqref{eq_lowsem_2} that
\[
\Psi(u) < \liminf_{j \to \infty} \Psi(u_j) = \Psi(u),
\]
a contradiction. 
\end{proof}

\begin{lemma}\label{lem_phi1}
If $n \ge 3$ and $f \in C(\R)$ satisfies
\begin{equation}\label{eq_phi_basic}
\limsup_{s \to 0} \frac{|f(s)|}{|s|^{p-1}} < \infty \qquad \text{with } \, \left\{ \begin{array}{ll}
p = 2^* & \text{if } \, \mm = 0, \\[0.2cm]
p \in [2,2^*] & \text{if } \, \mm > 0,
\end{array}\right.
\end{equation}
then $\Phi : \cX_{\mm,p}(\R^n) \to \R$ satisfies $(\Phi_1)$.
\end{lemma}

\begin{proof}
By Lemma \ref{lem_reflexi_embe}, $\cX_{\mm,p}(\Rn) \hookrightarrow C_0(\Rn)$. From \eqref{eq_phi_basic}, 
it is standard to see that $\Phi \in C^1(\cX_{\mm,p}(\Rn) , \R)$. 
\end{proof}

Next, we check that critical points for $I$ are weak solutions to \eqref{nbi}. In the positive mass case, this is not entirely trivial due to the asymmetric role of $\Psi$ and $\Phi$ in the definition of a critical point. We have

\begin{prop}\label{prop_sol_crit}
Let $\Phi,\Psi$ be as in \eqref{eq_PhiPsi_BI}, 
and assume \eqref{eq_phi_basic}. Then, for each $\lambda \in \R^+$ the following are equivalent:
\begin{itemize}
\item[-] $u \in \mathcal{D}^{1,2}_{\mm,p}(\Rn)$ is critical for $I_\lambda \equiv \lambda \Psi - \Phi$; 
\item[-] $u \in \mathcal{D}^{1,2}_{\mm,p}(\Rn)$ is a weak solution to
\begin{equation}\label{eq_weaksol_lambda}
\diver \left( \frac{Du}{\sqrt{1-|Du|^2}} \right) + \frac{\mm(1-\lambda)}{2\lambda}|u|^{p-2}u + \lambda^{-1} f(u) = 0, 
\end{equation}
in particular, $u \in W^{2,q}_\loc(\R^n)$ for $q \in [2,\infty)$, and $\|Du\|_\infty<1$.
\end{itemize}
\end{prop}

To prove the result, we need the following general lemma. 

\begin{lemma}\label{lem_critpoint}
Let $X$ be a normed space, $u \in X$ and $\hat \Psi$, $\hat \Phi$ satisfy $(\Psi_1),(\Phi_1)$, respectively. Then, the following are equivalent:
\begin{itemize}
\item[{\rm (i)}] $u$ is a critical point for $\hat{\Psi} - \hat{\Phi}$;
\item[{\rm (ii)}] for some (or equivalently, every) convex, $C^1$ functional $T : X \to \R$, 
$u$ is a critical point for $\hat\Psi_T - \hat\Phi_T$ with $\hat \Psi_T \equiv \hat \Psi + T$ and $\hat \Phi_T \equiv \hat \Phi + T$. 
\end{itemize}
\end{lemma}

\begin{proof}
(i) $\Rightarrow$ (ii) for each $T$.

	Since $T$ is convex, $T(v) - T(u) - T'(u)(v-u) \ge 0$. 
By combining this inequality with the definition of $u$ being a critical point for $\hat\Psi - \hat\Phi$ 
	we deduce 
\[
\hat\Psi_T(v) - \hat\Psi_T(u) - \hat\Phi_T'(u)(v-u) \ge 0 \qquad \text{for all $v \in X$}, 
\]
thus $u$ is critical for $\hat\Psi_T - \hat\Phi_T$.

(ii) for some $T$ $\Rightarrow$ (i).

Assume that there exists a convex functional $T \in C^1(X,\R)$ for which 
\[
\big(\hat\Psi(w) + T(w)\big) - \big(\hat\Psi(u) + T(u)\big) - \big( \hat\Phi'(u) + T'(u)\big)(w-u) \ge 0 \qquad \text{for all }  w \in X.
\]
Fix $v \in X$ and write $w = (1-t)u + tv$, $t \in (0,1]$. Since $\hat\Psi$ is convex and $T$ is of class $C^1$, 
\[
\hat\Psi(w) - \hat\Psi(u) \le t(\hat\Psi(v) - \hat\Psi(u)), \qquad T(w)-T(u) = t T'(u)(v-u) + o(t),
\]
which yields 
\[
t(\hat\Psi(v)-\hat\Psi(u)) + t T'(u)(v-u) + o(t) - t\big( \hat\Phi'(u) + T'(u)\big)(v-u) \ge 0.
\]
Dividing by $t$ and letting $t \to 0$ we conclude 
\[
\hat\Psi(v)-\hat\Psi(u) - \hat\Phi'(u)(v-u) \ge 0,
\]
as claimed.
\end{proof}

\begin{proof}[Proof of Proposition \ref{prop_sol_crit}]
Applying Lemma \ref{lem_critpoint} with 
\[
T(u) = \frac{\mm \lambda}{2p}\norm{u}_p^p, \quad \hat\Psi(u) = \lambda\int_{\Rn} \left( 1- \sqrt{1-|Du|^2}\right), \quad \hat\Phi(u) = \int_{\Rn} \left( F(u) + \frac{\mm(1-\lambda)}{2p} \abs{u}^p \right)
\]
we deduce that $u$ is critical for $I_\lambda$ if and only if it is critical for $\hat \Psi - \hat \Phi$. In this respect, 
notice that $\Phi,\hat \Phi$ and $T$ are of class $C^1$ by Lemma \ref{lem_phi1}. 
Proposition \ref{prop_criticalpoints_2} then ensures that critical points of $\hat \Psi - \hat \Phi$ coincide with weak solutions to \eqref{eq_weaksol_lambda}, and enjoy the claimed regularity properties. 
\end{proof}

We next prove that under our assumptions $I_\lambda$ satisfies $\rm (IB)$.

\begin{lemma}\label{lem_IB}
Assume that $f \in C(\R)$ satisfies {\rm (f1)}. Then, $I_\lambda: \cX_{\mm,p}(\R^n) \to \R$ satisfies ${\rm (IB)}$.
\end{lemma}	

\begin{proof}
Without loss of generality, we can assume $c>0$ in the definition of $\cK_{[a,b]}^c$. By Proposition \ref{prop_sol_crit}, critical points $u \in \cK_{[a,b]}^c$ for $I_\lambda$ are weak solutions to \eqref{eq_weaksol_lambda}. Moreover, $u \in W^{2,q}_\loc(\R^n)$ for $q \in [2,\infty)$, and $\|Du\|_\infty<1$. The latter, together with $u \in \cX_{\mm,p}(\R^n)$ and (f1), imply
\[
\int_{\Rn} \left( \frac{|Du|^2}{\sqrt{1-|Du|^2}} + \mm|u|^p + |F(u)| \right) < \infty,
\]
thus by combining the Pohozaev identity \eqref{eq_poho_main} and $I_\lambda(u) \le c$ we deduce
	\begin{equation}\label{eq_poho_mm}
	\begin{aligned}
	\int_{\Rn} \frac{|D u|^2}{\sqrt{1 - |D u|^2 }} &= n\int_{\R^n} \left( 1-\sqrt{1-|Du|^2} \right) 
	- \frac{n}{\lambda} \int_{\Rn} \left\{  \frac{\mm (1-\lambda)}{2p} \abs{u}^p + F(u)  \right\} 
	\\
	& = \frac{n}{\lambda} I_\lambda (u) \leq \frac{nc}{\lambda}  \leq \frac{nc}{a}. 
	\end{aligned}
	\end{equation}
Therefore, $\cK_{[a,b]}^c \subset D(\Psi)$ is bounded in $\mathcal{D}^{1,2}(\R^n)$, hence in $\cX_{0,2^*}(\R^n)$ and in $C(\R^n)$ because of \eqref{eq_embe} and \eqref{eq_confronorm}. This settles the case $\mm =0$. Assume $\mm > 0$, and 
fix $\delta$ such that 
	\[
		f(s) s + \frac{\mm}{2} \abs{s}^p \leq - \frac{\mm}{4} \abs{s}^p \quad \text{if $\abs{s} \leq \delta$}.
	\]
This is possible by the definition of $\mm$. Choose a constant $\kappa$ so that $\|u\|_{2^*} + \|u\|_\infty \le \kappa$ for each $u \in \cK_{[a,b]}^c$. Then, there exists a constant $C' = C'(\delta,\kappa)$ such that
	\[
		f(s) s + \frac{\mm}{2} \abs{s}^p \leq - \frac{\mm}{4}\abs{s}^p + C' \abs{s}^{2^\ast} \quad \text{for each $\abs{s} \leq \kappa$}. 
	\]
Using this with $I_\lambda (u) \leq c$ and $\Psi(u) \geq 0$, we get 
	\[
		c \geq I_\lambda (u) \geq - \int_{\Rn} \qty{ F(u) + \frac{\mm}{2p} \abs{u}^p } \geq \frac{\mm}{4p} \norm{u}_p^p - \frac{C'}{2^\ast} \norm{u}_{2^\ast}^{2^\ast}. 
	\]
Therefore, $\cK_{[a,b]}^c$ is bounded in $\mathcal{D}^{1,2}_{\mm,p}(\R^n)$ and, by \eqref{eq_confronorm}, in $\cX_{\mm,p}(\Rn)$.
\end{proof}

%

Property $(\Phi_2)$ does not hold in $\cX_{\mm,p}(\R^n)$ because of the translation invariance of $\Phi$, and forces us to restrict $\Psi,\Phi$ to suitable closed subspaces $X \hookrightarrow \cX_{\mm,p}(\R^n)$ where $(\Phi_2)$ is restored. Notice that any of $(\Psi_1),(\Psi_2),(\Psi_3),(\Phi_1)$ is inherited by $\Psi,\Phi$ when restricted to any closed subspace $X$. Assume that a topological group $G$ acts continuously by isometries on $\cX_{\mm,p}(\R^n)$ with an action $\cdot \ : G \times X \to X$, and define the subspace of $G$-invariant functions
	\begin{equation}\label{def_Xo}
		\cX_{\mm,p}(\R^n)_G \equiv \Set{ u \in \cX_{\mm,p}(\R^n) |  g \cdot u = u \ \text{for all $g \in \cO$} }.
	\end{equation}	

\begin{example}\label{ex_fundamental}
Let $G$ be a closed subgroup of the isometry group $\mathrm{Iso}(\R^n)$ of $\R^n$, and consider the natural action induced on $\cX_{\mm,p}(\R^n)$ by defining $(g \cdot u)(x) \equiv u(g^{-1}(x))$. The action is continuous and, for each $g \in G,$ $u \to  g \cdot u$ is an isometry of $\cX_{\mm,p}(\R^n) $. 
\end{example}

In the present paper we shall focus for simplicity on the next two examples:

\begin{example}\label{ex_rad}
Choosing in Example \ref{ex_fundamental} the orthogonal group ${\rm O}(n) \le \mathrm{Iso}(\R^n)$ acting in the standard way by matrix multiplication on $\Rn$, we obtain the subspace of radially symmetric functions with respect to the origin.
\end{example}

\begin{example}\label{ex_skew} 
For $n$ and $d \in \N$ satisfying
\begin{equation}\label{nek_intro2}
n \ge 4, \qquad 2 \le d \le \frac{n}{2}, \qquad \text{and} \quad n - 2d \neq 1,
\end{equation}
write a point of $\R^n$ as $x = (x_1,x_2,x_3) \in \R^d \times \R^d \times \R^{n-2d}$. Consider the subgroup
\[
H \equiv {\rm O}(d) \times {\rm O}(d)\times {\rm O}(n-2d) \le  {\rm O}(n)
\]
sitting in ${\rm O}(n)$ as block-diagonal matrices, and the action induced by that of Example \ref{ex_rad} on $\cX_{\mm,p}(\R^n)$. Consider also the discrete subgroup $\la\tau \ra$ generated by the involution $\tau \in \mathrm{Iso} (\Rn)$, $\tau (x_1,x_2,x_3) \equiv (x_2,x_1,x_3)$, and its action on $\cX_{\mm,p}(\R^n)$ given by
\[
(\tau \circ u)(x) \equiv - u\big( \tau(x) \big).
\]
Then, $\tau$ induces an action on $\cX_{\mm,p}(\Rn)_H$, and $\big( \cX_{\mm,p}(\R^n)_H \big)_{\la\tau \ra}$ is the subspace of $H$-invariant functions which are odd with respect to $\tau$. To our knowledge, the idea of considering such symmetries to produce non-radial solutions was first introduced in \cite{BW2}.
\end{example}

Following \cite[\S1.5]{Wi96}, we describe a sufficient condition on a group $G \le \mathrm{Iso}(\R^n)$, acting on $\cX_{\mm,p}(\R^n)$ as in Example \ref{ex_fundamental}, so that $\Phi$ restricted to $\cX_{\mm,p}(\R^n)_G$ satisfies $(\Phi_2)$. Let $m(y,r,\cO)$ be the maximal number of disjoint balls of radius $r$ centered at points in the $G$-orbit of $y$, namely:
	\[
	m(y,r,\cO) \equiv \sup \Set{ \ell \in \mathbb{N} | \begin{array}{l}
	\text{there exist $g_1,\dots, g_\ell \in \cO$ such that} \\[0.2cm]
	\text{$B( g_i y , r  ) \cap B(g_j y, r) = \emptyset$ if $i \neq j$}.
	\end{array} }
	\]
The following property is proved in \cite[Th\'eor\`eme III.4]{Li82} (see also \cite[\S 1.5]{Wi96}). 
We include the argument for the sake of completeness.
	\begin{prop}\label{prop:1}
		Let $\{u_k\}_{k=1}^\infty \subset \cX_{\mm,p}(\Rn)_G \cap D(\Psi)$ be bounded in $\cX_{\mm,p}(\Rn)$, and assume that  
		for some $r>0$, 
	\begin{equation}\label{1}
		\lim_{|y| \to \infty} m(y,r,\cO) = \infty.
	\end{equation}
Then, there exists $u \in \cX_{\mm,p}(\Rn)_G$ such that, up to a subsequence, $\norm{u_k - u }_{q} \to 0$ holds for any $q \in (p, \infty)$. 
	\end{prop}

\begin{remark}\label{rem_propriety1}
A direct check shows that both of the groups ${\rm O}(n)$ in Example \ref{ex_rad} and $H$ in Example \ref{ex_skew} satisfy \eqref{1}. 
\end{remark}

To prove \Cref{prop:1}, we need the following lemma:

	\begin{lemma}\label{lem:2}
		Let $r>0$ and $\{u_k\}_{k=1}^\infty \subset D(\Psi)$ be a bounded sequence in $\cX_{\mm,p}(\R^n)$. If 
		\[
			\lim_{k \to \infty} \sup_{y \in \Rn} \int_{B(y,r)} \abs{u_k}^{p} = 0,
		\]
	then $u_k \to 0$ strongly in $L^q(\Rn)$ for any $q \in (p,\infty)$. 
	\end{lemma}

	\begin{proof}
First, by \eqref{eq_embe} we know that $\{u_k\}_k$ is bounded in $L^\infty(\Rn)$. Hence, replacing balls with cubes, our assumption implies 
that for each $q \in [p,\infty)$
	\begin{equation}\label{2}
		\lim_{k \to \infty} \sup_{y \in \Rn} \int_{y+Q} \abs{u_k}^{q} = 0, \quad 
		\text{where} \ Q \equiv [0,1]^n. 
	\end{equation}
Moreover, in view of the $L^\infty$ bound and the interpolation inequality, 
it suffices to show that $\norm{u_k}_q \to 0$ for some $q \in (p,p^\ast)$. 
By Sobolev's embedding $W^{1,p} (Q) \hookrightarrow L^q(Q)$, we compute 
	\[
		\begin{aligned}
			\norm{u_k}^q_{q} 
			= \sum_{y \in \Z^n} \int_{y+Q} \abs{u_k}^q 
			&\leq \qty( \sup_{y \in \Z^n} \int_{y+Q} \abs{u_k}^{q} )^{1-\frac{p}{q}} 
			\sum_{y \in \Z^n} \qty( \int_{y+Q} \abs{u_k}^q )^{\frac{p}{q}} 
			\\ 
			&\leq C_0 \qty( \sup_{y \in \Z^n} \int_{y+Q} \abs{u_k}^{q} )^{1 - \frac{p}{q}  } 
			\sum_{y \in \Z^n} \norm{u_k}_{W^{1,p} (y+Q)}^{p} 
			\\
			&= C_0 \qty( \sup_{y \in \Z^n} \int_{y+Q} \abs{u_k}^{q} )^{1 - \frac{p}{q}  } \norm{u_k}_{W^{1,p} (\Rn)}^{p}. 
		\end{aligned}
	\]
Because of \eqref{eq_embe} (i) and our assumptions, $\norm{u_k}_{W^{1,p} (\Rn)}$ is bounded, so the conclusion follows from \eqref{2}.
	\end{proof}

	\begin{proof}[Proof of \Cref{prop:1}]
Let $\{u_k\}_{k=1}^\infty \subset \cX_{\mm,p}(\R^n)_G \cap D(\Psi)$ be bounded in $\cX_{\mm,p}(\R^n)$. By Remark \ref{rem_D12_in_Linfty}, we can assume that $u_k \rightharpoonup u$ weakly in $\cX_{\mm,p}(\R^n)$ and locally uniformly in $\Rn$. The local uniform convergence guarantees that $u \in \cX_{\mm,p}(\R^n)_G$ and that $\norm{Du}_\infty \le 1$, whence $u \in D(\Psi)$ by \eqref{eq_dom_psi}. Setting $v_k:= u_k - u$, the claim follows from Lemma \ref{lem:2} once we show that 
	\begin{equation}\label{3}
		\lim_{k \to \infty} \sup_{y \in \Rn} \int_{B(y,r)} \abs{v_k}^{p} \to 0.
	\end{equation}
In fact, for $y \in \Rn$, choose $g_1,\dots, g_{m(y,r,\cO)} \in \cO$ so that $\{B(g_i y,r)\}_i$ are pairwise disjoint. By \eqref{eq_embe} (i), 
in our assumptions $\norm{v_k}_{p} \le C$ for some constant $C$. On the other hand,  
	\[
		C \ge \norm{v_k}_{ p}^{p} \geq \int_{ \bigcup_{i=1}^{m(y,r,\cO)} B(g_i y,r)} \abs{v_k}^{p} 
		= m(y,r,\cO) \int_{B(y,r)} \abs{v_k}^{p}.
	\]
Let $\e>0$ be arbitrary. It follows from \eqref{1} that there exists $R_\e>0$ such that 
	\[
		\sup_{k \geq 1} \int_{B(y,r)} \abs{v_k}^{p}  < \e \quad \text{for all $\abs{y} \geq R_\e$}.
	\]
From $v_k \to 0$ locally uniformly in $\R^n$, 
	\[
		\lim_{k \to \infty} \sup_{\abs{y} \leq R_\e} \int_{B(y,r)} \abs{v_k}^{p} = 0,
	\]
which implies 
	\[
		\limsup_{k \to \infty} \sup_{y \in \Rn} \int_{B(y,r)} \abs{v_k}^{p}  \leq \e.
	\]
Thus, \eqref{3} holds by the arbitrariness of $\e$. 
	\end{proof}

\begin{prop}\label{prop:3}
Let $G \le \mathrm{Iso}(\R^n)$ act on $\cX_{\mm,p}(\R^n)$ as in Example \ref{ex_fundamental}, and assume that
	\[
		\quad \lim_{|y| \to \infty} m(y,r,\cO) = \infty \qquad \text{for some } \,  r>0 .
	\]
Then, $\Phi : \cX_{\mm,p} (\Rn)_G \to \R$ satisfies $(\Phi_2)$. 
\end{prop}

	\begin{proof}
Assumption (f1) and Lemma \ref{lem_phi1} guarantee that $\Phi \in C^1( \cX_{\mm,p}(\R^n), \R)$. Let 
\[
\{u_k\}_{k=1}^\infty \subset \cX_{\mm,p}(\R^n)_G \cap D(\Psi)
\]
be a bounded sequence. By Lemma \ref{lem_reflexi_embe} and Remark \ref{rem_D12_in_Linfty}, up to a subsequence, we may assume that $u_k \rightharpoonup u$ weakly in $\cX_{\mm,p}(\R^n)$ and locally uniformly in $\R^n$, and that $\|u_k\|_\infty \le c_0$ for some $c_0>0$. We shall prove that 
\begin{equation}\label{eq_liminf_todo}
\liminf_{k \to \infty} \Phi'(u_k)(v-u_k) \ge \Phi'(u)(v-u).
\end{equation}
Define
\[
g(s) \equiv f(s) + \frac{\mm}{2}|s|^{p-2}s.
\]
For any fixed $\eps>0$, choose $R>0$ so that $\norm{v}_{L^p(\Rn\backslash B_R)} < \eps$. Since $u_k \to u$ uniformly in $B_R$, it holds
\[
\disp [\Phi'(u_k) - \Phi'(u)](v) = \disp \int_{\Rn} [g(u_k)-g(u)]v = o_k(1) + \int_{\Rn \backslash B_R} [g(u_k)-g(u)]v 
\]
as $k \to \infty$. On the other hand, by (f1) there exists a constant $C$ such that $|g(s)| \le C|s|^{p-1}$ for $|s| \le c_0$, hence
\[
\begin{array}{lcl}
\disp \left|\int_{\Rn \backslash B_R} [g(u_k)-g(u)]v\right| & \le & \left( \norm{g(u_k)}_{\frac{p}{p-1}} + \norm{g(u)}_\frac{p}{p-1} \right)  \norm{v}_{L^p(\Rn \backslash B_R)} \\[0.5cm]
& \le & C_2\left( \norm{u_k}^{p-1}_p + \norm{u}^{p-1}_p \right) \norm{v}_{L^p(\Rn \backslash B_R)} \\[0.5cm]
& \le & C_3 \eps \left(\sup_k \norm{u_k}_{\cX_{\mm,p}}\right)^{p-1}.
\end{array}
\]	
Therefore, letting first $k \to \infty$ and then $\eps \to 0$, 
we deduce that $\Phi'(u_k)(v) \to \Phi'(u)(v)$ as $k \to \infty$. 
Write $g(s) s = \wt{g}_+(s) - \wt{g}_-(s)$ where $\wt{g}_\pm (s) \equiv \max \set{ 0, \pm g(s) s }$. 
Then, \eqref{eq_liminf_todo} follows provided that
\[
\liminf_{k \to \infty} \int_{\Rn} \qty( \wt{g}_-(u_k) - \wt{g}_+ (u_k) ) \ge \int_{\Rn} \qty( \wt{g}_- (u) - \wt{g}_+(u) ).
\]
We will prove that
\begin{equation}\label{eq_liminf_todo_2}
\lim_{k \to \infty} \int_{\Rn} \wt{g}_+(u_k) = \int_{\Rn} \wt{g}_+(u), \qquad  \liminf_{k \to \infty} \int_{\Rn} \wt{g}_-(u_k) \ge \int_{\Rn} \wt{g}_-(u).  
\end{equation}
The second claim follows by Fatou's lemma. 
As for the first, fix $q>p$ and $\eps > 0$. By ${\rm (f1_{\mm})}$, there exists $C_{\e,q,c_0} > 0$ such that 
	\[
	\wt{g}_+(u_k) \leq \e \abs{u_k}^{p} + C_{\e,q,c_0} \abs{u_k}^{q}, \qquad \wt{g}_+(u) \leq \e \abs{u}^{p} + C_{\e,q,c_0} \abs{u}^{q}.
	\]
Because of Proposition \ref{prop:1}, we can assume that $\lim_{k \to \infty}\norm{u_k -u}_q = 0$, up to subsequence. Therefore, we can fix $R$ large enough so that
\[
C_{\e,q,c_0} \int_{\Rn \backslash B_R} \left( \abs{u_k}^q + \abs{u}^q \right) < \e \qquad \text{for all } k \in \mathbb{N}.
\]
Since $u_k \to u$ in $L^\infty_\loc (\Rn)$, we compute 
\[
\begin{aligned}
	\abs{\int_{\Rn} \qty( \wt{g}_+(u_k) - \wt{g}_+(u) ) } & \le 
	o_k(1) + \int_{\Rn \backslash B_R} \abs{ \wt{g}_+(u_k) - \wt{g}_+(u) } \\
	& \le o_k(1) + \eps \int_{\Rn \backslash B_R} \left( \abs{u_k}^p + \abs{u}^p \right) + C_{\e,q,c_0} \int_{\Rn \backslash B_R} \left( \abs{u_k}^q + \abs{u}^q \right) \\
	& \le o_k(1) + \eps C_3 \left( \norm{u_k}_{\cX_{\mm,p}}^p + \norm{u}_{\cX_{\mm,p}}^p\right) + \eps \\
	& \le o_k(1) + C_4 \eps.
\end{aligned}
\]
By letting $k \to \infty$ and then $\eps \to 0$ we obtain the first in \eqref{eq_liminf_todo_2}, concluding the proof. 
\end{proof}
	
\begin{remark}
By modifying the above argument, in the case $\mm = 0$ and if
\[
f(s) = o \left( |s|^{2^*-1} \right) \qquad \text{as } \, s \to 0,
\]
one can actually show that
		\[
		\Phi' \ \ : \ \ \cX_0(\R^n)_G \cap D(\Psi) \longrightarrow \big(\cX_0(\Rn)\big)^\ast
		\]
is compact, a stronger property than $(\Phi_2)$.
\end{remark}

We are ready to prove our main existence result. 

\begin{proof}[Proof of Theorem \ref{teo_exist_zeromass_intro}]
First of all, by using Proposition \ref{prop_criticalpoints_2} and Lemma \ref{lem_reflexi_embe} we deduce that any weak solution $u \in \cX_{\mm,p}(\R^n)$ to \eqref{nbi} satisfies $u\in W^{2,q}_\loc(\R^n)$ for each $q \in [2,\infty)$ and $\|Du\|_\infty<1$. Referring to Examples \ref{ex_rad} and \ref{ex_skew}, we set for notational convenience
\begin{equation}\label{def_Xj}
X_{\mm,1} \equiv \cX_{\mm,p}(\R^n)_{{\rm O}(n)}, \qquad X_{\mm,2} \equiv \big( \cX_{\mm,p}(\R^n)_H \big)_{\la\tau\ra}.  
\end{equation}
In view of \eqref{eq_confronorm}, to prove the result we shall construct: 
\begin{itemize}
\item a positive solution $u \in X_{\mm,1}$ with $I(u) \in (0,\infty)$;
\item when $f$ is odd, for each $j \in \{1,2\}$, infinitely many distinct solutions $\{u_k\}_k \subset X_{\mm,j}$ with 
$I(u_k) \in (0,\infty)$ and $I(u_k) \to \infty$ as $k \to \infty$ 
(provided $n,d$ verify \eqref{nek_intro} when $j=2$).
\end{itemize}
%
%
%
Define $\Psi,\Phi$ as in \eqref{eq_PhiPsi_BI}, and let 
\[
I_\lambda \equiv \lambda\Psi - \Phi \ : \ \cX_{\mm,p}(\Rn) \to \R. 
\]
In our stated assumptions, Lemmas \ref{lem_psi123} and \ref{lem_phi1} ensure the validity of $(\Psi_1),(\Psi_2),(\Psi_3),(\Phi_1)$ on $\cX_{\mm,p}(\R^n)$. In particular, from $(\Phi_1)$ and the inclusion $\cX_{\mm,p}(\Rn) \hookrightarrow \mathcal{D}^{1,2}_{\mm,p}(\R^n)$, we can apply Proposition \ref{prop_sol_crit} and (a) in \eqref{eq_confronorm} to deduce that $u \in \cX_{\mm,p}(\R^n)$ is a weak solution to \eqref{nbi} if and only if 
\[
\Psi(v) - \Psi(u) - \Phi'(u)(v-u) \ge 0 \qquad \text{for each $v \in D(\Psi)$}. 
\]
By the definition of $\Psi$, the above inequality also holds when $\norm{Dv}_\infty>1$ and thus weak solutions to \eqref{nbi} correspond to critical points of $I = I_1$ in the sense of Szulkin. 

The groups $G= {\rm O}(n), H$, respectively, in Examples \ref{ex_rad} and \ref{ex_skew}, are compact and act continuously on $\cX_{\mm,p}(\R^n)$, and $\la \tau \ra$ acts continuously by isometries on $\cX_{\mm,p}(\R^n)_G$. Moreover, $\cX_{\mm,p}(\R^n)$ is reflexive, thus so is $\cX_{\mm,p}(\R^n)_G$. Therefore, by the principle of nonsmooth symmetric criticality proved in \cite[Theorem 3.16]{koba_ota} 
(see Proposition \ref{prop:sym-critical} in Appendix \ref{s:PSC} for another proof which does not require reflexivity) 
a function $u \in X_{\mm,j}$ is a critical point of $I_\lambda : \cX_{\mm,p}(\R^n) \to \R$ if and only if it is critical for the restriction of $I_\lambda$ to $X_{\mm,j}$, namely, if and only if
\[
\lambda \big(\Psi(v) - \Psi(u)\big) - \Phi'(u)(v-u) \ge 0 \qquad \text{for any }v \in X_{\mm,j}.
\]
Moreover, by Proposition \ref{prop:3} and Remark \ref{rem_propriety1}, $\Phi$ enjoys $(\Phi_2)$ provided that we restrict its domain to either $X_{\mm,1}$ or to $\cX_{\mm,p}(\R^n)_H$ in Example \ref{ex_skew}. In the second case $(\Phi_2)$ is therefore inherited by the restriction of $\Phi$ to $X_{\mm,2}$. Lemma \ref{lem_IB} guarantees that ${\rm (IB)}$ holds for $I_\lambda : X_{\mm,j} \to \R$. In summary, $(\Psi_1),(\Psi_2),(\Psi_3)$ and $(\Phi_1),(\Phi_2),{\rm (IB)}$ are satisfied. To apply Theorem \ref{mpt} (respectively, Theorem \ref{smt} if $f$ is odd), 
it is therefore sufficient to check the uniform mountain pass condition \eqref{mp} (respectively, \eqref{smp}) for, say, $\lambda \in [1/2,2]$.

	In our assumptions, $I_\lambda(0)=0$ for each $\lambda$. Set $\rho_0 > 0$ and consider $u$ with $\|u\|_{\cX_{\mm,p}} = \rho_0$. 
By \eqref{eq_embe}, there exists a constant $C>0$ such that $\|u\|_\infty \le C\rho_0$. 
Hence, because of (f1), for each $\eps > 0$ there exists $\rho_0$ such that the following inequality holds: 
\[
F(u) + \frac{\mm}{2p} \abs{u}^p \leq \e \abs{u}^p \qquad \text{for any $u \in \cX_{\mm}(\Rn)$ with $\norm{u}_{\cX_{\mm,p}} = \rho_0$}. 
\]
Moreover, by using Remark \ref{rem_useful_identities},
\[
\Psi(u) \ge \frac{\mm}{4p} \norm{u}_p^p + \frac{1}{2}\norm{Du}_2^2.
\]
We therefore have the following chain of inequalities for each $\lambda \ge 1/2$:
	\[
	\begin{aligned}
	I_\lambda(u) & \ge \frac{1}{2}\Psi(u) - \int_{\R^n} \left( F(u) + \frac{\mm}{2p}|u|^p\right) \\
	& \ge \left(\frac{\mm}{4p} - \eps\right) \norm{u}_p^p + \frac{1}{4}\norm{Du}_2^2 \ge \left(\frac{\mm}{4p} - \eps\right) \norm{u}_p^p + c_1 \norm{u}_{2^*}^2 +  \frac{1}{8}\norm{Du}_2^2, 
	\end{aligned}
	\]
where in the last line Sobolev's inequality was used. 
In both cases $\mm = 0$ ($p = 2^\ast$) and $\mm > 0$, 
by choosing $\eps >0$ small enough and $\rho_0$ accordingly 
there exists $C_2>0$ such that 
	\[
		I_\lambda (u) \geq C_2 \norm{u}_{\cD_{\mm,p}^{1,2}}^p 
		\quad \text{for each $u \in \cX_{\mm,p}(\Rn)$ with $\norm{u}_{\cX_{\mm,p}} = \rho_0$}. 
	\]
Hence, by \eqref{eq_confronorm}, there exists $\alpha_0>0$ such that $I_\lambda(u) \ge \alpha_0$ on $\set{u | \norm{u}_{\cX_{\mm,p}} = \rho_0}$ 
for each $\lambda \ge 1/2$.

\begin{itemize}
\item[(i)] Property \eqref{mp}.\\
In the radial case, we choose a radially symmetric function $u_0 \in C_c^\infty(\Rn)$ such that  
	\[
		u_0(x) \equiv \begin{cases}
			t_0 & \text{for  $|x| \leq L$},\\
			0 & \text{for  $L+3t_0 \leq |x|$}
		\end{cases}
	\]
	and $\norm{D u_0}_\infty  < 1/2$, where $t_0$ is the value in (f2). For sufficiently large $L > 0$, a direct computation gives $I_2(u_0) < 0$, so \eqref{mp} holds.\\
\item[(ii)] Property \eqref{smp}.\\
In the radial setting, by \cite[Theorem 10]{BL-2}, for each $k \in \N$ there exist $R_k, M_k$ independent of $\zeta \in \partial \DD^k$ and odd maps 
\[
\pi_k \in C\big(\partial \DD^k, H^1_\rad(\Rn)\big) 
\]
valued in the set $H^1_\rad(\Rn)$ of radial functions in $H^1(\Rn)$ and satisfying 
	\begin{equation}\label{eq_prop_pik_BI}
		\supp \left( \pi_k (\zeta) \right) \subset B(0,R_k), \quad 
		\norm{D\pi_k(\zeta)}_\infty \leq M_k, \quad 
		\int_{\Rn} F ( \pi_k(\zeta) ) \geq 1 \quad \text{for any } \zeta \in \partial \DD^k.  
	\end{equation}
Define $\gamma_{0,k}$ by 
	\begin{equation}\label{def_gamma0k}
		\gamma_{0,k} (\zeta) (x) \equiv \pi_k(\zeta) \left( \frac{x}{L} \right) \quad 
		\quad \text{for $L \gg 1$}. 
	\end{equation}
Notice that for sufficiently large $L$, $\norm{D\gamma_{0,k} (\zeta)}_\infty \leq M_k/ L < 1$. 
Thus, by using \eqref{eq_confronorm} and \eqref{eq_embe}, $\pi_k(\partial \DD^k)$ is bounded in $\cX_{\mm,p}(\Rn)$, hence in $C(\Rn)$. 
From the first in \eqref{eq_prop_pik_BI} and the dominated convergence theorem, 
one readily sees that $\gamma_{0,k} \in C(\partial \DD^k, X_{\mm,1})$, and by  Remark \ref{rem_useful_identities} 
	\[
		\begin{aligned}
			I_2(\gamma_{0,k} (\zeta) ) 
			&\leq 2L^{n} \int_{B(0,R_k)} 
			\Big( 1 - \sqrt{1 - M_k^2/L^2 }\Big) - L^n \int_{\Rn} F(\pi_k (\zeta)) 
			\\
			&\leq L^n \left[2\frac{M_k^2}{L^2}|B(0,R_k)| - 1 \right]  < 0,
		\end{aligned}
	\]
which implies \eqref{smp}. Likewise, in dimension $n \ge 4$ and for each $2 \le d \le n/2$, 
in the argument of \cite[Lemma 4.3]{jeanjean_lu} the authors produced, for each $k \in \N$, an odd map 
$\pi_k \in C ( \partial \DD^k , \big(H^1(\Rn)_H\big)_{ \la \tau \ra } )$ 
such that \eqref{eq_prop_pik_BI} holds for suitable $M_k,R_k$ independent of $\zeta$. 
Defining again $\gamma_{0,k}$ as in \eqref{def_gamma0k}, the same computation as above yields 
$\pi_k \in C( \partial \DD^k , X_{\mm, 2} )$ and $\sup_{\zeta \in \partial \DD^k} I_2(\gamma_{0,k} (\zeta) ) < 0$, as required. 
\end{itemize} 	
By applying Theorem \ref{mpt} (resp. Theorem \ref{smt} if $f$ is odd), 
we get all of our conclusions apart from the claim that a positive solution in $X_{\mm,1}$ can be found. To see this, we consider the problem with continuous nonlinearity 
\[
\hat f(s) = \left\{ \begin{array}{ll}
f(s) & \quad \text{if } \, s \ge 0, \\[0.2cm]
- \mm |s|^{p-2}s & \quad \text{if } \, s < 0.
\end{array}\right.
\] 
Then, $\hat f$ satisfies ${\rm (f1_{\mm,p})}$ and (f2) since so does $f$, and therefore there exists a solution $u \in X_{\mm,1}$ to
\[
\diver \left(\frac{D u}{\sqrt{1-|D u|^2}}\right) + \hat f(u) = 0 \qquad \text { in } \ \ \R^n.
\]
Moreover,  by Proposition \ref{prop_criticalpoints_1} and Lemma \ref{lem_reflexi_embe}, $u \in C^1(\Rn)$ and $\|Du\|_\infty<1$. 
Assume that $\{u<0\} \neq \emptyset$. 
As in the proof of Theorem \ref{teo_exist_zeromass_intro}, we may use $u_-(x) \equiv \max \set{ -u(x) , 0 } \in \cX_{\mm,p} (\Rn)$ to deduce 
\[
\int_{\{u<0\}} \left(\frac{|D u|^2}{\sqrt{1-|D u|^2}}\right) + \mm \int_{\{u<0\}} |u|^{p} = 0,
\]
a contradiction. Therefore, $u \ge 0$ and thus $u$ solves \eqref{nbi}. 
The positivity of $u$ follows from the weak Harnack inequality (see \cite[Theorem 8.18]{GiTr01}) and the fact that $u \geq 0$ satisfies 
	\[
		0 = -\diver (a Du ) - f(u(x)) \leq -\diver (a Du) + \qty( \frac{f(u(x))}{u(x)} )_- u(x), \quad 
		a(x) \equiv \frac{1}{\sqrt{1-\abs{Du(x)}^2}}. 
	\]
%
%
This concludes the proof.
\end{proof}

\section{Deformation lemmas}\label{sec:Deformation}

The aim of this section is to prove general deformation lemmas in order to establish the monotonicity method pioneered in \cite{St88,St88b,J99,JT98} for the family $\{I_\lambda\}_{\lambda}$. The results below relate to \cite[Proposition 2.3]{S86}. Although our conclusions are slightly weaker, differently from \cite{S86} we do not require a priori the validity of the Palais-Smale condition
(in this respect, see also \cite[Theorem 3.1]{alves}). Let us recall that a chief difficulty in proving a deformation lemma for lower semicontinuous functionals 
is that the closure of a relatively compact subset contained in a strip $I_\lambda^{-1}([a,b])$ may not lie in the same strip. 
This problem does not occur for continuous functionals as those treated in \cite[Theorems 2.14 and 2.15]{CDeGM93} 
and \cite[Theorem 2.3]{C99}, 
and makes the construction of an energy decreasing flow which produces a bounded Palais-Smale sequence a challenging task. 
Notice that in Lemmas \ref{l:defor-lem} and \ref{sydefolemma} below, we allow $I_\lambda$ to increase along the flow in some regions, albeit in a controlled way.

We introduce some notation: for each $\lambda > 0$, we denote by $\cK_\lambda$ the set of all critical points of $I_\lambda$: 
		\[
		\cK_\lambda \equiv \left\{ u \in X \ | \ \text{$u$ is a critical point of $I_\lambda$} \right\}.
		\]
Also, for a subset $C \subset X$, the symbol $\overline{C}$ stands for the closure of $C$ in $X$. We shall prove two results, the second one by assuming that $\Psi,\Phi$ are even.

\subsection{The Deformation Lemma}

	\begin{lemma}\label{l:defor-lem}
		Suppose that $(\Psi_1)$ and $(\Phi_1)$ hold. 
		Fix $\lambda \in \R^+$ and let $A \subset X$ be a closed and bounded set. 
		For $c \in \R$ and $\sigma >0$, define 
		\[
			\cA_{c,3\sigma} \equiv A \cap \left\{ c - 3 \sigma \leq I_\lambda \leq c + 3 \sigma \right\}.
		\]
		Assume that there exists $\delta_0 > 0$ such that 
		for each $u \in \cA_{c,3\sigma}$ we may find $v=v(u) \in X \setminus \{u\}$ with  
		\begin{equation}\label{ineq-I'}
			\lambda \left(  \Psi(v) - \Psi(u) \right) - \Phi'(u) \left( v - u \right) < - 5 \delta_0 \| v - u \|.
		\end{equation}
		Then	for every set $K \subset \cA_{c,3\sigma}$ which is relatively compact in $X$ 
		there exist small neighborhoods $W$ and $\wt{W}$ of $\overline{K}$ with 
		\[
			\overline{K} \subset W \subset \overline{W} \subset \wt{W}, 
		\]
		$s_0 > 0$ and $\eta \in C([0, \infty) \times X, X)$ such that 
		\begin{enumerate}[{\rm (i)}]
			\item
			$\| \eta(s,w) - w \| \leq s$ for $(s,w) \in [0,\infty) \times X \quad$ and 
			$\quad \eta(s,w) = w$ for $s \in [0,\infty)$ if $w \not \in \wt{W}$; 
			\item
			$I_\lambda ( \eta(s,w) ) \leq I_\lambda (w) + \delta_0 s$ for all $(s,w) \in [0,s_0] \times X$; 
			\item
			$I_\lambda (\eta (s,w)) \leq I_\lambda (w) - 2 \delta_0 s$ for all $(s,w) \in [0,s_0] \times W$ with $I_\lambda (w) \geq c - \sigma$. 
			
		\end{enumerate}
	\end{lemma}

	\begin{remark}\label{rem:0.1}
		Since $I_\lambda$ is only assumed to be lower semicontinuous, the set $\cA_{c,3\sigma}$ may not be closed and 
		its closure may possibly contain points in $\{ I_\lambda < c - 3\sigma \}$. 
		Consequently, \eqref{ineq-I'} may fail at points of $\overline{K} \setminus K$, 
		a major obstacle  to construct an energy decreasing deformation flow. 
		This marks one of differences with the deformation in \cite[Proposition 2.3]{S86}, where 
		it is assumed that the set $K$ in Lemma \ref{l:defor-lem} is compact.  
	\end{remark}
	
	\begin{remark}\label{rem:Ekeland}
		Observe that Lemma \ref{l:defor-lem} holds for a fixed $\lambda \in \R^+$, 
		and it seems difficult to control the deformation for $I_\lambda$ locally uniformly in $\lambda$. 
		As a consequence, in the proof of Theorems \ref{mpt} and \ref{smt} 
		Ekeland's variational principle is hardly applicable, and 
		 we have to devise a different strategy based on an iteration method. 
	\end{remark}

	\begin{proof}[Proof of Lemma \ref{l:defor-lem}]
The following argument is adapted from the proof of \cite[Lemma 2.2 and Proposition 2.3]{S86}. 
Let $K \subset \cA_{c,3\sigma}$ be relatively compact. Without loss of generality, we may suppose 
	\begin{equation}\label{e:del-sig}
		\delta_0 < \sigma < 1. 
	\end{equation}
From the lower semicontinuity of $I_\lambda$ it follows that 
	\begin{equation}\label{eq:uppest-u}
		I_\lambda (u) \leq c + 3\sigma \quad \text{for each $u \in \ov{\cA_{c,3\sigma}}$}.
	\end{equation}
In particular, since $\overline{K} \subset \ov{\cA_{c,3\sigma}}$, the estimate in \eqref{eq:uppest-u} holds on $\ov{K}$.

We divide our arguments into several steps. 
In Step 1, we identify the building blocks to construct a pseudogradient vector field in a neighborhood of $K$, namely, for each $u \in \overline{K}$ we find $v = v(u) \in X$ and a small radius $r = r(u)$ such that, loosely speaking, for $w \in B(u,r)$ the vector $v-w$ acts as a pseudogradient vector: the value of $I_\lambda$ does not increase too much with respect to $I_\lambda(w)$ in direction $v-w$, and decreases at least at a fixed rate whenever $I_\lambda(w) \ge c-2\sigma$. Notice that the choice $v(u)$ in \eqref{ineq-I'} is admissible (and works well) only if $u \in \cA_{c,3\sigma}$. However, as $\cA_{c,3\sigma}$ is not closed, this is not always the case and the analysis for points $u \in \overline{K} \backslash \cA_{c,3\sigma}$ is more involved.

	\medskip
%

	\noindent
	\textbf{Step 1:} 
	\textsl{The following hold:
		\begin{enumerate}[{\rm (a)}]
			\item
			Suppose $u \in \ov{K} \cap \cA_{c,3\sigma}$ and 
			let $v=v(u) \in X \setminus \{u\}$ be as in \eqref{ineq-I'}. 
			Then there exists $r=r(u) > 0$ such that $ B(u,r) \subset \overline{\cA_{c,3\sigma} } + B(0,1) $, 
			$v \not \in \overline{B(u,r)}$ and
			\[
				\lambda \left( \Psi(v) - \Psi(w) \right) - \Phi'(w) (v-w) \leq - 3 \delta_0 \| v - w \| 
				\quad \text{for all $w \in B(u,r)$}. 
			\]
			\item
			Suppose $u \in \left( \ov{K} \setminus \cA_{c,3\sigma} \right)  \cap \cK_\lambda$ and set $v \equiv u$. 
			Then there exists $r=r(u) > 0$ such that $B(u,r) \subset \overline{\cA_{c,3\sigma}} + B(0,1) $ and 
				\[
					\lambda \left( \Psi(v) - \Psi(w) \right) - \Phi'(w) (v-w) \leq \frac{\delta_0}{2} \| v - w\| \quad \text{for any $w \in B(u,r)$}. 
				\]
			In addition, if $w \in B(u,r)$ satisfies $I_\lambda(w) \geq c - 2 \sigma$, then 
				\[
					\lambda \left( \Psi(v) - \Psi(w) \right) - \Phi'(w) (v-w) \leq -3 \delta_0 \| v - w \|. 
				\]
			\item
			For each $u \in \ov{K} \setminus \left( \cA_{c,3\sigma} \cup \cK_\lambda\right) $, 
			there exist $v=v(u) \in X \setminus \{u\}$ and $r=r(u) > 0$ such that 
			$B(u,r) \subset \overline{\cA_{c,3\sigma}} + B(0,1)$, $v \not \in \ov{B(u,r)}$ and 
				\[
					\lambda \left( \Psi(v) - \Psi(w) \right) - \Phi'(w) (v-w) \leq \frac{\delta_0}{2} \| v - w \| \quad \text{for any $w \in B(u,r)$}. 
				\]
			In addition, if $w \in B(u,r)$ satisfies $I_\lambda(w) \geq c - 2 \sigma$, then 
				\[
				\lambda \left( \Psi(v) - \Psi(w) \right) - \Phi'(w) (v-w) \leq -3 \delta_0 \| v - w \|. 
				\]
		\end{enumerate}
	}
	\begin{proof}
To prove (a), let $u \in \ov{K} \cap \cA_{c,3\sigma}$ and $v$ be as in \eqref{ineq-I'}. 
By the lower semicontinuity of $\Psi$ and the fact that $v \neq u$, 
there exists $r = r(u) > 0$ such that $B(u,r) \subset \ov{\cA_{c,3\sigma}} + B(0,1)$, 
$v \not \in \ov{B(u,r)}$ and for all $w \in B(u,r)$
	\[
		\lambda \left( \Psi(v) - \Psi(w) \right) - \Phi'(w) (v- w) < -4\delta_0 \| v - w \|.
	\]
Hence, (a) holds.

	To prove (b) and (c), recall $\ov{K} \subset \ov{\cA_{c,3\sigma}}$ and \eqref{eq:uppest-u}:
	\[
		I_\lambda(u) \leq c + 3 \sigma \quad \text{for any $u \in \ov{K}$}. 
	\]
Therefore, if $u \in \ov{K} \setminus \cA_{c,3\sigma}$, 
then the closedness of $A$ and the fact $K \subset A$ yield $u \in A$ and 
	\begin{equation}\label{e:ulevel}
		I_\lambda(u) < c - 3\sigma. 
	\end{equation}

	We prove (b). Let $u \in \left( \ov{K} \setminus \cA_{c,3\sigma} \right)  \cap \cK_\lambda$ 
and put $v \equiv u$. By $u \in \cK_\lambda$,  
	\[
		\lambda \left( \Psi (w) - \Psi(v) \right) - \Phi'(v) (w-v) \geq 0 \quad \text{for any $w \in X$}.
	\]
Thus, 
	\[
		\begin{aligned}
			\lambda \left( \Psi(v) - \Psi(w)  \right) - \Phi'(w) (v-w) 
			&\leq (\Phi'(w)-\Phi' (v) ) (w-v) \\
			&\leq \| \Phi'(w) - \Phi'(v) \|_* \| w - v \|. 
		\end{aligned}
	\]
By the continuity of $\Phi'$, if $r = r(u)>0$ is sufficiently small, 
then for any $w \in B(u,r) \subset \overline{\cA_{c,3\sigma}} + B(0,1) $
	\[
		\lambda \left( \Psi(v) - \Psi(w) \right) - \Phi'(w) (v-w) 
		\leq \frac{\delta_0}{2} \| w - v \|. 
	\]
Hence, the first inequality in (b) holds.

	When $w \in B(u,r)$ satisfies $I_\lambda(w) \geq c - 2 \sigma$, 
by recalling \eqref{e:ulevel} and writing $I_\lambda (u) = c - 3\sigma - a_u$ with some $a_u>0$, 
it follows that 
	\begin{equation}\label{e:wu-est}
		\begin{aligned}
			\sigma + a_u 
			&\leq  
			I_\lambda(w) - I_\lambda (u) 
			\\
			&= 
			\lambda \left( \Psi(w) - \Psi(u) \right) + \left( \Phi(u) - \Phi(w) \right)
			\\
			&= 
			\lambda \left( \Psi(w) - \Psi(u) \right) + \Phi'(w) (u-w) 
			+ \int_0^1 \left\{ \Phi'(w+\theta (u-w)) - \Phi'(w) \right\} \left(u -w \right) \rd{\theta} .
		\end{aligned}
	\end{equation}
From the continuity of $\Phi'$ at $u$, shrinking $r=r(u)$ if necessary, we obtain 
	\[
		\sigma \leq \lambda \left( \Psi(w) - \Psi(u) \right) + \Phi'(w) (u-w) 
		\quad \text{for every $w \in B(u,r)$ with $I_\lambda (w) \geq c -  2 \sigma$}
	\]
or equivalently, 
	\[
		\lambda \left( \Psi(u) - \Psi(w) \right) - \Phi'(w) (u-w) \leq - \sigma
		\quad \text{for every $w \in B(u,r)$ with $I_\lambda (w) \geq c -  2 \sigma$}. 
	\]
Recalling that we have chosen $v=u$, up to further shrinking $r=r(u) > 0$ if necessary we infer
	\[
			\lambda \left( \Psi(v) - \Psi(w) \right) - \Phi'(w) (v-w) 
			\leq - \sigma \leq -3\delta_0 \| v - w \| 
			\quad \text{for any $w \in B(u,r)$ with $I_\lambda (w) \geq c - 2\sigma$,}
		\]
which completes the proof of (b).

	For (c), by $u \not \in \cK_\lambda$ we may find $v=v(u) \in X \setminus \{u\}$ and $b_u>0$ such that 
\[
\lambda \left( \Psi(v) - \Psi(u) \right) - \Phi'(u) (v-u) = -2b_u < 0.
\]
We claim that $v$ can be chosen arbitrarily close to $u$. In fact, consider 
\[
v_t \equiv t v + (1-t) u \qquad t \in (0,1). 
\]
From the convexity of $\Psi$, it follows that 
\[
\Psi(v_t) \leq t \Psi(v) + (1-t) \Psi(u)
\]
and 
\begin{equation}\label{e:v_tu}
	\lambda ( \Psi(v_t) - \Psi(u) ) - \Phi'(u) (v_t - u) 
	\leq t \left\{ \lambda \left( \Psi(v) - \Psi(u) \right) - \Phi'(u) (v-u) \right\} = -2b_u t < 0. 
\end{equation}
For each $t \in (0,1)$, there exists $r_t > 0$ such that 
$B(u,r_t) \subset \overline{\cA_{c,3\sigma}} + B(0,1)$, $v_t \not \in \ov{B(u,r_t)}$ and 
\[
\lambda \left( \Psi (v_t) - \Psi(w) \right) - \Phi'(w) (v_t - w) < - b_u t \leq \frac{\delta_0}{2} \| v_t - w \| \quad \text{for every $w \in B(u,r_t)$}. 
\]
Thus, the first inequality in (c) holds.

	Let $w \in B(u,r_t)$ satisfy $I_\lambda (w) \geq c - 2 \sigma$. 
Writing $I_\lambda(u) = c - 3\sigma - a_u$ with $a_u>0$, in a similar way to obtain \eqref{e:wu-est}, we have 
	\[
		\sigma + a_u \leq \lambda \left( \Psi(w) - \Psi(u) \right) - \Phi'(u) (w-u) + o(\| w - u \|). 
	\]
By shrinking $r_t>0$ if necessary, we get 
\begin{equation}\label{e:wu}
	\sigma \leq \lambda \left( \Psi(w) - \Psi(u) \right)  - \Phi'(u) (w-u) \quad 
	\text{for all $w \in B(u,r_t)$ with $I_\lambda (w) \geq c - 2\sigma$}.
\end{equation}
Hence, \eqref{e:v_tu} and \eqref{e:wu} imply 
\[
\begin{aligned}
	&\lambda \left( \Psi(v_t) - \Psi(w) \right) - \Phi'(w) (v_t - w) 
	\\
	= \ &
	\lambda \left\{ \Psi(v_t) - \Psi(u) + \Psi(u) - \Psi(w) \right\}
	- \Phi'(w) (v_t-w) + \Phi'(u) (w-u) 
	- \Phi'(u) (w-u)
	\\
	\leq \ & \lambda \left( \Psi(v_t) - \Psi(u) \right) - \Phi'(w) (v_t-w) - \Phi'(u) (w-u) - \sigma 
	\\
	< \ & \Phi'(u) (v_t - u)- \Phi'(w) (v_t-w) - \Phi'(u) (w-u) - \sigma 
	\\
	= \ & 
	\left( \Phi'(u) - \Phi'(w) \right) (v_t - w) - \sigma
	\\
	\leq \ & 
	\| \Phi'(u) - \Phi'(w) \|_* \| v_t - w \| - \sigma 
	\\
	= \ & \left\{ \| \Phi'(u) - \Phi'(w) \|_* - \frac{\sigma}{\| v_t - w \|} \right\} \| v_t - w \|. 
\end{aligned}
\]
Since $t \in (0,1)$ is arbitrary and we may take a smaller $r_t$, for sufficiently small $t$ and $r_t$, we get 
\[
\lambda \left( \Psi(v_t) - \Psi(w) \right) - \Phi'(w) (v_t - w) \leq -3 \delta_0 \| v_t - w \| \quad 
\text{for each $w \in B(u,r_t)$ with $I_\lambda (w) \geq c - 2\sigma$}
\]
and (c) holds. 
	\end{proof}

To proceed, since $\overline{K}$ is compact and $\Phi' \in C(X,X^\ast)$, there exists $r_1 > 0$ such that 
\begin{equation}\label{e:diff-Phi'}
	\sup \Set{ \norm{ \Phi'(v) - \Phi'(u) }_*  | u \in \ov{K}, \ \ \| v - u \| \leq r_1 } 
	\leq \frac{\delta_0}{4}. 
\end{equation}

\smallskip 

	\noindent
	\textbf{Step 2:} 
	\textsl{Construction of $\eta$. }

	\begin{proof}
For each $u \in \ov{K}$, we have $v = v(u)$ and $B(u,r)$ as in Step 1. 
Since $r = r(u)$ can be taken small, we may suppose 
	\[
	r(u) \le \frac{r_1}{2}, \qquad \text{that is,} \qquad 	\| u - w \| \leq \frac{r_1}{2} \quad \text{for all }\, w \in B(u,r(u)), 
	\]
with $r_1$ as in \eqref{e:diff-Phi'}. Using the compactness of $\ov{K}$, from $\ov{K} \subset \bigcup_{u \in \ov{K}} B(u,r(u))$ we infer the existence of finitely many distinct points $u_1,\dots, u_k \in \ov{K}$ so that 
	\[
		\ov{K} \subset \bigcup_{i=1}^k B(u_i,R_i), \qquad 
		\dist \left( \overline{K} , w \right) \leq \frac{r_1}{2} \qquad  \quad \text{for all } \, w \in \bigcup_{i=1}^k B(u_i,R_i), 
	\]
where we write $R_i \equiv r(u_i)$. We also use the symbol $v_i = v(u_i)$. 
Since the $u_i$'s are finite, we may find $r_0 \in (0,r_1)$ such that 
	\[
		\begin{array}{ll}
			\ov{B(u_i,2 r_0)} \subset B(u_i,R_i) & \qquad \text{for } \, 1 \leq i \leq k, 
			\\[0.3cm]
			\ov{B(u_i, 2r_0)} \cap \ov{B(u_j,2 r_0)} = \emptyset 
			& \qquad \text{for each $i,j \in \{1, \dots, k\}$ with $i \neq j$}. 
		\end{array}
	\]

	Put 
	\[
		V_i \equiv B(u_i, R_i) \setminus \bigcup_{j \in \{ 1,\dots, k \}, \, j\neq i} \overline{B (u_j, r_0)} \quad 
		\text{for $i=1,\dots, k$}, 
	\]
and define 
	\[
	\wt{W} \equiv \bigcup_{i=1}^k V_i.
	\]	
From the definition of $V_i$ and the choice of $r_0 \in (0,r_1)$, we observe that 
\begin{equation}\label{e:ball-Vi}
		\begin{aligned}
			&\overline{B(u_i,2r_0)} \subset V_i \quad (1 \leq i \leq k), \quad 
			\dist \left( \overline{K} , w \right) \leq \frac{r_1}{2} \quad \text{for any $w \in \wt{W}$},  
			\\
			&
			\wt{K} \equiv \overline{K} \cup \bigcup_{i=1}^k \overline{B(u_i,r_0)} \subset \wt{W}, 
			\qquad  0 < \dist \left( \wt{K} , \, \partial \wt{W} \right). 
		\end{aligned}
	\end{equation}
Notice that the last property in \eqref{e:ball-Vi} follows from
	\[
		\dist \left( \overline{K} , \, \partial \wt{W} \right) > 0, \qquad 
		\dist \left(  \bigcup_{i=1}^k \overline{B(u_i,r_0)} , \, \partial \wt{W} \right) > 0.
	\]
Moreover, by the very definition of $V_j$,  
	\begin{equation}\label{e:pos-dist}
		\ov{B(u_i,r_0)} \cap V_j = \emptyset \qquad \text{for each $i,j \in \{1,\dots, k\}$ with $i \neq j$}.
	\end{equation}

For $\e>0$, write 
	\[
		V_{i,\e} \equiv \left\{ x \in V_i \ | \ \dist (x, V_i^c ) > \e  \right\} \quad \text{for $i=1,\dots, k$}.
	\]
From \eqref{e:ball-Vi}, there exists $\zeta_0>0$ such that
	\[
		\overline{B(u_i,r_0)} \subset V_{i, 3\zeta_0} \quad \text{for $1 \leq i \leq k$}, \quad 
		\wt{K} \subset \bigcup_{i=1}^k V_{i,3\zeta_0} \equiv W_1. 
	\]
We also set 
	\[
		W_2 \equiv \bigcup_{i=1}^k V_{i,2\zeta_0}, \quad W_3 \equiv \bigcup_{i=1}^k V_{i,\zeta_0}, \quad 
		W_4 \equiv \wt{W}. 
	\]
Then, 
	\begin{equation*}
		\overline{W_j} \subset W_{j+1} \quad \text{and} \quad \dist \left( W_j, X \setminus W_{j+1} \right) > 0 \quad (1 \leq j \leq 3). 
	\end{equation*}
Define 
	\[
		V_0 \equiv X \setminus \overline{W}_1, \quad V_{0,\zeta_0} \equiv X \setminus \overline{W}_2. 
	\]
It is easily seen that
	\[
		X = \left( X \setminus \overline{W}_2\right) \cup W_3 = \bigcup_{i=0}^k V_{i,\zeta_0}, \quad 
		\dist \left( V_{0,\zeta_0} , \, X \setminus V_0 \right) = \dist \left( \overline{W_1} , X \setminus \overline{W_2} \right) > 0.
	\]
Furthermore, from $\wt{K} \subset W_1$, it follows that 
	\begin{equation}\label{int-V0-B}
		V_0 \cap \overline{B(u_i,r_0)} = \emptyset \qquad (1 \leq i \leq k). 
	\end{equation}

	For $i=0,\dots,k$, pick $\wt{\rho}_i \in C(X,\R)$ so that 
	\[
		0 < \wt{\rho}_i \quad \text{on} \ V_{i,\zeta_0}, \quad 
		\supp \wt{\rho}_i \subset V_i, \quad \dist \left( \supp \wt{\rho}_i, \, X \setminus V_i \right) > 0
	\]
and set 
	\[
		\rho_i (w) \equiv \frac{ \wt{\rho}_i(w) }{ \sum_{j=0}^k \wt{\rho}_j(w) }.
	\]
Then for each $i=0,1,\dots,k$, 
	\begin{equation}\label{e:partofunity}
		\begin{aligned}
			&\rho_i \in C( X , [0,1] ), \quad 
			\sum_{i=0}^k \rho_i = 1 \quad \text{on} \ X, 
			\quad 
			\supp \rho_i \subset V_i, \quad \dist \left(  \supp \rho_i, X \setminus V_i \right) > 0. 
		\end{aligned}
	\end{equation}
In particular, $\rho_i = 0$ on $X \setminus W_4$ for every $i =1,\dots, k$.

	To define $\eta(s,w)$, we consider the decomposition $\{1,\dots, k\} = J_1 \cup J_2$, where 
	\begin{equation}\label{def_J1J2}
		J_1 \equiv \left\{ i \ | \ u_i \in \cK_\lambda \right\}, \quad J_2 \equiv \{1,\dots, k\} \setminus J_1. 
	\end{equation}
By the definition of $v_i$ and by Step 1, 
	\begin{equation}\label{eq_nice_vi}
	v_i = u_i \qquad \text{if } \, i \in J_1, \qquad 
	v_i \not \in \ov{V}_i  \qquad \text{if } \, i \in J_2.
	\end{equation}
For any $i \in J_1$, we define $\eta_i \in C([0,\infty) \times X , X)$ by 
		\[
			\eta_i (s,w) \equiv 
			\begin{dcases}
				\rho_i(w) \left[w + s \frac{v_{i} - w}{\| v_i - w\|} \right]
				& \text{if $0 \leq s < \| v_i - w\|$},\\
				\rho_i(w) v_i & \text{if $s \ge \| v_i - w \|$}.
			\end{dcases}			 
		\]
On the other hand, for $i \in J_2$ we put
	\[
		\eta_i(s,w) \equiv \rho_i(w) \left[w + s \frac{v_{i} - w}{\| v_i - w\|} \right].
	\]
By \eqref{eq_nice_vi}, $v_i \not\in \overline{V}_i \supset \supp \rho_i$, so setting $\eta_i(s,w) = 0$ for $w=v_i$ yields a continuous map $\eta_i \in C([0,\infty) \times X, X  )$. 
Finally, for $i=0$, define 
	\[
		\eta_0(s,w) \equiv \rho_0 (w)w \in C( [0,\infty) \times X, X ). 
	\]
Using these $\eta_i$'s, we define $\eta(s,w)$ by 
	\[
		\eta(s,w) \equiv \sum_{i=0}^k \eta_i(s,w) \in C([0,\infty) \times X , X)
	\]
and set 
	\[
		W \equiv W_1 = \bigcup_{i=1}^{k} V_{i, 3\zeta_0}.
	\]
Remark that $W$ is an open set satisfying 
$\overline{K} \subset W \subset \overline{W} \subset \wt{W}$. 
Then, we shall prove that $\eta$, $W$ and $\wt{W}$ satisfy (i)--(iii) in Lemma \ref{l:defor-lem}. 
	\end{proof}


	\noindent
	\textbf{Step 3:} 
	\textsl{$\eta$ satisfies {\rm (i)--(iii)}.}

	\begin{proof}
We first prove (i). Notice that for each $(s,w) \in [0,\infty) \times X$ and $i \in J_2 \cup \{0\}$, 
	\[
		\| \eta_i(s,w) - \rho_i (w) w \| \leq \rho_i(w) s.
	\]
When $i \in J_1$, if $\rho_i(w) > 0$ and $ 0 \leq s < \| v_i - w\|$, then 
	\[
		\| \eta_i(s,w) - \rho_i (w) w \| \leq \rho_i(w) s. 
	\]
On the other hand, if $i \in J_1$, $\rho_i(w)>0$ and  $s \geq \| v_i - w\|$, then 
	\[
		\| \eta_i(s,w) - \rho_i (w)w \| \leq \rho_i (w) \| v_i - w\| \leq \rho_i (w) s.
	\]
Combining these estimates, we observe that for every $(s,w) \in [0,\infty) \times X$, 
	\[
		\| \eta(s,w) - w\| \leq \sum_{i=0}^k \| \eta_i(s,w) - \rho_i (w)w \| \leq \sum_{i=0}^k s \rho_i(w) = s.
	\]
Thus, the first assertion in (i) holds. 

Regarding the second assertion in (i), from the definition of $\wt{W}$ and \eqref{e:partofunity}, it follows that 
if $w \not\in \wt{W}$, then $\rho_j(w) = 0$ for every $j=1,\dots,k$ and $\rho_0(w) = 1$. 
Therefore, $\eta(s,w) = \rho_0 (w)w  = w$ holds for any $(s,w) \in [0,\infty) \times (X \setminus \wt{W})$, as claimed.

It remains to prove (ii) and (iii). 
Without loss of generality, we may assume $w \in \wt{W}$. We split the proof into two cases:

\medskip 

\noindent \textbf{First case}: assume that 
\[
w \in \bigcup_{i \in J_1} B(u_i,r_0),
\]
where $J_1$ is as in \eqref{def_J1J2}. Choose an index $i_0 \in J_1$ for which $w \in B(u_{i_0},r_{0})$. In this case, by $\overline{B(u_{i_0} , r_0)} \subset V_{i_0,3\zeta_0}$, \eqref{e:pos-dist}, \eqref{int-V0-B} 
and \eqref{e:partofunity}, we remark that $\rho_j(w) = 0$ for any $j \in \{0,\dots,k\} $ with $j \neq i_0$, hence 
\begin{equation}\label{eq_propriety_case1}
\rho_{i_0}(w) = 1, \qquad \eta(s,w) = \eta_{i_0}(s,w).
\end{equation} 
Also, $u_{i_0} \in \cK_\lambda$ and $v_{i_0} = u_{i_0}$ by \eqref{eq_nice_vi}.

	If $w = v_{i_0}$, then by definition we have $\eta_{i_0}(s,w) = w$ for any $s \geq 0$. 
Also, we claim that $I_\lambda(w) < c - 3 \sigma$, which directly implies (ii) and there is nothing to prove for (iii). 
Indeed, otherwise, by combining \eqref{eq:uppest-u}, $w = v_{i_0} \in \cK_\lambda \cap \overline{K}$ and 
the closedness of $A$, from $\overline{K} \subset A$ we would deduce $w \in A \cap \{ c - 3 \sigma \leq I_\lambda \leq c + 3 \sigma \} = \cA_{c,3\sigma}$, 
which violates \eqref{ineq-I'}.

Assume therefore that $w \in B(u_{i_0} , r_0) \setminus \{ u_{i_0} \}$. 
From Step 1 and the definition of $\eta_{i_0}$, if $ 0 \leq s \leq \| w - v_{i_0} \|$, then
	\[
		\eta(s,w) = \eta_{i_0}(s,w) = \left( 1 -  \tau_{i_0}   \right) w + \tau_{i_0} v_{i_0} \quad 
		\text{where} \ \tau_{i_0} = \tau_{i_0} (s,w) \equiv \frac{s}{\| w - v_{i_0} \|} \leq 1 
	\]
and by the convexity of $\Psi$, 
	\[
		\begin{aligned}
			I_\lambda(\eta (s,w)) 
			&\leq 
			\lambda (1-\tau_{i_0}) \Psi(w)
			+ \lambda \tau_{i_0} \Psi(v_{i_0}) - \Phi(\eta_{i_0} (s,w) )
			\\
			&= 
			I_\lambda (w) + \Phi(w) + \tau_{i_0} \lambda ( \Psi(v_{i_0}) - \Psi(w)  )  - \Phi(\eta_{i_0} (s,w) ) .
		\end{aligned}
	\]
Notice that by $v_{i_0} = u_{i_0} \in \cK_\lambda \cap \overline{K}$, $\tau_{i_0} \leq 1$ and $r_0 < r_1$,  
on the segment
	\[
		[0,1] \ni \theta \ \ \mapsto \ \ w_\theta \equiv w + (1-\theta) \tau_{i_0} (v_{i_0} - w ),
	\]
it holds  
	\[
		\begin{aligned}
			\left\|w_\theta - v_{i_0} \right\| 
			&= 
			\left\{ 1 - (1 -\theta) \tau_{i_0} \right\} \left\| w - v_{i_0} \right\| 
			\leq r_0 < r_1 \qquad \text{for any } \theta \in [0,1].
		\end{aligned}
	\]
From \eqref{e:diff-Phi'}, we get
	\[
		\begin{aligned}
			\Phi(w) - \Phi( \eta_{i_0} (s,w) ) 
			&= 
			\int_0^1 \frac{\rd}{\rd \theta} \Phi \left( w_\theta \right) \dd{\theta} =  - \tau_{i_0} (v_{i_0} - w) \int_0^1 \Phi' \left(w_\theta\right) \dd{\theta}
			\\
			&= 
			- \left( \tau_{i_0} (v_{i_0} - w) \right) \left\{ \Phi'(w) + \int_0^1 \left[ \Phi'(w_\theta) - \Phi'(v_{i_0}) + \Phi'(v_{i_0})- \Phi'(w) \right]\dd{\theta} \right\}.
			\\
			&\le -\tau_{i_0} \Phi'(w) (v_{i_0} - w) + 2  \tau_{i_0} \|v_{i_0} - w\| \max_{\theta \in [0,1]} \big\|\Phi'(w_\theta) - \Phi(v_{i_0})\big\|.
			\\
			& \leq 
			- \tau_{i_0} \Phi'(w)  (v_{i_0} - w) 
			+ \frac{ \delta_0}{2} \tau_{i_0} \| v_{i_0} - w \| 
			\\
			&= 
			- \tau_{i_0} \Phi'(w)  (v_{i_0} - w) + \frac{ \delta_0}{2} s. 
		\end{aligned}
	\]
This implies that for all $w \in B(u_{i_0}, r_0) \setminus \{u_{i_0}\}$ ($i_0 \in J_1$) 
and $s \in [0, \| v_{i_0} - w \|]$, 
	\begin{equation}\label{e:bas-inq}
		I_\lambda(\eta(s,w)) 
		\leq 
		I_\lambda(w) 
		+ \tau_{i_0} \left[ \lambda \left\{ \Psi(v_{i_0}) - \Psi(w) \right\} - \Phi'(w) (v_{i_0} - w)  \right] 
		+ \frac{\delta_0}{2} s. 
	\end{equation}
Recalling $\tau_{i_0} \leq 1$ and Step 1 (b), we observe that 
for every $w \in B(u_{i_0}, r_0) \setminus \{u_{i_0}\} \subset V_{i_0} $ and $s \in [0, \| v_{i_0} - w \|]$, 
	\begin{equation}\label{e-ii-smalls}
		I_\lambda(\eta (s,w)) 
		\leq 
		I_\lambda (w) + \tau_{i_0} \frac{\delta_0}{2} \| v_{i_0} - w \| + \frac{\delta_0}{2}s  
		= 
		I_\lambda (w) + \delta_0 s. 
	\end{equation}
Finally, if $s > \| v_{i_0} - w \|$, then $\eta (s,w) = \eta_{i_0} (s,w)$ and the definition of $\eta_{i_0}$ yield 
	\[
		I_\lambda (\eta(s,w)) = I_\lambda (\eta ( \| v_{i_0} - w \| , w  ) ) \leq I_\lambda (w) + \delta_0 \| v_{i_0} - w \| 
		\leq I_\lambda (w) + \delta_0 s. 
	\]
In conclusion, (ii) is verified for all $(s,w) \in [0,\infty) \times B(u_{i_0} , r_0)$ with $i_0 \in J_1$. 

To check (iii), let $w \in B(u_{i_0}, r_0) \setminus \{u_{i_0}\} $ satisfy $I_\lambda(w) \geq c - \sigma$. 
For any $s \in [0, \| v_{i_0} - w \| ]$, and taking into account the definition of $\tau_{i_0}$,  
Step 1 (b) and \eqref{e:bas-inq} imply 
	\[
		I_\lambda (\eta(s,w)) 
		\leq 
		I_\lambda (w) - \tau_{i_0} 3\delta_0 \| v_{i_0} - w \| + \frac{\delta_0}{2} s 
		\leq 
		I_\lambda (w) - 2 \delta_0 s.
	\]
On the other hand, when $\| v_{i_0} - w \| < s$, it follows from $I_\lambda (v_{i_0}) < c - 3\sigma$ that 
	\[
		I_\lambda (\eta(s,w)) = I_\mu(v_{i_0}) < c - 3 \sigma \leq I_\lambda (w) - 2 \sigma. 
	\]
Since $\delta_0 < \sigma$ holds due to \eqref{e:del-sig}, if $\| v_{i_0} - w \| < s \leq 1$, then 
	\[
		I_\lambda (\eta(s,w)) < I_\lambda (w) - 2\delta_0 s.
	\]
Hence, for every $(s,w) \in [0,1] \times B(u_{i_0}, r_0)$ with $i_0 \in J_1$ and $I_\lambda (w) \geq c - \sigma$, 
	\[
		I_\lambda (\eta(s,w)) \leq I_\lambda (w) - 2 \delta_0 s.
	\]
Thus, (iii) holds on $[0,1] \times \bigcup_{i \in J_1} B(u_{i} , r)$, too. 

\medskip 
\noindent \textbf{Second case}: assume that 
\[
w \in \wt{W} \setminus \bigcup_{i \in J_1} B(u_i,r_0).
\]
By the definition of $\eta_i$, for any $s \in [0,r_0/2]$ (so that $s < \|w-v_i\|$ for each $i \in J_1$) we have
	\begin{equation}\label{e:u+pert}
		\eta(s,w) = w + s \sum_{j=1}^k \rho_j(w) \frac{v_j-w}{\| v_j - w \|} 
		= w + s x, \quad x \equiv \sum_{j=1}^k \rho_j(w) \frac{v_j - w}{\| v_j - w \|}, \quad 
		\| x \| \leq 1.
	\end{equation}
Moreover,  it follows from \eqref{e:partofunity} and the second in \eqref{eq_nice_vi} that 
	\[
		\dist \left( v_j , \supp \rho_j \right) > 0 \quad \text{for every $j \in J_2$}. 
	\]	
Also, since $v_j = u_j$ for all $j \in J_1$, we may choose $s_1 \in (0, r_0/2)$ so that 
	\[
		s \sum_{j=1}^k \frac{\rho_j (w) }{\| v_j - w \|} \leq 1 \quad 
		\text{for every $s \in [0,s_1]$ and $w \in \wt{W} \setminus \bigcup_{i \in J_1} B(u_{i} , r_0)$}.
	\]
The inequality can be rewritten as
	\[
		\sum_{j=1}^k \tau_j(s,w) \leq 1, \quad 
		\tau_j (s,w) \equiv s \frac{\rho_j (w)}{\| v_j - w \|} \qquad \text{for each } \, 1 \leq j \leq k,
	\]
so that
	\[
		\eta(s,w) = \left(1 - \sum_{j=1}^k \tau_j  \right) w + \sum_{j=1}^k \tau_j v_j.
	\]
From the convexity of $\Psi$, we infer 
	\begin{equation}\label{e:est-Imu1}
		\begin{aligned}
			I_\lambda \left( \eta(s,w) \right) 
			&\leq
			\lambda \left[ \left( 1 - \sum_{j=1}^k \tau_j \right) \Psi(w) 
			+ \sum_{j=1}^k \tau_j \Psi(v_j) \right] - \Phi \left( \eta(s,w) \right)
			\\
			&= I_\lambda (w) + \sum_{j=1}^k 
			\tau_j \left\{ \lambda \left( \Psi(v_j) - \Psi(w) \right)  \right\} + \Phi(w) - \Phi(\eta(s,w)).
		\end{aligned}
	\end{equation}
We shall estimate the term $\Phi(w) - \Phi(\eta(s,w))$. 
Since $w \in \wt{W}$, $\dist (w, \overline{K}) \leq r_1/2$ holds due to \eqref{e:ball-Vi}. 
The compactness of $\overline{K}$ implies that there exists $y_w \in \overline{K}$ such that
	\[
		\| w - y_w \| \leq \frac{r_1}{2}. 
	\]
Therefore, if $0 \leq s \leq s_1$ and 
$\theta \in [0,1]$, then by $s_1 < r_0/2 < r_1/2$, 
	\[
		\left\| w + (1-\theta) s x - y_w \right\| \leq \frac{r_1}{2} + s < r_1. 
	\]
By \eqref{e:u+pert} and \eqref{e:diff-Phi'}, 
for every $s \in [0,s_1]$ we have  
		\begin{align*}
			\Phi(w) - \Phi(\eta(s,w)) 
			&= 
			\int_0^1 \frac{\rd}{\rd \theta } \Phi( w + (1-\theta) s x ) \dd{\theta}
			\\
			&= 
			- \Phi'(w) s x + \int_0^1 \left( \Phi'( w + (1-\theta) s x ) - \Phi'(w) \right) (-sx) \dd{\theta}
			\\
			&= 
			-\Phi'(w)  sx + \int_0^1 \left( \Phi'( w +(1-\theta) s x ) - \Phi'( y_w ) \right) (-sx) \dd{\theta}
			\\
			&\quad + \int_0^1 \left( \Phi'( y_w ) - \Phi'(w) \right) (-s x) \dd{\theta}
			\\
			&\leq - \Phi'(w) sx + \frac{\delta_0}{2} s 
			= - \sum_{j=1}^k \tau_j \Phi'(w) (v_j - w) + \frac{\delta_0}{2}s. 
		\end{align*}
Plugging into \eqref{e:est-Imu1}, for $s \in [0,s_1]$ we get
	\begin{equation}\label{e:basic-ineq2}
		I_\lambda (\eta (s,w)) 
		\leq 
		I_\lambda (w) 
		+ \sum_{j=1}^k \tau_j 
		\left\{ \lambda \left( \Psi( v_j ) - \Psi(w) \right) - \Phi'(w) (v_j - w) \right\} + \frac{\delta_0}{2} s .
	\end{equation}
Notice that if $\tau_j > 0$ for some $j =1,\dots, k$, then $w \in V_j$ and thus Step 1 gives 
	\begin{align*}
	I_\lambda (\eta(s,w) ) 
	\leq I_\lambda (w) + \sum_{j=1}^k \tau_j \frac{\delta_0}{2} \| v_j - w \| + \frac{\delta_0}{2}s 
	= I_\lambda (w) + \sum_{j=1}^k \rho_j(w) \frac{\delta_0}{2} s + \frac{\delta_0}{2} s 
	\leq  I_\lambda (w) + \delta_0 s.
	\end{align*}
This concludes the proof of (ii) whenever $(s,w) \in [0,s_1] \times \left( \wt{W} \backslash \bigcup_{i \in J_1} B(u_i,r_0)\right)$.

	To check the validity of (iii), assume that 
\[
w \in W \setminus \bigcup_{i \in J_1} B(u_{i}, r_0), \qquad I_\lambda (w) \geq c - \sigma.
\] 
Since $W \equiv W_1$ and $V_0 \cap W = \emptyset$, \eqref{e:partofunity} yields $\rho_0(w) = 0$. 
Thus, from \eqref{e:basic-ineq2} and Step 1 it follows $ \sum_{j=1}^k \rho_j(w) = 1$ and 
	\begin{align*}
	I_\lambda (\eta(s,w)) 
	&\leq I_\lambda (w) + \sum_{j=1}^k \tau_j (-3\delta_0 \| v_j - w \| ) 
		+ \frac{\delta_0}{2} s \\ 
	&= I_\lambda (w) -3\delta_0 \sum_{j=1}^k \rho_j(w) s + \frac{\delta_0}{2} s 
	\leq I_\lambda (w) - 2\delta_0 s
	\end{align*}
for each $s \in [0,s_1]$, and (iii) holds. 
We have thus shown properties (ii) and (iii) in all cases. 
	\end{proof}

	From Step 3, by choosing $s_0 \equiv \min \{ s_1, 1 \}$, 
all the conclusions in Lemma \ref{l:defor-lem} hold and we complete the proof. 
	\end{proof}

\subsection{Symmetric Deformation Lemma}

	Throughout this subsection we assume condition \eqref{even}, which is defined as  
\[
	\Psi(-u) = \Psi(u), \quad \Phi(-u) = \Phi(u) \quad \text{for any $u \in X$}.
\]
For a subset $A \subset X$, we say that $A$ is \textit{symmetric} if $-A = A$ holds, where $-A \equiv \{x \in X \mid -x \in A\}$. 

\begin{remark}\label{r:criticality-0}
	Under \eqref{even} and $(\Phi_1)$, it is easily seen that $\Phi'(0) = 0$. 
	Moreover, if $\Psi$ satisfies $(\Psi_1)$, then by $\Psi(u) \geq 0 = \Psi(0) $ for every $u \in X$, 
	we infer that $0$ is a critical point of $I_\lambda$. 
\end{remark}

To establish multiplicity results for even functionals, we need to construct an odd deformation which is the counterpart of Lemma \ref{l:defor-lem}.

\begin{lemma}[Symmetric Deformation Lemma]\label{sydefolemma}
	Assume that $(\Psi_1)$, $(\Phi_1)$ and \eqref{even} hold. 
	Fix $\lambda \in \R^+$ and let $A \subset X$ be a bounded, closed and symmetric set and 
	for $c \in \R$ and $\sigma > 0$, define
	\[
	\cA_{c,3\sigma} \equiv A \cap \left\{ c - 3\sigma \leq I_\lambda \leq c + 3 \sigma \right\}.
	\]
	Suppose that there exists $\delta_0 > 0$ such that for each $u \in \cA_{c,3\sigma}$ 
	we may find $v = v(u) \in X \setminus \{u\}$ with 
	\begin{equation}\label{noPS2}
		\lambda \left( \Psi(v) - \Psi(u) \right) - \Phi'(u) (v-u) < - 5 \delta_0 \| v - u \|.
	\end{equation}
	Then for every relatively compact and symmetric set $K \subset \cA_{c,3\sigma}$ 
	there exist symmetric neighborhoods $W$ and $\wt{W}$ of $\overline{K}$ with 
		\[
		\overline{K} \subset W \subset \overline{W} \subset \wt{W}, 
		\]
		$s_0>0$ and $\eta \in C([0,\infty) \times X, X)$ such that 
	\begin{enumerate}[{\rm (i)}]
		\item
		$\| \eta(s,w) - w\| \leq s$ and $\eta(s,-w) = -\eta(s,w)$ for all $(s,w) \in [0,\infty) \times X$, 
		and $\eta(s,w) = w$ for all $s \in [0,\infty)$ if $w \not \in \wt{W}$; 
		\item
		$I_\lambda ( \eta(s,w) ) \leq I_\lambda(w) + \delta_0 s$ for all $(s,w) \in [0,s_0] \times X$; 
		\item
		$I_\lambda(\eta(s,w)) \leq I_\lambda(w) - 2\delta_0 s$ for all $(s,w) \in [0,s_0] \times W$ with $I_\lambda(u) \geq c - \sigma$. 
	\end{enumerate} 
\end{lemma}

\begin{proof}
	We follow the same scheme as in the proof of Lemma \ref{l:defor-lem}, just pointing out the differences needed to get the oddness of the deformation flow.

	By \eqref{even}, notice that $\cA_{c,3\sigma}$ is symmetric, namely $u \in \cA_{c,3\sigma}$ yields $-u \in \cA_{c,3\sigma}$. 
	Therefore, for every $u \in \cA_{c,3\sigma}$, \eqref{noPS2} and \eqref{even} imply 
	\[
	\lambda ( \Psi(-v) - \Psi(-u) ) - \Phi^\prime(-u) (-v - (-u))
	< - 5 \delta_0 \Vert -v - (-u)\Vert. \]
For each $u \in \cA_{c,3\sigma}$, by defining $\wt{v}(u) \equiv \{ v(u)-v(-u) \} /2,$
	we see from \eqref{even}, \eqref{noPS2} and the convexity of $\Psi$ that 
	\[ \begin{aligned}
		\lambda ( \Psi(\wt{v}(u)) - \Psi(u) ) - \Phi^\prime(u) (\wt{v}(u) - u ) 
		&\le  \frac 12 \left[ \lambda \left\{ \Psi(v(u)) - \Psi(u) \right\} - \Phi^\prime(u) (v(u) - u) \right] 
		\\
		& \quad +
		\frac 12 \left[ \lambda \left\{ \Psi(v(-u)) - \Psi(-u) \right\} - \Phi^\prime(-u) (v(-u) - (-u) )\right]  \\
		&<  \frac 12 \left( - 5 \delta_0 \Vert v(u) - u\Vert \right)+ \frac 12 \left( - 5 \delta_0 \Vert v(-u) - (-u)\Vert \right) \\
		&\leq  -5 \delta_0 \Vert \wt{v}(u) - u\Vert. \end{aligned}
	\]
Remark that $\wt{v} (u) \neq u$ since the second-last inequality is strict. 
Thus, we may assume that $v$ is odd with respect to $u \in X$.

	For Step 1 in the proof of Lemma \ref{l:defor-lem}, we remark that $\overline{K}$ and $\cK_\lambda$ are symmetric. 
	Therefore, in (a) and (b), we can take $v(u)$ and $r(u)$ so that $v(-u) = - v(u)$ and $r(-u) = r(u)$, and 
	prove Step 1 (a) and (b). 
	For Step 1 (c), if $v \in X \setminus \{u\}$ satisfies 
	\[
	\lambda \left( \Psi(v) - \Psi(u) \right) - \Phi'(u) (v-u) = - 2b_u < 0,
	\]
	then \eqref{even} gives 
	\[
	\lambda \left( \Psi(-v) - \Psi(-u) \right) - \Phi'(-u) \left( - v - (-u) \right) = - 2b_u < 0.
	\]
	Hence, following the proof of Step 1 (c), we can see that the claim in (c) holds with $v(-u) = - v(u)$ and $r(-u) = r(u)$.

	Since $\overline{K}$ is symmetric and compact, we notice that 
	there are $u_1,\dots, u_k \in \overline{K}$ such that 
	by setting $u_{k+1} = 0$ when $0 \in \overline{K}$, 
	\[
	\overline{K} \subset 
	\begin{dcases} 
		\bigcup_{i=1}^k \left\{ B(u_i, r(u_i)) \cup B( - u_i , r(u_i) ) \right\} & \text{if $0 \not \in \overline{K}$},
		\\
		B(u_{k+1} , r(0) ) \cup \bigcup_{i=1}^k \left\{ B(u_i, r(u_i)) \cup B( - u_i , r(u_i) ) \right\} & \text{if $0 \in \overline{K}$}.
	\end{dcases}
	\]
	Write $u_{-j} \equiv - u_j$ for $1 \leq j \leq k+1$ and $R_i \equiv r(u_i)$ for $1 \leq |i| \leq k+1$, and 
	when $0 \not \in \overline{K}$ define 
	\[
	\begin{aligned}
		V_i &\equiv B(u_i,R_i) \setminus \bigcup_{1 \leq |j| \leq k, \, j \neq i} \overline{B(u_j,r_0)} & &\text{for $1 \leq |i| \leq k$}
	\end{aligned}
	\]
and when $ 0 \in \overline{K}$, 
	\[
		\begin{aligned}
			V_i& \equiv B(u_i,R_i) \setminus \bigcup_{1 \leq |j| \leq k+1, \, j \neq i}  \overline{B(u_j,r_0)} & &\text{for $1 \leq |i| \leq k$}, 
			\quad 
			V_{k+1} \equiv B(0,R_{k+1}) \setminus \bigcup_{1 \leq |j| \leq k} \overline{B (u_j, r_0)}. 
	\end{aligned}
	\]
	In any case, we have 
		\[
			V_{-j} = - V_j \quad \text{for $1 \leq j \leq k$}, \quad - V_{k+1} = V_{k+1}, \quad 
			- \bigcup_{1 \leq |i| \leq k} V_i = \bigcup_{1 \leq |i| \leq k} V_i. 
		\]
	Arguing as in the proof of Step 2 of Lemma \ref{l:defor-lem}, 
	we may find $( \wt{\rho}_i )_{1 \leq |i| \leq k}$ (resp. $ (\wt{\rho}_i)_{1 \leq |i| \leq k} $ and $\wt{\rho}_{k+1}$) such that 
	$ 0< \wt{\rho}_i$ on $V_{i,\zeta_0}$, $\supp \wt{\rho}_i \subset V_i$, $\wt{\rho}_{-j} (u) = \wt{\rho}_{j} (-u)$ for $1 \leq j \leq k$, 
	$\wt{\rho}_{k+1} (-u) = \wt{\rho}_{k+1} (u)$ and $\wt{\rho}_0(-u) = \wt{\rho}_0(u)$. 
	From this family of functions, we construct the partition of unity $(\rho_i)_{1 \leq |i| \leq k}$ (resp. $(\rho_{i})_{ 1 \leq |i| \leq k }$ and $\rho_{k+1}$) 
	as before.

	To define $\eta(s,u)$, consider 
	\[
	\begin{aligned}
		J_1 &\equiv \left\{ i \ | \ 1 \leq |i| \leq k , \ u_i \in \cK_\lambda \right\} \quad 
		\left(\text{resp. } \left\{ i \ | \ \text{$1 \leq |i| \leq k$ or $i=k+1$}, \ u_i \in \cK_\lambda \right\}  \right),
		\\
		J_2 & \equiv \left\{ -k, \dots, -1, 1, \dots, k \right\} \setminus J_1 .
	\end{aligned}
	\]
	We remark that $k+1 \in J_1$ if $0 \in \overline{K}$ and 
	see from the properties of $V_i$  that for $1 \leq j \leq k$, $j \in J_1$ if and only if $-j \in J_1$. 
	Defining $\eta_i(s,u)$ and $\eta(s,u)$ as in Step 2, we see that 
	$\eta_{-j} (s,u) = -\eta_{j} (s,-u)$ for $1 \leq j \leq k$, $\eta_{k+1} (s,-u) = -\eta_{k+1} (s,u)$ 
	and $\eta_0(s,-u) = - \eta_0(s,u)$, 
	hence $\eta(s,-u) = - \eta(s,u)$.

	The rest of the argument is unchanged compared to the proof of Lemma \ref{l:defor-lem}. 
\end{proof}

\section{Proof of Theorem \ref{mpt} and Theorem \ref{smt}}
\label{sec:pf-main}

In the proof of of Theorems \ref{mpt} and \ref{smt} we assumed $({\rm IB})$, that is, the boundedness of the set of critical points $\cK_{[a,b]}^c$ for each $[a,b] \subset (0,\infty)$ and $c \in \R$. As a matter of fact, under the further validity of $(\Psi_1),(\Psi_3),(\Phi_1)$ and $(\Phi_2)$ a stronger property holds: $\cK_{[a,b]}^c$ is compact. More precisely, we prove the next


\begin{prop}\label{IBtoIC}
Suppose that $(\Psi_1),(\Psi_3),(\Phi_1), (\Phi_2)$ hold. Assume that there exist $c \in \R$ and sequences $\{\e_j\}, \{\lambda_j\} \subset (0,\infty)$ and $\{u_j\} \subset X$ with 
\[
\e_j \to 0, \qquad \lambda_j \to \lambda \in (0,\infty), \qquad u_j \rightharpoonup u \ \ \text{weakly in $X$} \qquad \text{ as } \, j \to \infty
\]
and 
\[
\lambda_j( \Psi(v)-\Psi(u_j)) - \Phi'(u_j)(v-u_j) \ge - \e_j \norm{v-u_j} \qquad \text{ for each } \, v \in X.
\]
Then, 
\begin{equation}\label{eq_conclu_weakstrong}
\lim_{j \to \infty}\Vert u_j -u\Vert = 0, \qquad \text{$u$ is critical for $I_\lambda$,} \qquad I_{\lambda}(u) = \lim_{j \to \infty} I(u_j) \in \R.
\end{equation}
In particular, in the stated assumptions, 
\begin{itemize}
\item[-] for each $\lambda \in (0,\infty)$ and $c \in \R$, $I_\lambda$ satisfies the bounded Palais-Smale condition at level $c$, namely 
any bounded Palais-Smale sequence for $I_\lambda$ at level $c$ has a strongly convergent subsequence in $X$;
\item[-] property $({\rm IB})$ implies that the set $\cK_{[a,b]}^c$ is compact for each $[a,b] \subset (0, \infty)$ and $c \in \R$.
\end{itemize} 
\end{prop}
\begin{proof}
We first claim that $u \in D(\Psi)$ and $u$ is a critical point of $I_\lambda$. 
Remark that the sequence $\{u_j\}$ is bounded in $X$ due to the weak convergence. 
Hence, up to passing to a subsequence, 
we can assume that $\{u_j\}$ itself satisfies the conclusion in $(\Phi_2)$. 
Having fixed $v \in D(\Psi)$ and $\eps \in (0,\lambda)$, we select $j_0$ large enough so that 
$\lambda_j > \lambda-\eps$ for $j \ge j_0$. Since $\Psi \ge 0$, if $j \ge j_0$, then we get 
\[
	\begin{aligned}
		\lambda_j\Psi(v) - (\lambda-\eps)\Psi(u_j) - \Phi^\prime(u_j)(v-u_j) 
		&\ge \lambda_j(\Psi(v) - \Psi(u_j)) - \Phi^\prime(u_j)(v-u_j) \\
		&\ge - \e_j\norm{v-u_j} \ge - C\e_j
	\end{aligned}
\]
for some constant $C$. 
Taking a $\limsup$ as $j \to \infty$ and using $(\Psi_1)$ and $(\Phi_2)$, we infer
\[ 
\lambda\Psi(v) - (\lambda-\eps)\Psi(u) - \Phi^\prime(u)(v-u) \ge 0 \qquad \text{for every }v \in X.
\]
Hence, $\Psi(u)<\infty$, and letting $\eps \to 0$ we deduce that $u$ is critical for $I_\lambda$.

	Next, taking the $\limsup$ as $j \to \infty$ in the inequality 
\[
	\lambda_j \Psi(u) - \Phi'(u_j)(u-u_j) \ge \e_j \norm{u-u_j} + \lambda_j \Psi(u_j)
\]
and using $(\Phi_2)$ we get $\limsup_{j \to \infty} \Psi(u_j) \le \Psi(u).$ Since $\Psi$ is weakly lower-semicontinuous, it follows that $\lim_{j \to \infty} \Psi(u_j) = \Psi(u)$ and thus, by $(\Psi_3)$, $\lim_{j  \to \infty}\Vert u_j -u\Vert = 0$. 
Hence, by $(\Phi_1)$, $I_\lambda(u) = \lim_{j \to \infty}I_{\lambda_j}(u_j)$, 
which completes the proof of \eqref{eq_conclu_weakstrong}.

To check the bounded Palais-Smale condition, it is enough to extract from a bounded Palais-Smale sequence $\{u_j\}$ at level $c$ a weakly convergent subsequence $u_j \rightharpoonup u$ by using $(\Phi_2)$. Applying the above, $u_j$ converges strongly to $u$, as claimed.

	Similarly, if (IB) holds, then pick any bounded sequence $\{u_j\} \subset \cK_{[a,b]}^c$. 
Because of $(\Phi_2)$, we can extract a weakly convergent subsequence and, up to further subsequences, we can assume that $\lambda_j \to \lambda  \in [a,b]$ and $u_j \rightharpoonup u$ weakly in $X$ as $j \to \infty$. 
The above guarantees that $u_j \to u$ strongly in $X$ and that $u \in \cK_{[a,b]}^c$. Hence, $\cK_{[a,b]}^c$ is compact.
\end{proof}

To proceed, we collect some basic properties of $ c_k(\lambda),  k=0,1,2\ldots,$ defined in \eqref{mpv} and \eqref{smpv}. 
We first show 
	\begin{lemma}\label{lem:Gkck}
		If \eqref{mp}, $(\Psi_1)$ and $(\Phi_1)$ hold, then $ \Gamma _0 \neq \emptyset$ and $ 0< \alpha_0 \leq c_0(\lambda) < \infty$ 
		for each $\lambda \in [1-\overline{\e} , 1 + \overline{\e} ]$. 
		Similarly, $\Gamma_k \neq \emptyset$ and $\alpha_0 \leq c_k(\lambda) < \infty$ hold 
		for every $k \geq 1$ and $\lambda \in [1-\overline{\e} , 1 + \overline{\e} ]$ under \eqref{smp}, $(\Psi_1)$ and $(\Phi_1)$. 
	
	\end{lemma}

	\begin{proof}
For $\Gamma_0$ and $c_0(\lambda)$, 
consider a path $\gamma_0(t) \equiv t u_0$ ($t \in [0,1]$) where $u_0$ appears in \eqref{mp}. 
It is easily seen that $\gamma_0 \in \Gamma_0$, and $(\Psi_1)$ and $I_{1+\ov{\e}} (u_0) \leq 0$ yield 
$u_0 \in D(\Psi)$ and 
\[
0 \leq \Psi(\gamma_0(t)) \leq t \Psi(u_0) + (1-t) \Psi(0) = t \Psi(u_0) < \infty.
\] 
From $I_{1+\ov{\e}} (u_0) < 0$ and $(\Phi_1)$ we infer that  $\sup_{0 \leq t \leq 1} I_\lambda (\gamma(t)) < \infty $ and $c_0(\lambda) < \infty$ 
for every $\lambda \in [1-\ov{\e} , 1 + \ov{\e}]$. 
The assertion $\alpha_0 \leq c_0(\lambda)$ follows from the fact $\gamma([0,1]) \cap \partial B (0,\rho_0)$ and \eqref{mp}.

	For $\Gamma_k$ and $c_k(\lambda)$, consider a map $\wt{\gamma}_k \in C( \DD^k , X )$ defined by 
	\[
	\wt{\gamma}_k (\zeta) \equiv \begin{dcases}
		\abs{\zeta} \pi_{0,k} \qty( \frac{\zeta}{\abs{\zeta}} ) & \text{if} \ 0 < \abs{\zeta} \leq 1, \\
		0 & \text{if} \ \zeta = 0.
	\end{dcases}
	\]
	Clearly $\wt{\gamma}_k \in \Gamma_k$ holds. Moreover, $(\Psi_1)$, $(\Phi_1)$ and \eqref{smp} imply 
	\[
	\begin{aligned}
	\disp \sup_{\zeta \in \DD^k} \Psi( \wt{\gamma}_k (\zeta) ) 
	\leq \disp \sup_{\zeta \in \DD^k} \abs{\zeta} \Psi \qty( \pi_{0,k} \qty( \frac{\zeta}{\abs{\zeta}} ) )
	&\leq \sup_{\zeta \in \partial \DD^k} \Psi (\pi_{0,k} (\zeta)) \\ 
	&= (1+\overline{\e})^{-1} \sup_{\zeta \in \partial \DD^k} 
	\qty{  I_{1+ \overline{\e} } \qty( \pi_{0,k} (\zeta) ) + \Phi \qty( \pi_{0,k} (\zeta) )  } 
	\\
	&\leq \disp (1+\ov{\e})^{-1} \sup_{ \zeta \in \partial \DD^k } \Phi( \pi_{0,k} (\zeta) ) < \infty.
	\end{aligned}
	\]
	Hence, $c_k(\lambda) < \infty$ holds for all $\lambda \in [1-\ov{\e}, 1 + \ov{\e}]$. 
	Furthermore, $\alpha_0 \leq c_k(\lambda)$ follows from the fact $\partial B(0,\rho_0) \cap \gamma (\DD^k) \neq \emptyset$ for each $\gamma \in \Gamma_k$. 
	\end{proof}

	\begin{prop}\label{monot1}
	For each $k=0,1\ldots,$ the function $\lambda \mapsto c_k(\lambda)$ is nondecreasing on $[1 - \ov{\e} , 1 + \ov{\e}]$, hence differentiable at almost all $\lambda \in [1 - \ov{\e}, 1 + \ov{\e}]$. 
		Moreover, $c_k$ is continuous from the right in $[1 - \ov{\e} , 1 + \ov{\e})$; thus, $c_k$ is upper semicontinuous. 
	\end{prop}
\begin{proof}
The monotonicity of $c_k$ follows from $\Psi \ge 0$. For the right continuity of $c_k$, fix $\lambda \in [1- \ov{\e} , 1 + \ov{\e})$ and 
$\e \in (0,\infty)$ arbitrarily. Choose $\gamma_\e \in \Gamma_k$ so that 
$\sup_{ \zeta \in \DD^k } I_\lambda (\gamma_\e(\zeta)) < c_k(\lambda) + \e$. 
Since $\sup_{\zeta \in \DD^k} \Psi(\gamma_\e(\zeta)) < \infty$ due to 
$\max_{\zeta \in \DD^k} | \Phi(\gamma_\e(\zeta)) | < \infty$ and $\sup_{ \zeta \in \DD^k } I_\lambda (\gamma_\e(\zeta)) < c_k(\lambda) + \e$, 
 it follows that $ 0 \leq I_\mu (\gamma_\e(\zeta)) - I_\lambda (\gamma_\e (\zeta)) < \e$ for every $\zeta \in \DD^k$ 
 if $ \mu - \lambda >0$ is sufficiently small.
Then,   
	\[
		c(\mu) \leq \sup_{\zeta \in \DD^k } I_\mu(\gamma_\e(\zeta)) \le \sup_{ \zeta \in \DD^k } I_\lambda (\gamma_\e(\zeta)) + \e \le  c_k(\lambda) + 2\e
	\]
if $ \mu - \lambda >0$ is sufficiently small. This implies that $c_k$ is right continuous. 
	\end{proof}

Here is our main result on the existence of bounded Palais-Smale sequences.

\begin{thm}\label{amps}
Suppose that $(\Psi_1)$, $(\Psi_2),$ $(\Phi_1)$ and either \eqref{mp}, or \eqref{even} and \eqref{smp}, hold. 
Assume also that at $\lambda \in (1 - \ov{\e} , 1 + \ov{\e})$, $c'_k(\lambda)$ exists $($with $k=0$, under assumption \eqref{mp}$)$. 
Then there exists a Palais-Smale sequence $\{u_j\}_{j=1}^\infty$ for $I_\lambda$ such that 
$I_\lambda (u_j) \to c_k(\lambda)$ and $\{u_j\}_{j=1}^\infty$ is bounded in $X$. 
	\end{thm}

%

\begin{proof}
For a fixed $k \in \{0,1,\ldots\},$  assume that $c_k'(\lambda)$ exists at some  $\lambda \in (1 - \ov{\e} , 1 + \ov{\e})$. 
We put 
	\[
		M_0 \equiv c_k'(\lambda) + 3.
	\]
From now on, we just assume $k \ge 1$ since the case $k=0$ can be handled in the same manner.
From $(\Psi_2)$, there exists $R_0>0$ such that 
	\begin{equation}\label{R_0}
		\left\{ u \in X \ | \ \Psi(u) \leq 6 M_0 \right\} \subset B(0, R_0 )
	\end{equation}
and we set 
	\begin{equation}\label{Aset}
		A \equiv \overline{B(0,R_0 + 1)}, \quad 
		\cA_{c_k(\lambda), 3\sigma} \equiv 
		A \cap \left\{ c_k(\lambda) - 3 \sigma \leq I_\lambda \leq c_k(\lambda) + 3 \sigma \right\}.
	\end{equation}
To complete the proof of Theorem \ref{amps}, it suffices to show that 
$\cA_{c_k(\lambda) , 3 \sigma}$ contains a Palais-Smale sequence of $I_\lambda$ for every $\sigma \in (0,1)$. 
For this purpose, we argue by contradiction and suppose that there is no Palais-Smale sequence 
in $\cA_{c_k(\lambda), 3\sigma}$ for some $\sigma \in (0,1)$. 
Since $\cA_{c_k(\lambda),3\sigma}$ shrinks as $\sigma$ decreases, without loss of generality we may assume that 
	\begin{equation}\label{i:sigma-clam}
		4 \sigma < \alpha_0 = \inf_{\| u \| = \rho_0} I_{ 1 - \overline{\e}_1 } (u) \leq c_k(\lambda). 
	\end{equation}
Since $\cA_{c_k(\lambda), 3\sigma}$ does not contain any Palais-Smale sequence of $I_\lambda$, 
we may find $\delta_0 \in (0, \sigma)$ such that for each $u \in \cA_{c_k(\lambda), 3 \sigma}$ 
there exists $v=v(u) \in X \setminus \{u\}$ such that 
	\[
		\lambda \left( \Psi(v) - \Psi(u) \right) - \Phi'(u) (v-u) < - 5 \delta_0 \| v - u\|. 
	\]
To derive a contradiction, we shall find a suitable relatively compact set $K \subset \cA_{c_k(\lambda), 3 \sigma}$ and exploit Lemmas \ref{l:defor-lem} 
and \ref{sydefolemma}. 
Since $c_k'(\lambda)$ exists, there exists $h_1>0$ such that 
	\begin{equation}\label{e:c'}
	\begin{array}{lcl}
	c_k(\lambda + h) & < & \disp c_k(\lambda) + \left( c_k'(\lambda) + 1 \right) h = c_k(\lambda) + (M_0-2) h \\[0.3cm]
	& < & \disp c_k(\lambda) + M_0h
		\qquad \text{for any $h \in (0,h_1)$}.
	\end{array} 
	\end{equation}
In addition, we may also suppose 
	\begin{equation}\label{e:h1}
		5 M_0 h_1 < \delta_0 < \sigma.
	\end{equation}
For $h \in (0,h_1)$, consider the following subclass $\Gamma_k^h$ of $\Gamma_k$
	\[
		\Gamma_k^h \equiv \left\{ \gamma \in \Gamma_k \ \Big| \ \sup_{\zeta \in \DD^k} I_{\lambda + h} \left( \gamma (\zeta) \right) \leq c_k(\lambda) + M_0 h \right\}.
	\]
Notice that \eqref{e:c'} implies $\Gamma_k^h \neq \emptyset$ for each $h \in (0,h_1)$. 
For each $h \in (0,h_1)$ and $\gamma \in \Gamma_k^h$, if $\zeta \in \DD^k$ satisfies 
	\[
		I_\lambda (\gamma (\zeta)) \geq c_k(\lambda) - 5M_0 h,
	\]
 then from $\gamma \in \Gamma_k^h$ we deduce 
	\[
		\Psi(\gamma(\zeta)) = \frac{ I_{\lambda + h} (\gamma (\zeta)) - I_\lambda (\gamma(\zeta)) }{h} 
		\leq 
		\frac{c_k(\lambda) + M_0 h - (c_k(\lambda) - 5M_0 h ) }{h} = 6M_0.
	\]
By the choice of $R_0$ in \eqref{R_0}, this means that 
for any $h \in (0,h_1)$ and $\gamma \in \Gamma_k^h$, the set 
	\[
		K_\gamma \equiv \left\{ \gamma(\zeta)  \  | \  I_\lambda (\gamma (\zeta)) \geq c_k (\lambda) - 5M_0 h \right\}
	\]
satisfies 
	\[
		K_\gamma \subset  B(0,R_0) \subset A. 
	\]
Furthermore, since $K_\gamma \subset \gamma (\DD^k)$, $K_\gamma$ is relatively compact in $X$ and 
\eqref{e:h1} yields $K_\gamma \subset \cA_{c_k(\lambda), 3\sigma}$. 
Moreover, since we are considering the case $k \ge 1$ (thus, we are assuming \eqref{even} and \eqref{smp}), $K_\gamma$ is symmetric.

Next, we choose any $h \in (0,h_1)$ and $\gamma_1 \in \Gamma_k^h$, and write 
	\[
		C_1 \equiv \gamma_1 (\DD^k), \quad 
		K_1 \equiv K_{\gamma_1} = \Set{ \gamma_1(\zeta)  \  | \  I_\lambda (\gamma_1 (\zeta)) \geq c_k(\lambda) - 5M_0 h } 
		\subset B(0,R_0) \cap \cA_{ c_k(\lambda), \sigma}  .
	\]
By  Lemma \ref{sydefolemma},  there exist $ \wt{\tau} > 0$, $\eta \in C([0,\wt{\tau}] \times X, X)$ and open (symmetric) sets $W , \wt{W}$ in $X$ such that 
	\begin{enumerate}
		\item[(a)] $\overline{K_1} \subset W \subset \overline{W} \subset \wt{W}$;
		\item[(b)] $\| \eta(s,w) - w\| \leq s,$ $ \eta(s,-w) =-\eta(s,w)$ for every $(s,w) \in [0,\wt{\tau}] \times X$
		and $\eta(s,w) = w$ for all $s \in [0,\wt{\tau}]$ if $w \not \in \wt{W}$; 
	\item[(c)] $I_\lambda (\eta(s,w)) \leq I_\lambda (w) + \delta_0 s$ for each $(s,w) \in [0,\wt{\tau}] \times X$; 
		\item[(d)] $I_\lambda (\eta(s,w)) \leq I_\lambda (w) - 2 \delta_0 s$ for any $(s,w) \in [0,\wt{\tau}] \times W$ with $I_\lambda(w) \geq c_k(\lambda) - \sigma$. 
	\end{enumerate}
	
We briefly explain the issues to overcome and the strategy. First, flowing $\gamma_1$ by $\eta$ may not lead to a map in $\Gamma_k$, as $\eta$ may not preserve the boundary condition on $\partial \DD^k$. For this reason, we have to modify the map $\wt{\gamma_2} \equiv \eta(\tau_1,\gamma_1)$ in a neighborhood of $\partial \DD^k$ to get an element 
$\gamma_2 \in \Gamma_k$ (perhaps, not in $\Gamma_k^h$) satisfying
\[
\sup_{\zeta \in \DD^k} I_\lambda( \gamma_2(\zeta)) \le c_k(\lambda) + M_0h - 2\delta_0 \tau_1 
\quad \text{for some $\tau_1 \geq \wt{\tau} > 0$}. 
\]
A second issue is that the time interval $\tau_1$ might be too small to obtain the following contradiction: 
\[
\sup_{\zeta \in \DD^k} I_\lambda( \gamma_2(\zeta)) < c_k(\lambda).
\] 
To remedy, we first choose $\tau_1$ close to the maximal time for which a deformation with properties (a)--(d) exists (for fixed $K_1,\delta_0$). Next, we proceed inductively to construct Cauchy sequences of maps $\{\gamma_j\} \subset \Gamma_k$ and of almost maximal times $\{\tau_j\} \subset (0,\infty)$ with the property
\[
\sup_{\zeta \in \DD^k} I_\lambda( \gamma_j(\zeta)) \le c_k(\lambda) + M_0h - 2\delta_0 \sum_{i=1}^j \tau_i.
\]
We shall prove that the relatively compact sets
\[
K_j \equiv \Set{ \gamma_j(\zeta)  \  | \  I_\lambda (\gamma_j (\zeta)) \geq c (\lambda) - 5M_0 h + \delta_0 \sum_{i=1}^{j-1} \tau_i }
\]
of points at which $I_\lambda\circ \gamma_j$ takes higher values 
are all contained in $\cA_{c_k(\lambda),3\sigma}$, and 
that their union $K_\infty \equiv \bigcup_{j=1}^\infty K_j \subset \cA_{c_k(\lambda),3\sigma}$ is relatively compact, too. 
By choosing a deformation $\eta_\infty$ in Lemma \ref{sydefolemma} with $K = K_\infty$ we are then able to reach our desired contradiction. 
Here is the point where $\tau_j$ being ``almost maximal" plays a role.

\medskip	
	\noindent
	\textbf{Step 1:} \emph{construction of $\gamma_2$, and its properties}. 
	
\medskip

For the given relatively compact symmetric set $K_1$, we may consider
	\[
		\cD_1 \equiv \left\{ \left( \eta, \tau, W , \wt{W} \right) \  \Big| \  \text{$\eta, W, \wt{W}$ satisfy (a)--(d)  for $s \in [0,\tau]$}
		\right\}.
	\]
We define 
	\[
		\wt{T}_1 \equiv \sup_{ \left( \eta , \tau , W, \wt{W} \right) \in \cD_1 } \tau \in (0,\infty] \quad \text{ and } \quad 
		T_1 \equiv \min \left\{ \frac{2M_0 h}{\delta_0} , \ \wt{T}_1 \right\}.
	\]
Remark that $\cD_1 \neq \emptyset$, and that \eqref{e:h1} with $h \in (0,h_1)$ gives $T_1 \leq 2M_0 h / \delta_0 \leq 2/5$. 
We choose $(\eta_1,\tau_1, W_1 , \wt{W}_1) \in \cD_1$ so that 
	\[
		\frac{3}{4} T_1 < \tau_1 < T_1
	\]
and define $\wt{\gamma}_2 \in C(\DD^k , X)$  by 
	\[
		\wt{\gamma}_2(\zeta) := \eta_1 (\tau_1 , \gamma_1(\zeta) ). 
	\]
Since $\wt{\gamma}_2$ may not coincide with $\pi_{0,k}$ on $\partial \DD^k$, in order to obtain a map in $\Gamma_k$ we define, for $\theta \in (0,1/4)$, the dilation 
	\[
	L_\theta \ \ : \ \ 	\DD^k(1-\theta) \to \DD^k, \qquad L_\theta (\zeta) \equiv \frac{\zeta}{1-\theta}.
	\]
Notice that 
	\[
	\max_{\zeta \in \DD^k(1-\theta)} \left|L_\theta (\zeta) - \zeta \right| = \frac{\theta}{1-\theta}. 
	\]
When $k=0$, we consider $L_\theta: [\theta,1-\theta] \to [0,1]$. 
By the uniform continuity of $\gamma_1$, we may find $\theta_1 \in (0,1/4)$ such that 
	\begin{equation}\label{i:theta1}
		\max_{|\zeta| \le 1-\theta_1} \| \gamma_1(\zeta) - \gamma_1(L_{\theta_1} (\zeta)) \| 
		+ \max_{ 1-\theta_1 \le |\zeta| \le 1}\| \gamma_1(\zeta) - \gamma_1(\zeta/|\zeta|) \|  \leq \frac{\tau_1}{2}. 
	\end{equation}
Then we define $\gamma_2 (\zeta)$ by 
	\[
		\gamma_2(\zeta) \equiv 
		\begin{dcases}
		      \eta_1\left(\tau_1,\gamma_1 ( L_{\theta_1}( \zeta) )\right) & \quad \text{if} \ |\zeta| \le 1-\theta_1, \\[0.3cm]
			\eta_1 \left( \tau_1  \frac{1- |\zeta|}{\theta_1}, \ \gamma_1 \qty( \frac{\zeta}{\abs{\zeta}} )  \right) 
			& \quad \text{if} \ 1 - \theta_1 \leq |\zeta| \leq 1.
		\end{dcases}
	\]
It is easily seen that $\gamma_2 \in \Gamma_k$ and 
\begin{align*}
		& \| \gamma_2(\zeta) - \gamma_1(\zeta) \|  \leq \\[0.3cm]
& \begin{dcases}
	  	\| \eta_1(\tau_1 , \gamma_1 (L_{\theta_1} (\zeta)) ) - \gamma_1(L_{\theta_1} (\zeta)) \| 
			+ \| \gamma_1 (L_{\theta_1} (\zeta)) - \gamma_1 (\zeta) \| & \text{if} \ |\zeta| \le 1-\theta_1, \\[0.3cm]
			\norm{\eta_1 \left( \tau_1  \frac{1- |\zeta|}{\theta_1}, \ \gamma_1 \qty( \frac{\zeta}{\abs{\zeta}} ) \right) - \gamma_1 \qty( \frac{\zeta}{\abs{\zeta}} ) }
			+ \norm{ \gamma_1 \qty( \frac{\zeta}{\abs{\zeta}} ) - \gamma_1(\zeta) }
			& \text{if} \ 1 - \theta_1 \leq |\zeta| \leq 1.
		\end{dcases} \end{align*}
Therefore, property (b) and \eqref{i:theta1} yield 
	\begin{equation}\label{diff:ga1ga2}
		\max_{\zeta \in \DD^k} \| \gamma_1(\zeta) - \gamma_2(\zeta) \| \leq \frac{3}{2} \tau_1. 
	\end{equation}
It follows from property (c), \eqref{smp}, \eqref{i:sigma-clam} and \eqref{e:h1} that if $\zeta  \in \DD^k \setminus \DD^k( 1 - \theta_1)$, then 
	\begin{equation}\label{est:g2-1}
		I_{\lambda} (\gamma_2(\zeta)) \leq \delta_0 \tau_1 < \delta_0 < \sigma <  c_k(\lambda) - 3\sigma < c_k(\lambda) - 5 M_0 h  . 
	\end{equation}
Similarly, when $\zeta \in \DD^k( 1 - \theta_1) \setminus  (\gamma_1 \circ L_{\theta_1})^{-1} ( K_{1} )$, we have 
$I_\lambda( \gamma_1(L_{\theta_1} (\zeta))) < c_k(\lambda) - 5M_0 h$ by the definition of $K_1$, and 
	\begin{equation}\label{est:g2-2}
		I_\lambda (\gamma_2(\zeta)) \leq I_\lambda ( \gamma_1( L_{\theta_1} (\zeta) ) ) + \delta_0 \tau_1 
		< c_k(\lambda) + \delta_0 \tau_1 - 5M_0 h. 
	\end{equation}
On the other hand, if $\zeta \in (\gamma_1 \circ L_{\theta_1})^{-1} (K_1) $, then $\gamma_1(L_{\theta_1} ( \zeta )) \in K_1$, and 
property (d) with the fact $\gamma_1 \in \Gamma_k^h$ implies 
	\[
		I_\lambda (\gamma_2(\zeta)) \leq I_\lambda ( \gamma_1(L_{\theta_1} (\zeta)) ) - 2\delta_0 \tau_1 
		\leq c_k(\lambda) + M_0 h - 2\delta_0 \tau_1. 
	\]
By the  properties $3T_1 /4 < \tau_1 < T_1 \leq 2M_0h/\delta_0$, we see that 
	\[
		\delta_0 \tau_1 - 5M_0 h \leq M_0 h - 2 \delta_0 \tau_1
	\]
and 
	\begin{equation}\label{est:g2Ilam}
		c_k(\lambda) \leq \sup_{\zeta \in \DD^k} I_{\lambda} (\gamma_2 (\zeta)) \leq c_k(\lambda) + M_0 h - 2 \delta_0 \tau_1.
	\end{equation}


\medskip	
	\noindent
	\textbf{Step 2:} \emph{construction of the sequence $\{\gamma_j\}_{j=1}^\infty$}. 
	
\medskip

	Next, put 
	\[
		C_2 \equiv \gamma_2(\DD^k), \quad 
		K_{2} \equiv \left\{ \gamma_2(\zeta)  \ |  \ I_\lambda (\gamma_2(\zeta)) \geq c_k(\lambda) + \delta_0 \tau_1 - 5M_0 h \right\}. 
	\]
By \eqref{diff:ga1ga2}, 
	\begin{equation*}
		\dist_H \left( C_1 ,C_2\right) \leq \frac{3}{2} \tau_1, 
	\end{equation*}
where $\dist_H$ stands for the Hausdorff distance between $C_1$ and $C_2$ defined by 
	\[
		\dist_H(A,B) \equiv \max \left\{ \sup_{x \in A} \dist (x, B), \ \sup_{y \in B} \dist (A,y)   \right\}.
	\]
Moreover, from \eqref{est:g2-1} and \eqref{est:g2-2} we infer that 
	\[
		\gamma_2^{-1} (K_2) \subset (\gamma_1 \circ L_{\theta_1})^{-1} (K_1) \subset \DD^k(1-\theta_1).
	\]
Hence, if $\zeta \in \gamma_2^{-1} (K_2)$, then property (b) yields 
	\[
		\| \gamma_2 (\zeta) - \gamma_1(L_{\theta_1} (\zeta)) \| 
		= \| \eta_1( \tau_1 , \gamma_1 (L_{\theta_1} (\zeta) ) ) - \gamma_1 (L_{\theta_1} (\zeta)) \| \leq \tau_1,
	\]
and the fact $\gamma_1(L_{\theta_1} (\zeta)) \in K_1 \subset B(0,R_0)$ gives 
	\[
		K_2 \subset B \left( 0 , \,  R_0 + \tau_1 \right).
	\]
Combining \eqref{est:g2Ilam} with $ \tau_1 < T_1 \leq 2/5$, 
we observe that $K_2 \subset \cA_{c_k(\lambda), 3\sigma}$ is relatively compact. 
Thus, Lemma \ref{sydefolemma} ensures the existence of a deformation $(\eta, \wt{\tau},W_2,\wt{W}_2)$ satisfying (a)--(d) with $K_2$ replacing $K_1$. 
Then, we consider
	\[
		\cD_2 \equiv \left\{ \left( \eta, \tau, W , \wt{W} \right) \  \Big| \ 
		\text{$(\eta, \tau, W, \wt{W})$ satisfies (a)--(d) for $s \in [0,\tau]$ and $K_2$ instead of $K_1$} \right\}
	\]
and define 
	\[
		\wt{T}_2 \equiv \sup_{ \left( \eta, \tau, W , \wt{W} \right) \in \cD_2  }  \tau, \quad 
		T_2 \equiv \min \left\{ \frac{2M_0h}{\delta_0}  - \tau_1 , \, \wt{T}_2 \right\}. 
	\]
Select $(\eta_2,\tau_2,W_2, \wt{W}_2 ) \in \cD_2$ and $\theta_2 \in (0,1/4)$ so that 
	\[
		\frac{3}{4} T_2 < \tau_2 < T_2
	\]
and 
	\[
			\max_{|\zeta| \le 1-\theta_2} \| \gamma_2(\zeta) - \gamma_2(L_{\theta_2} (\zeta)) \| 
		+ \max_{1-\theta_2 \le |\zeta| \le 1}\| \gamma_2(\zeta) - \gamma_2(\zeta/\zeta|) \|  \leq \frac{\tau_2}{2}.
	\]
Define $\gamma_3 \in \Gamma_k$ by 
	\[
		\gamma_3(\zeta) \equiv 
		\begin{dcases}
		      \eta_2\left(\tau_2,\gamma_2 ( L_{\theta_2}( \zeta) )\right) & \quad \text{if} \ |\zeta| \le 1-\theta_2, \\[0.3cm]
			\eta_2 \left( \tau_2  \frac{1- |\zeta|}{\theta_2}, \ \gamma_2 \qty( \frac{\zeta}{\abs{\zeta}} ) \right) 
			& \quad \text{if} \ 1 - \theta_2 \leq |\zeta| \leq 1.
		\end{dcases}
	\]
Arguing as before, we may prove 
	\[
		\max_{\zeta \in \DD^k} \| \gamma_2(\zeta) - \gamma_3(\zeta) \| \leq \frac{3}{2} \tau_2,
	\]
and 
	\[
		I_\lambda (\gamma_3(\zeta)) 
			\begin{dcases}
				\leq 
				c_k(\lambda) + M_0 h - 2\delta_0(\tau_1 + \tau_2) 
				&\text{if} \ \zeta \in (\gamma_2 \circ L_{\theta_2})^{-1} (K_2), 
				\\[0.2cm]
				< c_k(\lambda) + \delta_0(\tau_1 + \tau_2) - 5M_0h 
				&\text{otherwise}. 
			\end{dcases}
	\]
Since $\tau_1 + \tau_2 < 2M_0h/\delta_0$, we see that
	\[
	\begin{array}{l}
	\delta_0(\tau_1+\tau_2) - 5M_0 h < M_0 h - 2\delta_0 (\tau_1 + \tau_2), \\[0.3cm]
	c_k (\lambda) \le \sup_{\zeta \in \DD^k} I_\lambda \left( \gamma_3(\zeta) \right) \leq c_k (\lambda) + M_0 h - 2\delta_0 (\tau_1 + \tau_2).
	\end{array}
	\]
By putting 
	\[
		C_3 \equiv \gamma_3 (\DD^k), \quad K_3 \equiv 
		\left\{ \gamma_3(\zeta) \ | \ I_\lambda (\gamma_3(\zeta)) \geq c_k(\lambda) + \delta_0(\tau_1 + \tau_2) - 5M_0h   \right\},
	\]
we see from the same argument as before that 
	\[
	\begin{array}{l}
	\disp \dist_H (C_2,C_3) \leq \frac{3}{2} \tau_2, \qquad 
		\gamma_3^{-1} (K_3) \subset (\gamma_2 \circ L_{\theta_2})^{-1} (K_2) \subset \DD^k(1-\theta_2) , \\[0.4cm]
	K_3 \subset B \left( 0 , R_0 +  \tau_1 + \tau_2 \right). 
	\end{array}
	\]
Thus, $K_3 \subset \cA_{c_k(\lambda), 3\sigma}$ is relatively compact. We may therefore repeat the same procedure inductively to obtain 
$\{K_j\}_{j=1}^\infty$, $\{\tau_j\}_{j=1}^\infty \subset (0,\infty)$, $(\eta_j,\tau_j,W_j,\wt{W}_j) \in \cD_j$, 
$\{\theta_j\}_{j=1}^\infty \subset (0,1/4)$, $\{L_{\theta_j}\}_{j=1}^\infty$, $\{\gamma_j\}_{j=1}^\infty \subset \Gamma_k$ 
such that 
	\begin{equation}\label{prop:n}
		\begin{aligned}
			& C_{j} \equiv \gamma_{j} ( \DD^k ), \quad \dist_H(C_{j-1}, C_{j}) \leq \frac{3\tau_{j-1}}{2}, \\[0.3cm] 
			& K_{j} \equiv \left\{ \gamma_{j}(\zeta) \  \Big| \  I_{\lambda} (\gamma_{j}(\zeta)) \geq c_k(\lambda) - 5M_0 h  + \delta_0 \sum_{i=1}^{j-1} \tau_i \right\} 
			\subset B \left( 0 , R_0 + \sum_{i=1}^{j-1} \tau_i \right), 
			\\
			&
			\wt{T_j} \equiv \sup_{ (\eta, \tau, W, \wt{W}) \in \cD_j } \tau, 
			\quad T_j \equiv \min \left\{ \frac{2M_0h}{\delta_0} - \sum_{i=1}^{j-1} \tau_i, \, \wt{T}_j \right\}, \quad 
			\frac{3}{4} T_j < \tau_j < T_j, 
			\\
			&
			\sup_{\zeta \in \DD^k} I_\lambda (\gamma_{j+1}(\zeta)) \leq c_k(\lambda) + M_0h - 2\delta_0 \sum_{i=1}^{j} \tau_i, \quad \sum_{i=1}^j \tau_i \leq \frac{M_0h}{2\delta_0} < \frac{2}{5},
		\end{aligned}
	\end{equation}
where 
	\[
		\gamma_{j+1}(\zeta) \equiv 
		\begin{dcases}
		      \eta_j\left(\tau_j,\gamma_j ( L_{\theta_j}( \zeta) )\right) & \quad \text{if} \ |\zeta| \le 1-\theta_j, \\[0.3cm]
			\eta_j \left( \tau_j  \frac{1- |\zeta|}{\theta_j}, \ \gamma_j \qty( \frac{\zeta}{\abs{\zeta}} ) \right) 
			& \quad \text{if} \ 1 - \theta_j \leq |\zeta| \leq 1,
		\end{dcases}
		\quad 
		\max_{\zeta \in \DD^k} \norm{ \gamma_j(\zeta) - \gamma_{j+1}(\zeta) } \leq \frac{3\tau_j}{2}.
	\]

\medskip	
	\noindent
	\textbf{Step 3:} \emph{a compactness argument, and conclusion}. 
	
\medskip
	
Consider
	\[
		\scrS \equiv \left\{ K \subset X \ | \ K \neq \emptyset, \ \text{$K$ is compact} \right\}.
	\]
It is known that $(\scrS, \dist_H)$ is a complete metric space (we refer e.g. to \cite[Theorem 10.1.6]{Is81}).
Since $\sum_{i=1}^\infty \tau_i < \infty$ due to \eqref{prop:n}, we have 
	\[
		\dist_H(C_j, C_{j+\ell}) \leq \sum_{i=1}^\ell \dist_H ( C_{j+i-1} , C_{j+i} ) \leq \sum_{i=1}^{\ell} \frac{3}{2} \tau_{j+i-1};
	\]
thus $\{C_j\}_{j=1}^\infty$ is a Cauchy sequence in $(\scrS,\dist_H)$. Hence there exists $C_\infty \in \scrS$ such that 
	\[
		\dist_H(C_j, C_\infty) \to 0 \text{ as } j \to \infty. 
	\]
Consider 
	\[
		K_\infty \equiv \bigcup_{j=1}^\infty K_j.
	\]
We next claim that $K_\infty$ is relatively compact.  Indeed, let $\{u_p\}_{p=1}^\infty \subset K_\infty$ be a sequence and choose 
 $\{j_p\}_{p=1}^\infty$ satisfying $u_{p} \in K_{j_p}$. 
Without loss of generality, we may assume $j_p \to \infty$ as $p \to \infty$. 
By recalling that $K_j \subset C_j$, $\dist_H(C_j, C_\infty) \to 0$ and $C_\infty$ is compact, $\lim_{p \to \infty}\| u_{p} - w_p \| = 0$  for some $\{w_p\}_{p=1}^\infty \subset C_\infty$. 
From the compactness of $C_\infty$, up to a subsequence $w_p \to w_\infty \in C_\infty$ as $p \to \infty,$
whence $u_p \to w_\infty$ as $p \to \infty.$ This proves that $K_\infty$ is relatively compact.

	Since $K_j \subset \cA_{c_k(\lambda), 3 \sigma}$ for every $j \geq 1$, 
$K_\infty \subset \cA_{c_k(\lambda), 3\sigma}$ holds and Lemma \ref{sydefolemma} can be applied to obtain 
$\tau_\infty > 0$, $\eta_\infty \in C([0,\tau_\infty] \times X, X)$ and open sets $W_\infty$, $\wt{W}_\infty$ such that 
$\overline{K_\infty} \subset W_\infty \subset \overline{W_\infty} \subset \wt{W_\infty}$ and (a)--(d) hold with $K_\infty$ instead of $K_1$. 
Since $\sum_{i=1}^\infty \tau_i < \infty$, there are two alternatives: 
	\begin{enumerate}
		\item[(A)] $\displaystyle \sum_{i=1}^\infty \tau_i = \frac{2M_0h}{\delta_0} <\frac{2}{5}$;
		\item[(B)] $\displaystyle \sum_{i=1}^{\infty} \tau_i < \frac{2M_0h}{\delta_0}$. 
	\end{enumerate}
In case (A), choose a large $j$ such that $M_0 h/ \delta_0 < \sum_{i=1}^j \tau_i$. Then 
	\[
		c_k(\lambda) \leq \sup_{\zeta \in \DD^k} I_\lambda (\gamma_{j+1} (\zeta)) \leq c_k(\lambda) + M_0 h - 2 \delta_0 \sum_{i=1}^j \tau_i < c_k(\lambda) - M_0 h,
	\]
which is a contradiction. 
In case (B),  it follows from the definition of  $\wt{T}_j$ and the fact $K_j \subset K_\infty$ that $ \wt{T}_j \ge \tau_\infty$. 
Since $\sum_{i=1}^{\infty} \tau_i < \frac{2M_0h}{\delta_0}$, it follows that 
\[
T_j = \min \left\{ \frac{2M_0h}{\delta_0} - \sum_{i=1}^{j-1} \tau_i, \, \wt{T}_j \right\} \ge 
 \min \left\{ \frac{2M_0h}{\delta_0} - \sum_{i=1}^{\infty} \tau_i, \,  \tau_\infty \right\} > 0.
\] 
Since $\tau_j \ge \frac 34 T_j$, this contradicts the fact that $\sum_{i=1}^{\infty} \tau_i  \le \frac 25 < \infty.$
Thus, in either case, we get a contradiction and $\cA_{c_k(\lambda), 3\sigma}$ contains a Palais-Smale sequence. 
This completes the proof of Theorem \ref{amps}. 
\end{proof}


We are ready to prove Theorem \ref{mpt} and part of Theorem \ref{smt}, as a direct consequence of the next

\begin{thm}
Assume that $(\Psi_1)$, $(\Psi_2)$, $(\Psi_3),$ $(\Phi_1)$, $(\Phi_2)$ and \textup{(IB)}  hold. 
Then, $c_0(1)$ is a critical value of $I=I_1$  if $\eqref{mp}$ holds,
and $\{c_k(1)\}_{k=1}^\infty $ are critical values  of $I$ if  \eqref{even}  and  \eqref{smp} hold.
	\end{thm}	
\begin{proof}
We only prove that $\{c_k(1)\}_{k=1}^\infty $ are critical values of $I$, the other case being analogous. 
For $k \ge 1,$ $c_k(\lambda)$ is nondecreasing with respect to $\lambda \in  (1 - \ov{\e} , 1 + \ov{\e}].$
Thus, there exists a sequence $\{\lambda_l\}_l \subset (1, 1 + \ov{\e})$ such that $\lim_{l \to \infty}\lambda_l =1$ and $c^\prime_k(\lambda_l)$ exists for each $l=1,2,\ldots.$
By Theorem \ref{amps}, there exists a bounded Palais-Smale sequence $\{u_{l,j}\}_j$ for $I_{\lambda_l}$ with $\lim_{j \to \infty}I_{\lambda_l}(u_{l,j}) =c_k(\lambda_l).$
Then, from Proposition \ref{IBtoIC}, up to a subsequence $u_{l,j}$ converges strongly to some $u_l$ in $X$ as $j \to \infty$, and $u_l$ is a critical point of $I_{\lambda_l}$ with $I_{\lambda_l}(u_l) = c_k(\lambda_l)$. 
By Proposition \ref{monot1} and $\lambda_l > 1$, we have $c_k(\lambda_l) \to c_k(1)$ as $l \to \infty$. 
Furthermore, by using $({\rm IB})$, $\{u_l\} \subset \cK_{[1,1+\ov{\e}]}^{c_k(1+\ov{\e})+1}$ is bounded in $X$. 
Therefore, applying again Proposition \ref{IBtoIC}, we conclude that 
$u_l \to u$ strongly in $X$, for some $u$ which is critical for $I$ and satisfies $I(u) = c_k(1)$. This completes the proof.
\end{proof}

The proof of Theorem \ref{smt} will be completed if we prove the following divergence property of $c_k=c_k(1).$
\begin{thm}\label{t:diverge}
Assume that $(\Psi_1)$, $(\Psi_2)$, $(\Psi_3),$ $(\Phi_1)$, $(\Phi_2)$, \textup{(IB)},  \eqref{even}  and  \eqref{smp} hold. 
Then, 
\[
\lim_{k\to\infty} c_k(1) = \infty.
\]  
\end{thm}



\begin{proof}
To prove Theorem \ref{t:diverge}, we combine the arguments in \cite[Chapter 9]{Ra86} (cf. \cite{HIT10}) and \cite{S86} 
with the iteration scheme as in the proof of Theorem \ref{amps}. 
To this end, we introduce 
		\[
			\begin{aligned}
				\cE_{\Rn} &\equiv \left\{ E \subset \Rn \ | \ \text{$E$ is closed, symmetric and $0 \not \in E$} \right\},
				\\
				\cE &\equiv \left\{ F \subset X \ | \ \text{$F$ is closed, symmetric and $0 \not \in F$} \right\},
				\\
				\genus (E) &\equiv \text{Krasnoselski's genus of $E$}. 
			\end{aligned}
		\]
For $k \geq 1$ and $\lambda \in [ 1 - \overline{\e}_1, 1 ]$, we set 
	\[
	\begin{aligned}
		\Lambda_k 
		&\equiv \left\{ 
		\gamma \left( \overline{\DD^n \setminus \Sigma} \right) 
		\ \big| \ n \geq k, \ \gamma \in \Gamma_n, \ \Sigma \in \cE_{\Rn}, \ \Sigma \cap \partial \DD^n = \emptyset, \   \genus(\Sigma) \leq n - k
		\right\},
		\\[0.2cm]
		d_k(\lambda) & \equiv 
		\inf_{ B \in \Lambda_k } \sup_{ u \in B} I_\lambda.
	\end{aligned}
	\]
By $\genus (\emptyset) = 0$, $\gamma (\DD^k) \in \Lambda_k$ holds for all $\gamma \in \Gamma_{k}$. 
Moreover, since $\Lambda_{k+1} \subset \Lambda_k$ and $\Psi \ge 0$, 
for every $k \geq 1$ and $\lambda, \lambda_1, \lambda_2 \in [1 - \overline{\e}_1, 1]$, 
	\[
	d_k(\lambda) \leq c_{k}(\lambda), \quad 
	d_k(\lambda) \leq d_{k+1} (\lambda) \quad \textup{ and }  \quad  d_k(\lambda_1) \leq d_k(\lambda_2) \quad \text{if} \ \lambda_1 \leq \lambda_2.	\]
Therefore, it is enough to prove 
	\begin{equation}\label{e:d_k-div}
		d_k (1) \to \infty \quad \text{as $k \to \infty$}.
	\end{equation}


Following  \cite{Ra86} and  \cite{S86} (see also \cite{HIT10}), we prove the next result. 
	
	\begin{lemma}\label{p:prop-d_k}  The following properties hold:
		\begin{enumerate}[{\rm (i)}]
			\item
			if $C \in \Lambda_{k+s}$ and $Z \in \cE$ satisfies $\genus (Z) \leq s$ and $I_1|_{Z} > 0$,
			then $\ov{C \setminus Z} \in \Lambda_k$;
			\item
			for every $k \in \N$ and $C \in \Lambda_k$, it holds $C \cap \partial B(0,\rho_0) \neq \emptyset$, hence 
			\[
			0< \alpha_0 = \inf_{ \| u \| = \rho_0 } I_{1 - \overline{\e}_1} (u) 
			\leq d_k(1-\overline{\e}_1) \leq d_k(\lambda)  \quad \text{for each $\lambda \in [1- \overline{\e}_1 , 1]$}.
			\]
		\end{enumerate}
	\end{lemma}

	\begin{proof}
\noindent \textbf{(i)} 
Write $C = \gamma ( \ov{\DD^n \setminus \Sigma} ) \in \Lambda_{k+s}$ where $n \geq k+s$, $\gamma \in \Gamma_{n}$, 
$\Sigma \cap \partial \DD^n = \emptyset$, $\Sigma \in \cE_{\R^n}$ and $\genus (\Sigma) \leq n-(k+s)$. 
Let $Z \in \cE$ satisfy $\genus (Z) \leq s$ and $I_1|_{Z} > 0$. Remark that $\gamma^{-1} (Z) \in \cE_{\R^n}$. 
We claim  that 
		\begin{equation}\label{e:AZ1}
			\ov{C \setminus Z} = \gamma \left( \ov{ \DD^n \setminus \left( \Sigma \cup \gamma^{-1} (Z) \right) } \right).
		\end{equation}
		Notice that 	
		\[
		\begin{array}{lcl}
		\ov{C \setminus Z} & \supset & \disp 
		C \setminus Z \supset \gamma ( \DD^n \setminus \Sigma ) \setminus Z \\[0.3cm]
		& = & \disp \gamma \left( (\DD^n \setminus \Sigma) \setminus \gamma^{-1} (Z) \right) 
		= \gamma \left( \DD^n \setminus \left(  \Sigma \cup \gamma^{-1} (Z) \right) \right).
		\end{array}
		\]
Hence, it holds that $\ov{C \setminus Z} \supset \gamma ( \ov{\DD^n \setminus ( \Sigma \cup \gamma^{-1} (Z) )} )$. 
On the other hand, since $\gamma^{-1} (Z)$ is closed, we have 
		\[
		\ov{\DD^n \setminus \Sigma} \setminus \gamma^{-1} (Z) \subset \ov{\DD^n \setminus ( \Sigma \cup \gamma^{-1} (Z) )}, 
		\]
which implies 
		\[
		C \setminus Z = \gamma \left( \ov{\DD^n \setminus \Sigma} \setminus \gamma^{-1} (Z) \right) 
		\subset \gamma \left(  \ov{\DD^n \setminus ( \Sigma \cup \gamma^{-1} (Z) )}  \right).
		\]
From the compactness of $ \ov{\DD^n \setminus ( \Sigma \cup \gamma^{-1} (Z) )},$ it follows that 
$\ov{C \setminus Z} \subset \gamma (  \ov{\DD^n \setminus ( \Sigma \cup \gamma^{-1} (Z) )} )$. 
Thus, \eqref{e:AZ1} holds.

	The fact $I_{1} |_{Z} > 0 > I_{1}|_{ \pi_{0,n} (\partial \DD^n) }$ 
gives $\partial \DD^n \cap ( \Sigma \cup \gamma^{-1} (Z) )  = \emptyset$. By \eqref{e:AZ1} and the estimate 
		\[
		\genus \left( \Sigma \cup \gamma^{-1} (Z) \right) \leq \genus (\Sigma) + \genus \left( \gamma^{-1} (Z) \right) 
		\leq \genus (\Sigma) + \genus (Z) \leq n- k,
		\]
we conclude $\ov{C \setminus Z} \in \Lambda_{k}$.

		\noindent \textbf{(ii)} 
Write $C= \gamma ( \ov{\DD^n \setminus \Sigma} ) \in \Lambda_{k}$ where $n \geq k$, $\gamma \in \Gamma_{n}$, 
$\Sigma \cap \partial \DD^n = \emptyset$ and 
		$\genus (\Sigma) \leq n -k $. Set 
		\[
		\wh{U} \equiv \left\{ \zeta \in \Int \DD^n \ | \ \gamma(\zeta) \in B(0,\rho_0) \right\}.
		\]
It is easily seen that $0 \in \wh{U}$, $-\wh{U} = \wh{U}$ and $\wh{U} \subset \Rn$ is open. 
Denote by $U$ the connected component of $\wh{U}$ containing $0$. 
Since $U$ is a symmetric neighborhood of $0$, $\genus (\partial U) = n$ holds. 
Since $\|\gamma(\zeta) \| = \| \pi_{0,n} (\zeta) \| > \rho_0$ for each $\zeta \in \partial \DD^n$, we see that
		\begin{equation}\label{i:pOB1}
			\gamma (\partial U) \subset \partial B(0,\rho_0).
		\end{equation}

		Define $W \equiv \left\{ \zeta \in \DD^n \ | \ \gamma(\zeta) \in \partial B(0,\rho_0) \right\} \in \cE_{\R^{n}} $. 
Thanks to \eqref{i:pOB1}, we see that $\partial U \subset W$ and $n = \genus (\partial U) \leq \genus (W)$. 
On the other hand, letting $\eps$ small enough so that $\DD^n(\eps) \cap W = \emptyset$, by the properties of Krasnoselki's genus, $\genus(W) \le \genus(\overline{\DD^n \backslash \DD^n(\eps)}) = \genus(\partial \DD^n) = n$. Summarizing, 
		\[
		\genus (W) = n, \quad 
		\genus \left( \gamma \left(  \ov{W \setminus \Sigma} \right) \right) \geq \genus \left( \ov{W \setminus \Sigma} \right) 
		\geq \genus \left( W \right) - \genus \left( \Sigma \right) \geq k > 0.
		\]
In particular, $ \gamma ( \ov{W \setminus \Sigma} ) \neq \emptyset$. 
Since $\gamma ( \ov{W \setminus \Sigma} ) \subset C \cap \partial B(0,\rho_0) $, 
we have $\emptyset \neq C \cap \partial B(0,\rho_0)$. 
	\end{proof}

	To prove \eqref{e:d_k-div}, we argue by contradiction and suppose that $\{d_k(1)\}_k$ is bounded. 
From the monotonicity of $d_k (\lambda)$ in $k$ and $\lambda$, it follows that 
	\[
	d_\infty(\lambda) \equiv \lim_{k \to \infty} d_k(\lambda) \in [\alpha_0 , \infty) \quad 
	\text{for every $\lambda \in [1 - \overline{\e}_1 , 1]$}. 
	\]
Notice that $\lambda \mapsto d_\infty (\lambda)$ is nondecreasing on $[1 - \overline{\e}_1 , 1]$. 

\medskip 	
	\noindent
	\textbf{Claim 1:} \emph{There exists $\wt{\lambda}_0 \in \left[ 1 - \overline{\e}_1, \ 1 - \frac{\overline{\e}_1}{2} \right]$ such that 
	\[  d_\infty' ( \wt{\lambda}_0 )  \ \textup{ exists and } \ 
		d_\infty' ( \wt{\lambda}_0 ) \leq 3 \overline{\e}_1^{-1} \nu_0, \quad 
		\text{where} \quad \nu_0 \equiv d_\infty( 1 ) - d_\infty( 1 - \overline{\e}_1 ) \in [0,\infty).
	\]
} 
\begin{proof}
Since $d_\infty$ is nondecreasing, we obtain 
	\[
	\int_{1 - \overline{\e}_1}^{1} d_\infty'(\lambda) \dd{\lambda}
	\leq d_\infty (1) - d_\infty (1 - \overline{\e}_1) = \nu_0. 
	\]
Write 
	\[
	J_0 \equiv \left\{ \lambda \in [1 - \overline{\e}_1, 1] \ | \ d_\infty'(\lambda) \ \textup{ exists and } \ d_\infty'(\lambda) \leq 3 \overline{\e}_1^{-1} \nu_0 \right\}. 
	\]
When $\nu_0 > 0$, since $0 \leq d_\infty'(\lambda)$ a.e.  in $[1-\overline{\e}_1,1]$ and 
	\[
	\nu_0 \geq \int_{ [1 - \overline{\e}_1,1] \setminus J_0 } d_\infty' (\lambda) \dd{\lambda} 
	\geq 3 \overline{\e}_1^{-1}  \nu_0 \left| [1 - \overline{\e}_1 , 1] \setminus J_0 \right|, 
	\]
we obtain the following estimation
	\begin{equation}\label{estJ0}
	|J_0| \geq \frac{2}{3} \overline{\e}_1. 
	\end{equation}
Since \eqref{estJ0} also holds when $\nu_0 = 0$, Claim 1 directly follows.
\end{proof}	

	Choose $h_0 \in (0,  \overline{\e}_1/2)$ so that for any $h \in [0,h_0]$, 
	\begin{equation}\label{ie:h_0}
		d_\infty( \wt{\lambda}_0 + h ) \leq d_\infty( \wt{\lambda}_0 ) + \left( d_\infty'(\wt{\lambda}_0) + 1 \right) h 
		\leq d_\infty ( \wt{\lambda}_0 ) + \left( 3 \overline{\e}_1^{-1} \nu_0 + 1 \right) h .
	\end{equation}
From Proposition \ref{IBtoIC} and $I_\lambda(0) = 0$, the set 
	\[
	\begin{aligned}
		\cK_0 \equiv 
		\left\{ u \in X \ \Big| \ \text{$u$ is a critical point of $I_{\wt{\lambda}_0}$}, \ \frac{\alpha_0}{2} \leq I_{\wt{\lambda}_0} \leq d_\infty(1) + 1 \right\} 
	\end{aligned}
	\]
is compact and symmetric with $0 \not \in \cK_0$, hence $\cK_0 \in \cE$. 
Let 
	\[
	s_\infty \equiv \genus \left( \cK_0 \right) \in \N.
	\]
Since $\cK_0$ is compact and $I_{ \wt{\lambda}_0 }$ is lower semicontinuous, 
we may find $r_\infty \in (0, 1 )$ so that 
	\begin{equation}\label{low-est}
	\ov{\cK_0 +B(0,3r_\infty)} \in \cE, \quad 
	s_\infty = \genus \left( \ov{\cK_0 + B(0,3r_\infty)} \right), \quad 
	\frac{\alpha_0}{4} < I_{\wt{\lambda_0}} |_{ \cK_0 + \overline{B(0,3r_\infty)} } 
	\leq  I_1|_{ \cK_0 + \overline{B(0,3r_\infty)} } .
	\end{equation}
	According to ($\Psi_2$), for each $M>0$ there exists $R(M) > 0$ so that 
	\[
	\left\{ u \in X \ | \ \Psi(u) \leq M \right\} \subset B(0, R(M) ). 
	\]
	In what follows, we shall write 
	\[
	\begin{aligned}
		O_{r_\infty} &\equiv \cK_0 + B(0,r_\infty), \quad O_{3r_\infty} \equiv  \cK_0 + B(0,3r_\infty), \quad M_0  \equiv \overline{\e}_1^{-1} \nu_0 + 1, \\
		A 
		&\equiv 
		\ov{ B \left( 0, R(12M_0) +  3r_\infty \right) } \setminus O_{r_\infty}, \quad 
		\cA \equiv A \cap \left\{ \frac{\alpha_0}{2} \leq I_{\wt{\lambda}_0} \leq d_\infty(1) + 1 \right\}.
	\end{aligned}
	\]
Notice that $A \subset X$ is bounded, closed and symmetric. 
Furthermore, from Proposition \ref{IBtoIC} and the definition of $O_{r_\infty}$,
there is no Palais-Smale sequence in $\cA$. Therefore, 
there exists $\delta_0 \in (0,1)$ such that for every $u \in  \cA$ we may find $v=v( u ) \in X \setminus \{u\}$ such that 
	\begin{equation}\label{derivative}
		\wt{\lambda}_0 \left( \Psi(v) - \Psi(u) \right) - \Phi'(u) (v-u) < - 5 \delta_0 \| v - u \|.
	\end{equation}
We next choose $h_1 > 0$ so that 
	\begin{equation}\label{e:h_1}
		0 < h_1 \leq h_0 < \frac{\overline{\e}_1}{2} , 
		\quad 12 \delta_0^{-1} M_0 h_1  < \min \left\{ r_\infty, \,  \frac{\alpha_0}{6}, \, \frac{1}{3} \right\} \equiv \sigma 
	\end{equation}
and then, we can find $k_\infty \in \N$ such that 
	\begin{equation}\label{e:k_infity}
		d_\infty( \wt{\lambda}_0 + h_1 ) - d_k( \wt{\lambda}_0 + h_1 ) \leq h_1, \quad 
		d_\infty( \wt{\lambda}_0 ) - d_k (\wt{\lambda}_0) \leq h_1  \qquad \textup{for } \ \ k \geq k_\infty.
	\end{equation}

\medskip 	
	\noindent
	\textbf{Claim 2:} \emph{For every $ \wt{\lambda}_0 \leq \lambda_1 \leq \lambda_2 \leq \wt{\lambda}_0 + h_1$ and 
		$k_\infty \leq k_1, k_2$, 
		\[
		d_{k_2} (\lambda_2) - d_{k_1} (\lambda_1) \leq 3h_1 M_0.
		\]
}
	\begin{proof}
		By \eqref{ie:h_0}, \eqref{e:h_1} and \eqref{e:k_infity} with the monotonicity of $d_k(\lambda)$ on $k$ and $\lambda$, 
		\[
		\begin{aligned}
			d_{k_2} (\lambda_2) - d_{k_1} (\lambda_1) 
			&\leq d_{k_2} (\wt{\lambda}_0 + h_1) - d_{k_1} (\wt{\lambda}_0) 
			\leq d_\infty (\wt{\lambda}_0 + h_1) - d_{k_\infty} (\wt{\lambda}_0) 
			\\
			&\leq d_\infty ( \wt{\lambda}_0 + h_1 ) - d_\infty (\wt{\lambda}_0) + h_1 			
			\leq \left( 3 \overline{\e}_1^{-1} \nu_0 + 1 \right) h_1 + h_1 
			\\
			& \leq 3 h_1 M_0.  \qedhere
		\end{aligned}
		\]
	\end{proof}

	Recalling \eqref{derivative}, we shall apply Lemma \ref{sydefolemma} with
	\[
		\begin{aligned}
			\sigma = \min \left\{  r_\infty ,  \frac{\alpha_0}{6}, \frac{1}{3} \right\}, \quad 
			\cA_{ d_{k_\infty} (\wt{\lambda}_0 ) ,3\sigma} \equiv A \cap 
			\left\{ d_{k_\infty} (\wt{\lambda}_0) - 3 \sigma \leq I_{\wt{\lambda}_0} \leq d_{k_\infty} (\wt{\lambda}_0) + 3\sigma \right\} 
			\subset \cA. 
		\end{aligned}
	\]
Therefore, 
for every relatively compact symmetric set $K \subset \cA_{d_{k_\infty}(\wt{\lambda}_0),3\sigma}$, we may construct an odd deformation satisfying (i)--(iii) in Lemma \ref{sydefolemma}. To find suitable sets $K$, in what follows we put 
	\[
		\mu \equiv \wt{\lambda}_0 + h_1. 
	\]
For the given integers $k_\infty$ in \eqref{e:k_infity} and $s_\infty$ in \eqref{low-est},  choose $C \in \Lambda_{k_\infty + s_\infty}$ so that 
	\[
	\sup_{C} I_\mu \leq d_{k_\infty + s_\infty} (\mu) + h_1. 
	\]
Consequently, Claim 2, \eqref{e:h_1} and the inequalities $\delta_0 < 1$, $r_\infty < 1$, $M_0 \geq 1$ guarantee that
	\begin{equation}\label{ie:B}
		\sup_{C} I_\mu \leq d_{k_\infty} (\wt{\lambda}_0) + (3M_0 + 1) h_1 \leq d_{k_\infty} (\wt{\lambda}_0) + 4M_0 h_1 
		< d_{k_\infty} (\wt{\lambda}_0) + \sigma. 
	\end{equation}
Moreover, if $u \in C$ satisfies
	\[
		d_{k_\infty} (\wt{\lambda}_0) - \sigma <  d_{k_\infty} (\wt{\lambda}_0) - 8 M_0 h_1 \leq I_{\wt{\lambda}_0} (u),
	\]
then \eqref{ie:B}  yields 
	\begin{equation}\label{bdd-Psi}
		\begin{aligned}
			\Psi(u) = 
			\frac{ I_\mu (u) - I_{\wt{\lambda}_0} (u) }{\mu - \wt{\lambda}_0} 
			&\leq 
			\frac{d_{k_\infty} (\wt{\lambda_0}) + 4M_0 h_1 - \left( d_{k_\infty} (\wt{\lambda}_0) - 8M_0 h_1 \right)  }{h_1}
			= 12M_0. 
		\end{aligned}
	\end{equation}
From \eqref{bdd-Psi} it follows that  
	\begin{equation}\label{e:B}
		C \cap \left\{ I_{\wt{\lambda}_0}  \ge d_{k_\infty} (\wt{\lambda}_0) - 8M_0 h_1  \right\}
		\subset \left\{ u \in X \ | \ \Psi(u) \leq 12M_0 \right\} 
		\subset B(0, R(12M_0) ) .
	\end{equation}
	Set 
	\[
	C_0 \equiv \ov{C \setminus \ov{O_{3r_\infty}}}. 
	\]
We infer from \eqref{low-est}, \eqref{ie:B}, \eqref{e:B} and Lemma \ref{p:prop-d_k} that 
	\[
	C_0 \in \Lambda_{k_\infty}, \quad 
	K_0 \equiv C_0 \cap \left\{ d_{k_\infty} (\wt{\lambda}_0) - 8M_0 h_1 \leq I_{\wt{\lambda}_0}   \right\}
	\subset \cA_{ d_{k_\infty} (\wt{\lambda}_0 ) , 3\sigma }. 
	\]
Furthermore, 
	\[
	\di_X \left( C_0 , O_{r_\infty} \right) \geq 2 r_\infty.
	\]
Since $C_0 \in \Lambda_{k_\infty},$ we can write $C_0$ as 
	\[
		C_0 = \gamma_0 \left( \overline{\DD^n \setminus \Sigma_0 } \right), \quad 
		n \geq k_\infty, \quad \gamma_0 \in \Gamma_{n}, \quad \Sigma_0 \in \cE_{\R^{n}}, \quad 
		\Sigma_0 \cap \partial \DD^n = \emptyset, \quad 
		\genus \left( \Sigma_0 \right) \leq n - k_\infty. 
	\]
Then $K_0$ is rewritten as 
	\[
		K_0 = \left\{ \gamma_0(\zeta) \ \big| \ \zeta \in \overline{\DD^n \setminus \Sigma_0}, \ 
		 \ I_{ \wt{\lambda}_0 } (\gamma_0(\zeta)) \geq d_{k_\infty} (\wt{\lambda}_0) - 8M_0 h_1  \right\}. 
	\]
From now on, we will repeatedly use the following inequality: 
	\begin{equation}\label{basic-1}
		-8 M_0 h_1 + \delta_0 T < 4M_0 h_1 - 2\delta_0 T \quad \text{for all $T \in [ 0, 4 \delta_0^{-1} M_0 h_1 )$}. 
	\end{equation}
Since $K_0$ is relatively compact and symmetric, we may apply Lemma \ref{sydefolemma} and 
find $s_0>0$, $\eta \in C([0,s_0] \times X, X)$ and symmetric open sets $W,\wt{W}$ such that 
	\begin{enumerate}
		\item[(a)] $\overline{K}_0 \subset W \subset \overline{W} \subset \wt{W}$. 
		\item [(b)]$\| \eta(s,w) - w\| \leq s$ and $\eta(s,-w) = - \eta(s,w)$ for every $(s,w) \in [0,s_0] \times X$. 
		\item [(c)]$I_{\wt{\lambda}_0} (\eta(s,w)) \leq I_{\wt{\lambda}_0} (w) + \delta_0s$ 
		for every $(s,w) \in [0,s_0] \times X$. 
		\item[(d)] $I_{\wt{\lambda}_0} (\eta(s,w)) \leq I_{\wt{\lambda}_0} (w) - 2 \delta_0 s$ 
		for every $(s,w) \in [0,s_0] \times W$ with $I_{\wt{\lambda}_0} (w) \geq d_{k_\infty} (\wt{\lambda}_0) - \sigma$. 
	\end{enumerate}
As in the proof of Theorem \ref{amps}, we consider a family of deformations of $K_0$: 
	\[
		\cD_0 \equiv \left\{ \left( \eta, \tau , W ,  \wt{W} \right) \ \Big| \ 
		\text{$\eta ,  W , \wt{W}$ satisfy (a)--(d) for  $s \in [0,\tau]$} \right\}
	\]
and set 
	\[
		\wt{T}_0 \equiv \sup_{ (\eta , \tau , W, \wt{W} ) \in \cD_0 } \tau \ge s_0,\quad 
		T_0 \equiv \min \left\{ \frac{3M_0h_1}{\delta_0}, \, \wt{T}_0 \right\} .
	\]
Choose $(\eta_0,\tau_0,W_0,\wt{W}_0) \in \cD_0$ so that 
	\[
		\frac{3}{4} T_0 < \tau_0 < T_0 \leq 3 \delta_0^{-1} M_0 h_1 < \frac{r_\infty}{4},
	\] 
where the last inequality comes from \eqref{e:h_1}. Since for small $\theta_{0}  > 0$ the dilation map
	\[
	L_{\theta_0} \ \ : \ \ 	\DD^n(1-\theta_0) \to \DD^n, \qquad L_{\theta_0} (\zeta) \equiv \frac{\zeta}{1-\theta_0}
	\]	
satisfies 
	\[
		\max_{ |\zeta| \leq 1 - \theta_0 } \| \gamma_0(\zeta) - \gamma_0( L_{\theta_0} (\zeta) ) \| + 
		\max_{ 1 - \theta_0 \leq |\zeta| \leq 1 } \| \gamma_0 (\zeta) - \gamma_0( \zeta/|\zeta| ) \| < \frac{\tau_0}{2},
	\]
we define 
	\[
		\gamma_1 (\zeta) \equiv \begin{dcases}
			\eta_0 \left( \tau_0 , \gamma_0 \left( L_{\theta_{0}} (\zeta) \right) \right) & \quad \text{if} \ |\zeta| \leq 1 - \theta_{0},
			\\[0.2cm]
			\eta_0 \left(  \tau_0 \frac{1-|\zeta|}{\theta_0}, \ \gamma_0 \left( \frac{\zeta}{\abs{\zeta}}  \right)  \right) 
			& \quad \text{if} \ 1 - \theta_{0} \leq |\zeta| \leq 1.
		\end{dcases}
	\]
Then, $\gamma_1 \in \Gamma_{n}.$ Since $\theta_{0}>0$ is small and $L_{\theta_{0}}$ is an odd homeomorphism, 
	\[
		\Sigma_1 \equiv L_{\theta_{0}}^{-1} ( \Sigma_0 )  \in \cE_{\R^{n}}, \quad \Sigma_1 \cap \partial \DD^n = \emptyset, \quad 
		\genus ( \Sigma_1 ) \leq n - k_\infty,
	\]
and moreover
	\[
	\max_{\zeta \in \DD^n} \| \gamma_1(\zeta) - \gamma_0 (\zeta) \| \leq \frac{3}{2} \tau_0. 
	\]

Define  
	\[
		C_1 \equiv \gamma_1 \left( \overline{\DD^n \setminus \Sigma_1} \right) \in \Lambda_{k_\infty}.
	\]
As in the proofs of Theorem \ref{amps}, 
it follows from \eqref{basic-1} and \eqref{ie:B} that 
	\[
		d_{k_\infty} (\wt{\lambda}_0) \leq \sup_{C_1} I_{\wt{\lambda}_0} 
		\leq d_{k_\infty} (\wt{\lambda}_0) + 4M_0 h_1 - 2\delta_0 \tau_0;
	\]this implies that \[\tau_0 \leq 2\delta_0^{-1} M_0 h_1< \frac{r_\infty}{6}. \]
We next claim 
	\begin{equation}\label{dist:C0C1}
		\dist_H (C_0, C_1) \leq \tau_0. 
	\end{equation}
Let $C_1 \ni u_1 = \gamma_1(\zeta)$, where $\zeta \in \ov{\DD^n \setminus \Sigma_1}$. If $1 - \theta_0 \leq |\zeta| \leq 1$, then by 
	\[
		\gamma_1 \left( \zeta/|\zeta| \right) = \gamma_0 \left( \zeta/|\zeta| \right) 
		= \pi_{0,n} \left( \zeta/|\zeta| \right) \in C_0 \cap C_1,
	\]
the choice of $\theta_0$ gives $\| u_1 - \gamma_0(\zeta/|\zeta|) \| \leq \tau_0$. 
When $|\zeta| \leq 1 - \theta_0$, we find 
	\[
		L_{\theta_0} (\zeta) \in \overline{\DD^n \setminus \Sigma_0}, \quad \gamma_0 (L_{\theta_0} (\zeta)) \in C_0, 
		\quad \| u_1 - \gamma_0(L_{\theta_0} (\zeta) ) \| \leq \tau_0.
	\]
Thus, $\di_X(C_0,u_1) \leq \tau_0$ and since $u_1 \in C_1$ is arbitrary, $\max_{u_1 \in C_1} \di_X(C_0,u_1) \leq \tau_0$.

	On the other hand, let $C_0 \ni u_0 = \gamma_0(\zeta)$ where $\zeta \in \overline{\DD^n \setminus \Sigma_0}$. 
From the property  
\[
L_{\theta_0}^{-1} (\zeta) \in \overline{ \DD^n(1-\theta_0)  \setminus \Sigma_1 }, 
\]
it follows that 
	\[
		\gamma_1 \left( L_{\theta_0}^{-1} (\zeta) \right) \in C_1, \quad 
		\left\| u_0 - \gamma_1 ( L_{\theta_0^{-1}} (\zeta) ) \right\| 
		= 
		\left\| \gamma_0(\zeta) - \eta_0 \left( \tau_0, \gamma_0(\zeta) \right) \right\| \leq \tau_0.
	\]
Hence, $\di_X(u_0, C_1) \leq \tau_0$ and $\max_{u_0 \in C_0} \di_X(u_0,C_1) \leq \tau_0$, 
which yields \eqref{dist:C0C1}.

	Furthermore, \eqref{dist:C0C1} and $\di_X(C_0,O_{r_\infty}) \geq 2 r_\infty$ imply the inequality 
	\[
		\di_X(C_1 , O_{r_\infty}) \geq 2r_\infty - \tau_0. 
	\]
We next consider the set 
	\[
			K_1 \equiv 
			\left\{ \gamma_1(\zeta) \ \big| \  \zeta \in \ov{\DD^n \setminus \Sigma_1}, 
			\ I_{\wt{\lambda}_0} (\gamma_1(\zeta)) \geq d_{k_\infty} (\wt{\lambda}_0) - 8M_0 h_1 + \delta_0 \tau_0  \right\}.
	\]
Remark that $K_1$ is symmetric and relatively compact. 
Moreover, by combining (c) with (ii) in Lemma \ref{p:prop-d_k} and with $\delta_0 \in (0,1)$, for $1-\theta_0 \le |\zeta| \le 1$ we have 
	\begin{align*}
	I_{\wt{\lambda}_0}(\gamma_1(\zeta)) & \le I_{\wt{\lambda}_0}(\gamma_0(\zeta/|\zeta|)) + \delta_0 \tau_0 < \delta_0 \tau_0 \\[0.2cm]
	& \le \alpha_0 - 8 M_0h_1 + \delta_0 \tau_0 \le d_{k_\infty}(\wt{\lambda}_0) - 8 M_0h_1 + \delta_0 \tau_0.
	\end{align*} 	
Hence, we obtain
	\[
		\gamma_1^{-1} (K_1) \subset (\gamma_0 \circ L_{\theta_0})^{-1} (K_0), \quad 
		\max_{ \zeta \in \DD^n(1-\theta_0) \cap \overline{\DD^n \setminus \Sigma_1} } \| \gamma_1(\zeta) - \gamma_0(L_{\theta_0} (\zeta) ) \| \leq \tau_0.
	\]
Since $K_0 \subset B(0,R(12M_0))$ holds due to \eqref{e:B}, 
	\[
		K_1 \subset B(0, R(12M_0) + \tau_0 ), \quad \di_X \left( K_1, O_{r_\infty} \right) \geq 2r_\infty - \tau_0, \quad 
		K_1 \subset \cA_{ d_{k_\infty} (\wt{\lambda}_0), 3 \sigma }  .
	\]
By Lemma \ref{sydefolemma}, there exist $s_1 > 0$, an odd map $\eta \in C([0,s_1] \times X, X)$ and symmetric open sets $W,\wt{W}$ of $K_1$ satisfying (a)--(d) with $K_1$ instead of $K_0$. 
Thus, define  
	\[
		\cD_1 \equiv \left\{ \left( \eta, \tau , W ,  \wt{W} \right) \ \Bigg| \ \text{ $\eta , W , \wt{W}$ satisfy (a)--(d)  for $s \in [0,\tau]$} \right\} 
\]
and 
	\[
		\wt{T}_1 \equiv \sup_{ (\eta, \tau, W, \wt{W}) \in \cD_1 } \tau  \ge s_1, \quad 
		T_1 \equiv \min \left\{ \frac{3M_0 h_1}{\delta_0} - \tau_0, \  \wt{T}_1  \right\} . 
	\]
Select $(\eta_1,\tau_1,W_1,\wt{W}_1) \in \cD_1$ and $0 < \theta_1 \ll 1$ such that 
	\[
		\frac{3}{4}T_1 < \tau_1 < T_1 \leq 3 \frac{M_0 h_1}{\delta_0} - \tau_0 < \frac{r_\infty}{4} - \tau_0
	\]
and 
	\[
		\max_{|\zeta| \leq 1 - \theta_1} \| \gamma_1(\zeta) - \gamma_1 ( L_{\theta_1} (\zeta) ) \| 
		+ \max_{1 - \theta_1 \leq |\zeta | \leq 1} \| \gamma_1(\zeta) - \gamma_1(\zeta/ |\zeta|) \|   < \frac{\tau_1}{2}, 
	\]
where 
	\[
		L_{\theta_1} (\zeta) \equiv \frac{\zeta}{1-\theta_1} : \DD^n(1-\theta_1) \to \DD^n. 
	\]
Define $\gamma_2$ by  
	\[
		\Gamma_n \ni \gamma_2 (\zeta) \equiv \begin{dcases}
			\eta_1 \left( \tau_1 , \gamma_1 \left( L_{\theta_1} (\zeta) \right) \right) & \quad \text{if} \ |\zeta| \leq 1 - \theta_1, \\[0.2cm]
			\eta_1 \left( \tau_1 \frac{1-|\zeta|}{\theta_1}, \ \gamma_1 \left( \zeta/|\zeta| \right) \right) 
			& \quad \text{if} \ 1 - \theta_1 \leq |\zeta| \leq 1.
		\end{dcases}
	\]
From  $\genus (\Sigma_1) \leq n - k_\infty$ it follows that 
	\[
		\Sigma_2 \equiv L_{\theta_1}^{-1} (\Sigma_1) \in \cE_{\R^n}, \quad 
		\Sigma_2 \cap \partial \DD^n = \emptyset, \quad \genus (\Sigma_2) \leq n - k_\infty. 
	\]
Writing 
	\[
		C_2 \equiv \gamma_2 \left( \overline{\DD^n \setminus \Sigma_2} \right) \in \Lambda_{k_\infty}, 
	\]
we may prove 
	\[
		d_{k_\infty} (\wt{\lambda}_0) \leq \sup_{C_2} I_{\wt{\lambda}_0} 
		\leq d_{k_\infty} (\wt{\lambda}_0) + 4M_0 h_1 - 2\delta_0 (\tau_0 + \tau_1),
	\]
which yields 
	\[
		\tau_0 + \tau_1 \leq \frac{2M_0h_1}{\delta_0}  < \frac{r_\infty}{6}. 
	\]
Furthermore, by an argument analogous to the one we just employed for $C_0$ and $C_1,$  we see that 
	\[
		\dist_H (C_1,C_2) \leq \tau_1 , \quad \di_X (C_2 , O_{r_\infty}) \geq 2 r_\infty - \left( \tau_0 + \tau_1 \right). 
	\]
We next define 
	\[
		K_2 \equiv 
		\left\{ \gamma_2(\zeta) \ \big| \ 
		\zeta \in \ov{\DD^n \setminus \Sigma_2}, \ I_{\wt{\lambda}_0} (\gamma_2(\zeta)) \geq d_{k_\infty} (\wt{\lambda}_0) - 8M_0 h_1 + \delta_0 (\tau_0 + \tau_1) \right\}
	\]
and notice that $K_2$ is symmetric, relatively compact, $\gamma_2^{-1} (K_2) \subset (\gamma_1 \circ L_{\theta_1})^{-1} (K_1)$ and 
	\[
		\begin{aligned}
			& \di_X(K_2,O_{r_\infty}) \geq 2r_\infty - \tau_0 -\tau_1, \quad 
			K_2 \subset B( 0, R(12M_0) + \tau_0 + \tau_1 ), \quad K_2 \subset \cA_{ d_{k_\infty} (\wt{\lambda}_0) ,3\sigma}.
		\end{aligned}
	\]
Inductively, we obtain $\{(\eta_j,\tau_j,W_j,\wt{W}_j)\}_{j=0}^\infty$, $\{C_j\}_{j=0}^\infty$, 
$\{ \Sigma_j\}_{j=0}^\infty$, $\{\gamma_j\}_{j=0}^\infty$, $\{K_j\}_{j=0}^\infty$ 
and a sequence $\{\cD_j\}_{j=0}^\infty$ of deformations of $K_j$ such that 
	\[
		\begin{aligned}
			&C_j = \gamma_j \qty( \DD^n \setminus \Sigma_j ) \in \Lambda_{k_\infty}, \quad 
			\genus \qty( \Sigma_j ) \leq n - k_\infty, \quad 
			d_{k_\infty} (\wt{\lambda}_0) \leq \sup_{C_j} I_{\wt{\lambda}_0} \leq d_{k_\infty} (\wt{\lambda}_0) +4M_0 h_1 -2 \delta_0 \sum_{\ell =0}^j \tau_\ell, \\
			&\dist_H (C_j, C_{j+1}) \leq \tau_j, \qquad \di_X( C_j , O_{r_\infty} ) \geq 2r_\infty - \sum_{\ell =0}^{j-1} \tau_\ell, \qquad \sum_{\ell =0}^{j} \tau_\ell  \leq \frac{2M_0h_1}{\delta_0} < \frac{r_\infty}{6}, \\
			& \wt{T}_j \equiv \sup_{ (\eta , \tau , W, \wt{W}) \in \cD_j } \tau, \qquad T_j \equiv \min \left\{ \wt{T}_j, \ \frac{3M_0h_1}{\delta_0} - \sum_{\ell =0}^{j-1} \tau_\ell \right\}, \qquad  \frac{3}{4}T_j \le \tau_j < T_j,
			\\
			& K_j \equiv \Set{ \gamma_j(\zeta) |
			\zeta \in \overline{\DD^n \setminus \Sigma_j}, \ 
			I_{\wt{\lambda}_0} ( \gamma_j(\zeta) ) \geq d_{k_\infty} (\wt{\lambda}_0) - 8M_0 h_1 + \delta_0 \sum_{\ell =0}^{j-1} \tau_\ell 
			 }, \\
			& K_j \subset B \left( 0, R(12M_0) + \sum_{\ell =0}^{j-1} \tau_\ell \right).
		\end{aligned}
	\]
In particular, 
	\[
		K_j \subset \cA_{ d_{k_\infty} (\wt{\lambda}_0) , 3 \sigma } \quad \text{for each $j \geq 0$}. 
	\]
Recalling that $(\scrS, \dist_{H})$ is a complete metric space and the above properties, we see that 
$\{C_j\}_{j=0}^\infty$ is a Cauchy sequence in $(\scrS, \dist_H)$. 
Therefore, there exists $C_\infty \in \scrS$ such that $\dist_{H} (C_j, C_\infty) \to 0$ as $j \to \infty$.

	As before, a set defined by 
	\[
		K_\infty \equiv \bigcup_{j=0}^\infty K_j \subset \cA_{ d_{k_\infty} (\wt{\lambda}_0) , 3 \sigma }
	\]
is relatively compact and Lemma \ref{sydefolemma} with $K \equiv  K_\infty$ guarantees the existence of  deformation $(\eta_\infty, \tau_\infty, W_\infty , \wt{W}_\infty)$ of $K_\infty$; then
$(\eta_\infty, \tau_\infty, W_\infty, \wt{W}_\infty) \in \cD_j$ for each $j$ large enough. Thus, 
	\[
		\tau_j \ge \frac{3}{4}\min\left\{ \tau_\infty, \frac{3M_0h_1}{\delta_0} - \sum_{\ell =0}^{\infty} \tau_\ell \right\} \ge \frac{3}{4}\min\left\{ \tau_\infty, \frac{M_0h_1}{\delta_0}\right\} \qquad \text{for large enough $j$},
	\]
which contradicts 
	\[
		\sum_{\ell=0}^\infty \tau_\ell \leq \frac{r_\infty}{6}. 
	\]
Hence, we complete the proof. 
\end{proof}

		\section*{Acknowledgements}
		J.B. was supported by the National Research Foundation of Korea(NRF) grant funded by the Korea government(MSIT)(No. NRF-2023R1A2C1005734).
		N.I. was supported by JSPS KAKENHI Grant Numbers JP 19H01797, 19K03590 and  24K06802. 
		A.M. has been supported by the project 
		``Geometric problems with loss of compactness'' 
		from Scuola Normale Superiore, and by the PRIN Project 2022AKNSE4.
		L. Mari is supported by the PRIN project no. 20225J97H5 ``Differential-geometric aspects of manifolds via Global Analysis"
	
	\vspace{0.4cm}
	
	\noindent \textbf{Conflict of Interest.} The authors have no conflict of interest.

%
%
%
%
%
%

\appendix

\section{The principle of symmetric criticality}
\label{s:PSC}

The principle of nonsmooth symmetric criticality due to Kobayashi \& \^Otani \cite{koba_ota} considers a vast range of linear actions of a topological groups $G$ on a Banach space $X$. In this appendix, we provide a simpler proof of the principle in a setting which applies to our main theorems. 
Compared to \cite[Theorem 3.16]{koba_ota}, notice that Proposition \ref{prop:sym-critical} does not require $X$ to be reflexive.

Let $X$ be a Banach space and $G$ a compact topological group acting on $X$ linearly.  
Assume that the map $G \times X \ni (g,u) \mapsto g\cdot u \in X$ is continuous, and consider the closed subspace 
	\[
		X_G \equiv \Set{ u \in X | g \cdot u = u \ \text{for all $g \in G$} }.
	\]
We consider two functionals $\Psi$ and $\Phi$ satisfying 
	\begin{enumerate}[(i)]
		\item $\Psi : X \to (-\infty, \infty]$ is convex and lower semicontinuous with $D(\Psi) \neq \emptyset$. 
		\item $\Phi \in C^1(X,\R)$. 
		\item $\Psi$ and $\Phi$ are $G$-invariant, that is $\Psi(g \cdot u) = \Psi(u)$ and $\Phi(g \cdot u) = \Phi(u)$ 
		for every $g \in G$ and $u \in X$. 
	\end{enumerate}
Finally, set $I(u) \equiv \Psi(u) - \Phi(u) : X \to (-\infty, \infty]$.

	\begin{prop}\label{prop:sym-critical}
		Assume $u \in X_G$ is a critical point of $I |_{X_G} : X_G \to (-\infty, \infty] $. 
		Then $u$ is a critical point of $I : X \to (-\infty,\infty]$. 
	\end{prop}

	\begin{proof}
Let $u \in X_G$ be a critical point of $I |_{X_G}$. Then 
	\begin{equation}\label{u-cri-FixG}
		\Psi(v) - \Psi(u) - \Phi'(u) (v-u) \geq 0 \qquad \text{for each $v \in X_G$}.
	\end{equation}
Fix any $w \in X$ and let $\mu$ be a normalized Haar measure on $G$ ($\mu(G) = 1$). 
Since $(g,u) \mapsto g \cdot u$ is continuous and $G$ is compact, we set (cf. Rudin \cite[Theorems 3.20 and 3.27, and Definition 3.26]{Ru91})
	\begin{equation}\label{ave-w}
		\hat w \equiv  \int_G g \cdot w \dd{\mu},
	\end{equation}
that is, $\hat w$ is the unique element satisfying 
	\[
		\Lambda (\hat w) = \int_G \Lambda (g \cdot w) \dd{\mu} \qquad \text{for every $\Lambda \in X^\ast$}.
	\]
We first show $\hat w \in X_G$. To this end, let $h \in G$ and write $T_h : X \ni u \mapsto h \cdot u \in X$. 
For each $\Lambda \in X^\ast$, it follows from $\Lambda \circ T_h \in X^\ast$ and the left invariance of $\mu$ that 
	\begin{align*}
	\Lambda (h \cdot \hat w) &= (\Lambda \circ T_h) (\hat w) 
		= \int_G \left( \Lambda \circ T_h \right) (g \cdot w) \dd{\mu} \\
		&= \int_G \Lambda \left( hg \cdot w \right) \dd{\mu} 
		= \int_G \Lambda (g \cdot w ) \dd{\mu}  = \Lambda (\hat w). 
	\end{align*}
Since $\Lambda \in X^\ast$ is arbitrary, $h \cdot \hat w = \hat w$ holds for each $h \in G$.

	We now prove that $u$ is a critical point of $I : X \to (-\infty,\infty]$. For $w \in X$, consider $\hat w$ in \eqref{ave-w}. 
From \cite[Theorem 3.27]{Ru91}, $\hat w$ belongs to the closure of the convex hull of the orbit $G \cdot w$. Therefore, we can take convex combinations
	\[
		w_j \equiv \sum_{k=1}^{L_j} c_{j,k}(g_{j,k} \cdot w), \quad c_{j,k} \geq 0, \quad \sum_{k=1}^{L_j} c_{j,k} = 1
	\]
such that $\| w_j - \hat w\|\to 0$. Since $\Phi \in C^1(X,\R)$ satisfies $\Phi(g\cdot u ) = \Phi(u)$ for every $g \in G$, 
it follows that 
	\[
		\Phi'(g\cdot u) (g \cdot v) = \Phi'(u) v \qquad \text{for any $g \in G$ and $v \in X$}.
	\]
By the convexity and the lower semicontinuity of $\Psi$, we infer from \eqref{u-cri-FixG} and $\hat w,u \in X_G$ (hence, $g_{j,k}\cdot  u = u$) that 
	\[
		\begin{aligned}
			0 \leq \Psi(\hat w) - \Psi(u) - \Phi'(u) (\hat w -u) 
			& \leq \liminf_{j \to \infty} \Psi(w_j) - \Psi(u) - \lim_{j \to \infty} \Phi'(u) (w_j - u)
			\\
			& \leq \liminf_{j \to \infty} \left[ \Psi(w_j) - \Psi(u) - \Phi'(u) (w_j - u) \right] 
			\\
			& \leq \liminf_{j \to \infty} 
			\sum_{k=1}^{L_j} c_{j,k} \left[ \Psi( g_{j,k} \cdot w ) - \Psi(u) - \Phi'(u) ( g_{j,k} \cdot w - u ) \right] 
			\\
			&= \liminf_{j \to \infty}
			\sum_{k=1}^{L_j} c_{j,k} \left[ \Psi( w ) - \Psi(u) - \Phi'( g_{j,k} \cdot u)(g_{j,k} \cdot w) + \Phi'(u) u ) \right] 
			\\
			&= \liminf_{j \to \infty} 
			\sum_{k=1}^{L_j} c_{j,k} \left[ \Psi( w ) - \Psi(u) - \Phi'( u) ( w - u ) \right] 
			\\
			&= \Psi(w) - \Psi(u) - \Phi'(u) (w - u). 
		\end{aligned}
	\]
This means that $u$ is a critical point of $I: X \to (-\infty,\infty]$ and the proof is completed. 
	\end{proof}

\section{Regularity for the Born-Infeld solution} \label{appe_1}

In this section, we prove Proposition \ref{prop_criticalpoints_1}.


\begin{proof} We first address the equivalence in item (i). 

\smallskip 
\noindent
\textbf{(a) $\Rightarrow$ (b).}

\smallskip 
Assume that $u$ is a weak solution to \eqref{eq:BIrho} (in the sense of Remark \ref{rem_defboundary} if $\partial \Omega$ is not smooth). 
Fix a domain $\Omega \Subset \R^n$ and a function $\psi \in \cX_u(\Omega)$. 
Then, $u-\psi \in W^{1,\infty}(\Omega)$ and has zero boundary value. By a zero extension of $u-\psi$ outside of $\Omega,$ we see that $u-\psi \in \lip_c(\R^n)$.
We then test $u-\psi$ to \eqref{eq:BIrho} to deduce
\[
\int_{\Omega} \frac{Du \cdot (Du-D\psi)}{\sqrt{1-|Du|^2}} = \int_{\Omega} \rho(u-\psi) . 
\]
Here,  due to the second in \eqref{eq_condi_energy},  the set of points where $u$ is differentiable and 
$|Du|< 1$ has full measure in $\Omega$.  Inequality \eqref{eq_cauchyschwarz} holds therein; thus we see that
\[
\int_{\Omega} \sqrt{1-|D\psi|^2} - \sqrt{1-|Du|^2} \le \int_{\Omega} \rho(u-\psi), 
\]
which implies $I_{\rho}^\Omega(u) \le I_\rho^\Omega(\psi)$.

\smallskip 
\noindent
\textbf{(b) $\Rightarrow$ (a) and $u \in W^{2,q}_\loc(\R^n)$ for each $q \in [2,\infty)$, $|Du|<1$ on $\R^n$.}

\smallskip

For a local minimizer  $u$ of  $I_\rho$, we define a Lorentzian distance $\ell : \R^n \to \R$ from $0$ by  
\[
\ell(x) \equiv \sqrt{|x|^2 - (u(x)-u(0))^2}.  
\]
Since $|Du| \le 1$, we have $\ell \ge 0$ on $\R^n$ and moreover, in view of \eqref{eq_growth_infty}, the Lorentzian balls
\[
L_R \equiv \Set{ x \in \R^n | \ell(x) < R }
\]
are relatively compact in $\R^n$. 
Note that $L_R$ contains the Euclidean ball $B_R(0)$, which for convenience we denote by $B_R$. 
Fix $R>0$, define $\Omega \equiv L_{4R}$ and 
let $\{\rho_j\}$ be a sequence of smooth functions converging to $\rho$ in $L^1(\Omega)$. 
Taking $t_j \in (0,1)$ with $\lim_{j \to \infty}t_j =1$,  
we consider the solution $u_j \in \cY_{t_ju}(\Omega)$ to
	\[
	\left\{ \begin{aligned}
		- \diver \left( \frac{D u_j}{\sqrt{1-|D u_j|^2}} \right) &= \rho_j& & \text{on } \, \Omega, \\
		u_j &= t_j u& & \text{on } \, \partial \Omega. 
	\end{aligned}\right.
	\]
Since $t_j u$ is spacelike on $\partial\Omega$, more precisely $|t_ju(y) - t_ju(x)|< |y-x|$ for each $x,y \in \partial \Omega$ with $x \neq y$, 
by \cite{BS82} there exists a solution $u_j$, which is smooth and $|Du_j|<1$ on $\Omega$. 
We claim that $u_j \to u$ strongly in $C(\overline{\Omega})$ and in $W^{1,q}(\Omega)$ for each $q \in [1,\infty)$. 
This can be shown by adapting an argument in \cite[Proposition 3.11]{bimm} in the following way. 
First, since $\{t_j u\}$ is relatively compact in $C(\partial \Omega)$, by \cite[Proposition 3.5]{bimm} 
$\{u_j\}$ is relatively compact in $C(\ov{\Omega})$. 
In particular, since $|Du_j| \le 1$ for each $j$, $\{u_j\}$ is bounded in $W^{1,q}(\Omega)$ for any $q \in [1,\infty]$. Then there exists $v \in \cY_u(\Omega)$ such that 
$u_j \to v$, up to a subsequence,  in $C(\overline\Omega)$ and weakly in $W^{1,q}(\Omega)$ for any $q \in (1,\infty)$. 
To show that $v \equiv u$, notice that from the second in Remark \ref{rem_useful_identities} we get 
	\begin{equation}\label{ineq:vu_k}
		\begin{aligned}
			\int_{\Omega} \left( 1 - \sqrt{1 - |Dv|^2} \right) 
			= 
			\sum_{k=1}^\infty b_k \norm{D v}_{2k,\Omega}^{2k} 
			&\leq \sum_{k=1}^\infty b_k \liminf_{j \to\infty} 
			\norm{ Du_{j} }_{2k,\Omega}^{2k} 
			\\
			&\leq \lim_{n \to \infty} \liminf_{j \to \infty} \sum_{k=1}^n b_k \norm{ Du_{j} }_{2k,\Omega}^{2k}
			\\
			&\leq \liminf_{j\to\infty} \int_{ \Omega } 
			\left( 1 - \sqrt{1 - |Du_{j}|^2 } \right)  .
		\end{aligned}
	\end{equation}
Moreover, since $u_j \to v$ in $C( \overline{\Omega} )$ and $\rho_j \to \rho$ in $L^1(\Omega)$, 
it holds $\int_\Omega \rho_j u_j  \to  \int_\Omega \rho v $ as $j \to \infty$. Thus we get
	\[
	I_\rho^\Omega(v) \le \liminf_{j \to \infty} I_{\rho_j}^\Omega(u_j).
	\]
On the other hand, $u_j \in \cY_{t_ju}(\Omega)$ minimizes $I_{\rho_j}$; thus $I^\Omega_{\rho_j}(u_j) \le I^\Omega_{\rho_j}(t_ju)$. A simple computation gives
	\[
	I^\Omega_{\rho_j}(t_ju) \to I^\Omega_\rho(u),
	\]
	and therefore $I^\Omega_\rho(v) \le I^\Omega_\rho(u)$. Since the reverse inequality holds by our assumption on $u$ and the minimizer for $I_\rho$ is unique, we conclude that $v = u$ in $\Omega$, as claimed. The argument to prove the strong convergence $u_j \to u$ in $W^{1,q}(\Omega)$ for $q \in [1,\infty)$ then follows verbatim that in \cite[Proposition 3.11]{bimm}, see the second half of p.32 therein.

		Next, we prove the regularity of $u$ on $B_R$. 
Because of \eqref{eq_growth_infty} and a direct comparison between the Euclidean and Lorentzian distances, 
there exists $\hat{R}$ depending on $R, c_1,h$ such that 
\begin{equation}\label{eq_compabolla}
\Omega \equiv L_{4R} \Subset B_{\hat R}.
\end{equation}
Indeed, in our assumption, $\abs{u(0)} \leq c_1 - h(0)$ and 
	\[
		\abs{u(x) - u(0)} \leq c_1 + \abs{u(0)} + \abs{x} - h(x) \leq \bar{c}_1 + \abs{x} - h(x),
	\] 
where $\bar{c} \equiv 2c_1 -h(0)$. If $x \in L_{4R}$, then by expanding
\[
|x|^2 \le 16R^2 + (u(x)-u(0))^2 \le 16 R^2 + ( \bar{c}_1 + |x| - h(x) )^2
\]
and rearranging we get
\[
2|x| (h(x) - \bar{c}_1) \leq 16R^2 + (\bar{c}_1- h(x))^2
\]
from which the bound $|x| \le \hat R$ follows and recalling our growth assumptions on $h$. By a simple argument by contradiction, we deduce from $u_j \to u$ in $C(\overline\Omega)$ that for all sufficiently large $j$, 
	\[
		 L^{\rho_j}_{3R} \equiv \Set{ x \in \Omega | \ell^{p_j} (x) \equiv \sqrt{\abs{x^2} - \qty( u_j(x) - u_j(0) )^2 } < 3R } \Subset \Omega. 
	\]
We apply \cite[Lemma 2.1]{BS82} and the reasoning in the proof of Theorem 4.1 therein, 
to deduce that there exists a constant $\delta > 0$ depending on the $L^\infty$ bounds for $u_j$ and $\rho_j$ on $L_{3R}^{\rho_j}$, which are thus uniformly bounded in terms of 
\[
	\hat R, \quad c_1, \quad h \quad \text{and} \quad \|\rho\|_{L^\infty(B_{\hat R})}, 
\]
such that
\[
|Du_j| < 1-\delta \qquad \text{on } \, L_{2R}^{\rho_j}.
\] 
Remark also that $\{u_j\}$ is bounded in $W^{2,2}(L_{2R})$. 
Since $B_{2R} \subset L_{2R}^{\rho_j}$, we get local uniform estimates for $|Du_j|$ and for the $L^2$ norm of $|D^2u_j|$ on $B_{2R}$. Thus, for any $\eta \in \Lip_c(B_{2R
})$, by letting $j \to \infty$ in 
	\[
		\int_{\Rn} \frac{Du_j \cdot D\eta}{\sqrt{1 - \abs{Du_j}^2}} = \int_{\Rn} \rho \eta,
	\]
the dominated convergence theorem implies that \eqref{soleq} holds.

	Finally, for any $q \in [2,\infty)$ we show the existence of $C_q = C_q(\hat R,c,q)$ such that $\|u\|_{W^{2,q}(B_R)} \le C_q$. To this goal, we first argue as in \cite[Step 5 in Section 5.2]{bimm} to obtain $Du \in C^{\alpha_R} (B_R(0))$ 
for some $\alpha_R \in (0,1)$. Next, from $u \in W^{2,2}_{\rm loc} (\Rn)$ and writing 
	\[
		a_{ij} (x) \equiv \frac{\delta_{ij}}{\sqrt{1 - \abs{Du(x)}^2}} + \frac{ D_iu (x) D_j u (x) }{\qty( 1 - \abs{Du(x)}^2 )^{3/2}} \in C(\Rn), 
	\]
we infer that $u$ is a strong solution to 
	\[
		- \sum_{i,j=1}^n a_{ij} (x) D_{ij} u(x) = \rho(x) \quad \text{in} \ \Rn. 
	\]
Since $\|\rho\|_{L^\infty(B_{2R})} \le \|\rho\|_{L^\infty(B_{\hat R})} \le c$, elliptic regularity yields the desired $W^{2,q}$ estimate in $B_R$, see for instance \cite[Chapter 9]{GiTr01}. This completes the proof of (i). 

\smallskip

(ii) Suppose $u \in L^\infty(\Rn)$ and $\rho \in L^\infty(\Rn)$ satisfy $\|\rho\|_\infty + \|u\|_\infty \le c$. By (i), $u \in C^1(\Rn)$ and $\abs{Du (x)} < 1$ in $\Rn$. 
We prove the assertion by contradiction and suppose that there exist $\{\rho_i\}_i , \{u_i\}_i \in L^\infty(\Rn)$, with $u_i$ a solution with source $\rho_i$, and points $\{x_i\}_i \subset \Rn$ such that
\[
\|\rho_i\|_\infty + \|u_i\|_\infty \le c, \qquad \abs{Du(x_i)} \to 1 \ \ \text{ as } \, i \to \infty.
\]
Define 
	\[
		\bar{u}_i(x) \equiv u_i(x+x_i), \qquad \bar{\rho}_i(x) \equiv \rho_i(x+x_i).
	\]
Then, $\bar{u}_i$ is a (strong) solution with source $\bar{\rho}_i$ on $\Rn$, hence a local minimizer of $I_{\bar{\rho}_i}$. Since $\bar u_i$ is $1$-Lipschitz, 
we may also assume that $\bar u_i \to \bar u_\infty$ in $C_\loc (\Rn)$, and $\bar u_\infty$ is also $1$-Lipschitz with $\|\bar{u}\|_\infty \le c$. The inequality $\norm{\bar{\rho}_i} + \norm{\bar{u}_i}_\infty \le c$ and \cite[Theorem 3.2]{BS82} imply that $\bar u_i$ has no light segments in the following quantitative sense: for each $r>0$, there exists $R=R(r) > 0$ such that for every $i \geq 1$, 
	\[
		L^{\bar{\rho}_i}_{r} \equiv \Set{ x \in \Rn | \ell^{\bar{\rho}_i} (x) \equiv \sqrt{ \abs{x}^2 - \qty( \bar{u}_i(x) - \bar{u}_i(0) )^2 } < r} \subset B_R. 
	\]
According to the monotonicity formula in \cite[Lemma 2.1]{BS82}, we may find $\alpha \in (0,1/n)$ and $C = C(n) >0$ such that 
	\[
		C \exp \qty( r^2 c^2 + 1  ) \qty( 1 - \abs{D\bar{u}_i(0)} )^{\frac{\alpha}{2}} 
		\geq  r^{-n} \int_{L_r^{\bar{\rho}_i}} \qty( 1 - \abs{D\bar{u}_i}^2 )^{\alpha + 1} + r^{2-n} 
		\int_{L^{\bar{\rho}_i}_r} \abs{D^2\bar{u}_i}^2 .
	\]
From $\abs{D\bar{u}_i (0) } \to 1$ and the fact that $r>0$ can be arbitrary, it follows that $\abs{D \bar{u}_\infty} = 1$ and $D^2 \bar{u}_\infty = 0$ a.e. in $\Rn$. 
Thus, $\bar{u}_\infty(x) = a \cdot x + b$ for some $a \in \Rn$, $b \in \R$ with $\abs{a} = 1$, 
however, this contradicts $\norm{\bar {u}_\infty}_\infty \leq c$ and concludes the proof.
\end{proof}

\end{document}